\documentclass[10pt]{amsart}

\usepackage{times,amsfonts,amsmath,amstext,amsbsy,amssymb,
  amsopn,amsthm,upref,eucal, amscd}
\usepackage[T1]{fontenc}

\usepackage[pdftex]{graphicx}

\newtheorem{theorem}{Theorem}[section]
\newtheorem{lemma}[theorem]{Lemma}
\newtheorem{corollary}[theorem]{Corollary}
\newtheorem{proposition}[theorem]{Proposition}

\numberwithin{equation}{section}

\theoremstyle{definition}
\newtheorem{definition}[theorem]{Definition}

\newtheorem{remark}[theorem]{Remark}

\def\Cset{\mathbb{C}}

 \def\Rset{\mathbb{R}}
 \def\Zset{\mathbb{Z}}
 \def\Tset{\mathbb{T}}
 
\def\Iset{\mathbb{I}}
\def\Pset{\mathbb{P}}
\def\va{ \varepsilon}
\def\wt{\widetilde}
\def\wh{\widehat}
\def\leq{\leqslant }
\def\geq{\geqslant}
\def\RA{\mathbb{R}^{\mathcal A}}
\def\A{\mathcal{A}}
\def\iat{I_{\alpha}^t}
\def\iatl{I_{\alpha}^{t,(\ell)}}
\def\iatn{I_{\alpha}^{t,(n)}}
\def\iabn{I_{\alpha}^{b,(n)}}
\def\Log{{\rm Log}}
\def\iab{I_{\alpha}^b}

\def\AA{\mathcal{A}^{(2)}}
\def\RS{\mathbb{R}^{\Sigma}}
\def\RSO{\mathbb{R}_0^{\Sigma}}
\def\RO{R_{\omega}}
\def\cD{{\mathcal D}}
\def\whT{\widehat T}
\begin{document}

\title{Linearization of generalized interval exchange maps}

\author[S. Marmi]{Stefano Marmi}
\address{Scuola Normale
Superiore, Piazza dei Cavalieri 7, 56126 Pisa, Italy}
\email{ s.marmi(at)sns.it}
\author[P. Moussa]{Pierre Moussa}
\address{Institut de Physique Th\'eorique, CEA/Saclay,
91191 Gif-Sur-Yvette, France}
\email{ pierre.moussa(at)cea.fr}
\author[J.-C. Yoccoz]{Jean-Christophe Yoccoz}
\address{Coll\`ege de France, 3, Rue d'Ulm, 75005
Paris, France}
\email{ jean-c.yoccoz(at)college-de-france.fr }

\date{November 22 , 2011}


\begin{abstract}
A standard interval exchange map is a one-to-one map of the interval which is locally a translation except at finitely many
singularities. We define for such maps, in terms of the Rauzy-Veech continuous fraction algorithm, a diophantine
arithmetical condition called restricted Roth type which is almost surely satisfied in parameter space. Let $T_0$ be a
standard interval exchange map of restricted Roth type, and let $r$ be an integer $\geq 2$. We prove that, amongst $C^{r+3}$
deformations of $T_0$ which are $C^{r+3}$ tangent to $T_0$ at the singularities, those which are conjugated to $T_0$ by a
$C^r$ diffeomorphism close to the identity form a $C^1$ submanifold of codimension $(g-1)(2r+1) +s$. Here, $g$ is the genus
and $s$ is the number of marked points of the translation surface obtained by suspension of $T_0$. Both $g$ and $s$ can be
computed from the combinatorics of $T_0$.
\vskip 1. truecm
\centerline{\it To the memory of G\'erard Rauzy}

\end{abstract}

\maketitle

\tableofcontents

\section{Introduction}

\subsection{Presentation of the main result}

Many problems of stability in the theory of dynamical systems face the
difficulty of small divisors. The most famous example is probably given by
Kolmogorov-Arnold-Moser theory on the persistence of quasi-periodic solutions
of Hamilton's equations for quasi-integrable Hamiltonian systems (both
finite and infinite-dimensional, like nonlinear wave equations). This is a very
natural situation with many applications to physics and astronomy. What all
these different problems have in common is roughly speaking what follows:
one can associate some ``frequencies'' to the orbits under investigation and
some arithmetical condition is needed to prove their existence and stability.

The simplest example of quasiperiodic dynamics is given by irrational rotations of the circle.
Poincar\'e asked under which condition
a given homeomorphism of the circle is equivalent (in some sense, e.g. topologically or smoothly)
to some rotation and proved that any orientation-preserving homeomorphism of the circle with no periodic orbit is
semi-conjugate to an irrational rotation.
Denjoy proved that, when the rotation
number is irrational, adding regularity to a given homeomorphism $f$ (namely requiring $f$ to be piecewise $C^1$ with $Df$
of bounded variation) is enough to guarantee topological conjugacy to a rotation. The step to higher order differentiability
for the conjugacy $h$ requires
new techniques and additional hypotheses on the rotation number: a small divisor problem must be overcome and this
was first achieved (in the circle case) by Arnold in [A]: he proved that if the rotation number verifies a diophantine
condition and if the analytic diffeomorphism $f$ is close enough to a rotation,
then the conjugation is analytic. At the same time examples of analytic diffeomorphisms,
with irrational rotation number, for which the conjugation is not even absolutely continuous were given.
Later Herman ([He1]) proved a global result: there
exists a full Lebesgue measure set of  rotation
numbers for which a $C^\infty$ (resp. $C^\omega$) diffeomorphism is $C^\infty$ (resp. $C^\omega$) conjugated to
a rotation. In the finitely differentiable case one can prove a similar result
but the conjugacy is less regular than the diffeomorphism: this phenomenon of loss of
differentiability is typical of small divisors problems.

The suspension of  circle rotations produces linear flows on the two-dimensional torus. When analyzing the recurrence of
rotations or the suspended flows, the modular group $\hbox{GL}\, (2,\Zset)$ is of fundamental importance, providing the
renormalization scheme associated to the continued fraction of the rotation number.

A generalization of the linear flows on the two-dimensional torus is obtained by considering linear flows on
translation surfaces of higher genus (see e.g. [Zo1] for a nice introduction to the subject).
By a Poincar\'e section their dynamics can be reduced to (standard)
interval exchange maps (i.e.m.\ ), which generalize rotations of the circle.

A (standard) i.e.m.\ $T$ on an interval $I$ (of finite length) is
a one-to-one map which is locally a translation except at a finite number of discontinuities.
Thus $T$ is orientation-preserving and preserves Lebesgue measure. By asking only that
$T$ is locally an orientation-preserving homeomorphism one obtains the definition of a generalized i.e.m.
Let $d$ be the number of intervals of continuity of $T$.
When $d=2$, by identifying the endpoints of $I$, standard
i.e.m.'s correspond to rotations of the circle and generalized i.e.m.'s to homeomorphisms of the circle.
Standard i.e.m.\ can be suspended following the construction of Veech [Ve1] to give rise to translation surfaces.

Typical standard i.e.m.'s are minimal ([Kea1]) but note that
ergodic properties of minimal standard i.e.m.'s can differ substantially
from those of circle rotations: they need not be
ergodic ([Kea2], [KeyNew]) but almost every standard i.e.m.\ (both in the topological sense [RK]
and in the measure-theoretical sense [Ma, Ve2]) is ergodic. Moreover  the typical non rotational
standard i.e.m.\ is weakly mixing [AF].

Rauzy and Veech have defined an  algorithm
that generalizes the classical continued fraction
algorithm (corresponding to the choice $d=2$) and associates to an i.e.m.\ another i.e.m.\ which is its first
return map to an appropriate subinterval [Ra, Ve2]. Both
 Rauzy--Veech "continued fraction" algorithm and its accelerated
version due to Zorich [Zo2] are ergodic w.r.t.\ an absolutely continous
invariant measure in the space of normalized standard i.e.m.'s. However in the case of
the Rauzy--Veech algorithm the measure has infinite mass
whereas the invariant measure for the Zorich algorithm has finite
mass. The ergodic properties of these renormalisation dynamics in parameter space have been studied in detail ([Ve3],[ Ve4], [Zo3], [Zo4], [AvGoYo], [B], [AB],  [Y4]).

The possible combinatorial data for an i.e.m.\ (standard or generalized) are the vertices of {\it Rauzy diagrams};
the arrows of these diagrams correspond to the possible transitions under the Rauzy-Veech algorithm.

The Rauzy-Veech algorithm, which makes sense for generalized i.e.m.'s , stops if and only if the i.e.m.\ has a {\it
connection},
i.e.\ a finite orbit which starts and ends at a discontinuity. When
the i.e.m.\ has no connection, the algorithm associates to it an infinite path in a Rauzy diagram that can be viewed
as a ``rotation number''.

One can characterize the infinite paths associated to standard i.e.m.\ with no connections ($\infty$--complete paths, see
subsection
2.3). One says that a generalized i.e.m.\ $T$ is {\it irrational} if its associated path is $\infty$--complete; then $T$ is
semi--conjugated to any standard i.e.m.\ with the same rotation number [Y2].

This generalization of Poincar\'e's theorem suggests the following
very natural question: {\it what part of the theory of circle homeomorphisms and
diffeomorphisms generalizes to interval exchange maps ?}

All translation surfaces obtained by suspension from standard i.e.m.\ with a given Rauzy diagram have the same genus $g$,
and the same number $s$ of marked points; these numbers are related to the number $d$ of intervals of continuity
by the formula $d=2g+s-1$.

Regarding Denjoy's theorem, partial results ([CG], [BHM], [MMY2]) go in the negative direction, suggesting that
topological conjugacy to a standard i.e.m.\ has positive codimension in genus $g\ge 2$.

A first step in the direction of extending small divisor results
beyond the torus case was achieved by Forni's important paper ([For1], see also [For3])
on the cohomological equation associated to linear flows on
surfaces of higher genus.
In [MMY1], we considered the cohomological equation $\psi\circ T_0-\psi =\varphi$ for a standard i.e.m.\ $T_0$.
We found explicitly in terms of the Rauzy-Veech algorithm a full measure class of standard i.e.m.\ ( which we called Roth
type i.e.m.\ )
for which the cohomological equation has bounded solution provided that
the datum $\varphi $ belongs to a finite codimension subspace of the
space of functions having on each continuity interval a continuous derivative with bounded
variation. The improved loss of regularity (w.r.t.\ [For1]) will be decisive for the proof of our main result.

The cohomological equation is the linearization of the conjugacy equation $T\circ h=h\circ T_0$ for a generalized i.e.m.\
$T$
close to the standard i.e.m.\ $T_0$.

We say that a generalized i.e.m.\ $T$ is a {\it simple deformation of class} $C^r$ of a standard i.e.m.\ $T_0$ if
\begin{itemize}
\item $T$ and $T_0$ have the same discontinuities;
\item $T$ and $T_0$ coincide in the neighborhood of the endpoints of $I$ and of each discontinuity;
\item $T$ is a $C^r$ diffeomorphism on each continuity interval onto its image.
\end{itemize}

Our main result is a local conjugacy theorem which is stated in full generality in Section 5. For simple deformations
the result can be summarized as follows:

\vskip .3 truecm\noindent {\bf Theorem.}\hspace{5mm} {\it For almost all standard i.e.m.\ $T_0$  and for any integer $r\ge
2$, amongst
the $C^{r+3}$ simple deformations of $T_0$, those which are $C^r$-conjugate to $T_0$ by a diffeomorphism $C^r$ close to the
identity form a $C^1$ submanifold of
codimension $d^*=(g-1)(2r+1)+s$. }

\vskip .3 truecm\noindent

The standard i.e.m.\ $T_0$ considered in the theorem are the Roth type i.e.m.\ for which the Lyapunov exponents of the
KZ-cocycle (see subsection 2.6) are non zero (we call this {\it restricted Roth type}). They still form a full measure set
by Forni's theorem [For2].

The tangent space at $T_0$ to the $C^1$ submanifold of $C^{r+3}$ simple deformations which are $C^r$-conjugate to $T_0$
is formed of  $C^{r+3}$ functions $\varphi$ which vanish in a neighborhood of the singularities of $T_0$ and can be written as 
$$
\varphi = \psi\circ T_0 - \psi\, , 
$$
where $\psi$ is a  $C^r$ function vanishing at the singularities of $T_0$.

To extend this result to generalized i.e.m.'s $T$ of class $C^r$ which are not simple deformations of a standard i.e.m.\
$T_0$, there are gluing problems of the derivatives of $T$ at the discontinuities. Indeed there is  a conjugacy invariant
which is an obstruction to linearization (see Section 4).

An earlier result is presented in an unpublished manuscript of De La Llave and Gutierrez [DG], which was recently
communicated to us by P.\ Hubert. They consider standard i.e.m.\ with periodic paths for the Rauzy-Veech algorithm (for
$d=2$, this corresponds to rotations by a quadratic irrational). They prove that, amongst piecewise analytic generalized
i.e.m.\ , the bi-Lipschitz conjugacy class of such a standard i.e.m.\ contains a submanifold of finite codimension. They also prove that 
bi-Lipschitz conjugacy implies $C^1$-conjugacy.

The proof of our theorem is based on an adaptation of Herman's Schwarzian derivative trick. In [He2] Herman gave simple
proofs of local conjugacy theorems for diffeomorphisms $f$ of the circle: let $\omega$ denote the rotation number, assumed
to satisfy a diophantine condition $|\omega -p/q|\ge \gamma q^{-2-\tau}$ for some $\gamma>0$, $\tau < 1$, and
let $R_\omega$ be the corresponding rotation of the circle. Taking Schwarzian derivatives, the conjugacy equation $f\circ
h=h\circ R_\omega$ becomes $(Sh)\circ R_\omega -Sh=((Sf)\circ h)(Dh)^2$,
a linear difference equation in the Schwarzian derivative $Sh$ of the conjugacy (but the r.h.s.\ depends also on $h$).
Given a diffeomorphism $h$, one computes the r.h.s.\ $((Sf)\circ h)(Dh)^2$, solves the equation
$\psi\circ R_\omega - \psi = ((Sf)\circ h)(Dh)^2$ and then finds a diffeomorphism $\tilde{h}=\Phi (h)$ as smooth as $h$ with
$S\tilde{h}=\psi$. Herman now uses the  Schauder-Tychonov theorem to find a fixed point of $\Phi$ and thus the required
conjugacy.  He was aware of the possibility of using the contraction
principle (at the cost of one more derivative for $f$) as we do in our proof. Herman's method is presented in more detail in
Appendix B.1.

In Section 8 of the paper, we explain how to adapt our result to the setting of perturbations of linear flows on translation surfaces. 
Indeed we prove the following corollary of the main theorem (we refer the reader to Section 8 and to Appendix C for 
the definition of Roth-type translation surface and of simple deformation of a vertical vectorfield): 

\vskip .3 truecm\noindent {\bf Corollary.}\hspace{5mm} {\it Given a translation surface of restricted Roth type and  any integer $r\ge
2$, amongst the $C^{r+3}$ simple deformations of the vertical vectorfield, those which are $C^r$-equivalent to it by a diffeomorphism $C^r$ close to the
identity form a $C^1$ submanifold of codimension $d^*=(g-1)(2r+1)+s$. }

\vskip .3 truecm\noindent

\subsection{Open problems}

\smallskip\noindent

{\bf 1.}\hspace{3mm}{\it Prove the  theorem for $r=1$: for almost all standard i.e.m.\ $T_0$, amongst
the $C^{4}$ simple deformations of $T_0$, those which are $C^1$-conjugate to $T_0$ by a diffeomorphism $C^1$ close to the
identity form a $C^1$ submanifold of
codimension $d^*=3g-3+s$. }\\

A  {\it rationale} for this conjecture comes from the following argument.
Note that $d^*$ is equal here to $(d-1) + (g-1)$. The integer $d-1$ is the dimension of the space of standard i.e.m.\ up to
affine conjugacy. In order to have a $C^1$-conjugacy between a generalized i.e.m.\ $T$ and a standard i.e.m.\ $T_0$ with the
same rotation number, a necessary condition is that the Birkhoff sums of $\Log DT$ (equal to $\Log DT^n$) are bounded. The
integral of $\Log DT$ w.r.t.\ the unique invariant measure is automatically zero, taking care of the largest exponent of the
KZ-cocycle; killing the components w.r.t.\ the remaining $g-1$ positive exponents leads to the expected value of $d^*$.

On the other hand, when the derivatives of the iterates $ DT^n$ are allowed to grow exponentially fast, one could expect to
have wandering intervals (see [MMY2]).

This suggests the existence of a dichotomy between being $C^1$-conjugated to a standard i.e.m.\ and having wandering
intervals:
one can therefore ask whether the following is true:\\

\smallskip\noindent
{\bf 2.}\hspace{3mm}{\it  For almost all standard $T_0$,  any generalized
i.e.m.\ $T$ of class $C^4$ which is a simple deformation of $T_0$ and is
topologically conjugated to $T_0$ is also $C^1$-conjugated to $T_0$.}\\

The two conjectures above can be formulated in a slightly more general setting (not restricted to simple deformations) using
the conjugacy invariant introduced in Section 4.\\

\smallskip\noindent
{\bf 3.}\hspace{3mm} The local $C^r$ conjugacy class of a standard i.e.m.\ $T_0$ (of restricted Roth type) exhibited by our
theorem can be considered as a {\it local stable manifold}
for the renormalization operator ${\mathcal R}$ defined by  the Rauzy-Veech induction (with rescaling) on
generalized i.e.m.'s in a suitable functional space. By the standard techniques this local stable manifold extends to a
global stable manifold
$$W^s(T_0)=\cup_{n\ge 0} {\mathcal R}^{-n}(W^s_{loc}({\mathcal R}^n T_0)\,$$
which is the full $C^r$ conjugacy class of $T_0$.

\medskip
{\it Is this stable manifold ``properly
embedded'' in  parameter space?}

\medskip

More precisely, given a sequence of diffeomorphisms $h_n$ in ${\rm Diff}^r(\overline{I})$ such that $h_n\rightarrow\infty$,

\medskip

 {\it Is it possible that $h_n\circ T_0\circ h_n^{-1}\rightarrow T_0$ in the $C^{r+3}$ topology? Is it possible that
 $h_n\circ T_0\circ h_n^{-1}$ stays bounded in the $C^{r+3}$ topology?}

\medskip

In the case $d=2$
 , the answer to both questions is no. For the second question, this is a consequence of Herman's global conjugacy theorem
 for circle diffeomorphisms.\\

\smallskip\noindent
{\bf 4.}\hspace{3mm} {\it Describe the set of generalized $C^r$ interval exchange maps which are semi-conjugate to a given
standard
i.e.m.\ $T_0$ (with no connections).} \\

In the circle case, for a  diophantine rotation number,  one has a $C^\infty$ submanifold of
codimension $1$. In the Liouville case one  has still a topological manifold of codimension $1$ which is transverse to
all $1$-parameter strictly increasing families. One can therefore dare to ask:

\begin{enumerate}
\item {\it Is the above set a topological submanifold of codimension $d-1$?}
\item {\it if the answer is positive, does there exist a (smooth) field of "transversal" subspaces of dimension $d-1$?}
\end{enumerate}

The questions make sense for any $T_0$, but the answer could depend on the diophantine properties of $T_0$.

\smallskip\noindent
{\bf 5.}\hspace{3mm} {\it In a generic smooth family of generalized i.e.m.'s, is the rotation number irrational with
positive probability?}

In the circle case the answer is affirmative, thanks to Herman's theorem. This is not very likely in higher genus.

\smallskip\noindent
{\bf 6.}\hspace{3mm}{\it Let $r\ge 1$. Describe exactly (in terms of the Rauzy-Veech renormalization algorithm)  the set of
rotation
numbers such that the $C^r$ conjugacy class of $T_0$ has finite codimension in the space of $C^\infty$ generalized
i.e.m.'s.
Does this set depends on $r$?}

In the circle case, this set is (for any $r\geq 1$) the set of diophantine rotation numbers ([Y3],[He1]). In higher genus,
our theorem (in the stronger form stated in Section 5) guarantees that this set contains the restricted Roth type rotation
numbers and therefore has full measure.
It looks like that our methods extend to prove that the unrestricted Roth type rotation numbers also belong to this set (but the 
codimension of the  $C^r$ conjugacy class of $T_0$  is different).
Of course the codimension of the $C^r$ conjugacy class will depend on $r$ but the
point here is that we only require the codimension to be finite.

Note that the answer is not known even at the level of the cohomological equation!

A related question is the optimal loss of differentiability, for instance for restricted Roth type rotation numbers. A careful
reading of the proof (and of Appendix A) will convince the reader that we may consider  $C^{r+2+\tau}$ (for any $\tau >0$) simple 
deformations of  $T_0$ instead of  $C^{r+3}$ simple deformations and still get the same conclusion. On the other hand, the cohomological 
equation suggests that some form of the result could be true for  $C^{r+1+\tau}$  simple 
deformations of  $T_0$ (for any $\tau >0$). This is certainly true in genus $1$. 
This is however beyond the reach of our method.

\subsection{Summary of the paper}
In the next section we introduce standard and generalized
interval exchange maps. We recall the definition and the main properties of the Rauzy-Veech continued fraction
algorithm, and explain how it allows to define in a very natural way a "rotation number" for certain generalized i.e.m.'s. The
algorithm generates a dynamical system in parameter space, equipped with a very important cocycle, the Kontsevich-Zorich
cocycle.  The notations and the
presentation  of this section follow closely
the expository paper [Y1] (see also [Y2],[Y4]).

Section 3 is devoted to the study of the cohomological equation. We introduce a boundary operator on the  space of
piecewise-continuous functions
which vanishes on coboundaries and take care of the neutral component of the KZ-cocycle. We  review
the results of [MMY1] (Theorem 3.10), recalling in particular the definition of Roth type i.e.m. We actually improve on the
results of [MMY1] by showing that under the same assumptions one can obtain a continuous (instead of bounded)
solution. We also reformulate the results in higher smoothness using the boundary operator.

In Section 4 we introduce, for any integer $r\geq 1$,  an  invariant for $C^r$ conjugacy with values in the conjugacy
classes of the group $J^r$ of $r$-jets of orientation-preserving diffeomorphisms of $(\Rset,0)$. We show that it is also
preserved by the renormalization operator defined by the Rauzy-Veech algorithm. We  explain the
relation of this conjugacy invariant
with the boundary operator.

Section 5 contains the precise formulation of our main result (Theorem 5.1): $C^1$ parameter families of generalized i.e.m.'s
of class $C^{r+3}$
 through a standard i.e.m.\ $T_0$ of restricted Roth type are considered. It is assumed that the $C^{r+3}$-conjugacy invariant
 vanishes and an appropriate transversality hypothesis (related to the cohomological equation) is satisfied. The theorem
 then states that the local $C^r$-conjugacy class of $T_0$ intersects the family along a submanifold whose tangent space at
 $T_0$ is given by the cohomological equation. We also show how the hypothesis on the conjugacy invariant allows to reduce
 the proofs to the case of simple families.

Section 6 contains the proof of Theorem 5.1 when $r\geq 3$; following Herman, we use  Schwartzian derivatives to construct a
map whose fixed point is a candidate for the conjugating map. In the circle case, this fixed point is always the conjugating
map. In the present case, some extra equations representing gluing conditions have to be satisfied; these equations define
the local conjugacy class in parameter space.

Section 7 deals with the case $r=2$ of Theorem 5.1. Indeed, a $C^2$-diffeomorphism does not have in general a Schwartzian
derivative. We need a little improvement of Herman's Schwarzian derivative trick. We show how one can
effectively use the primitive of the Schwarzian derivative to construct a contracting map whose fixed point will turn out to
be the conjugacy, under appropriate gluing conditions.

In Section 8 we explain how to adapt our result to the simple deformations of linear flows on translation surfaces. After a brief
introduction to translation surfaces we study the action of the boundary operator at the level of the surface and we prove that 
the conjugacy invariant is trivial for simple deformations of the vertical vectorfield. We then introduce {\it restricted Roth type 
translation surfaces} and we prove the Corollary stated at the end of subsection 1.1.

In Appendix A  we show that the main result of [MMY1] (in the improved version of Theorem 3.10)  is
also valid with data whose
first derivatives are H\"older continuous instead of having bounded variation.

 Appendix B is devoted to the case of
circle diffeomorphisms. In subsection B.1, we deal with $C^r$-conjugacy, $r\geq 3$.  Herman's original result (through
Schauder-Tychonov fixed point theorem) gives a stronger conclusion in this setting; however, the simple variant based on the
fixed point theorem for contracting maps is a better preparation for the more difficult case of Section 6. In the same way,
subsection B.2 introduces the main idea of Section 7 in a simpler setting.

Finally Appendix C is devoted to the study of Roth-type translation surfaces. Proposition C.1 gives several equivalent formulations of 
condition (a) in the definition of a Roth-type i.e.m. (see subsection 3.3). This is then used in order to prove that the i.e.m.'s obtained as first return maps on 
an open bounded segment (in good position) of the vertical flow on a (restricted) Roth-type translation surface are of (restricted) Roth-type.

\vskip 1. truecm \noindent {\bf Acknowledgements}  This research
has been supported by the  following institutions: the Coll\`ege de France, the Scuola
Normale Superiore, the French ANR (grants 0863 Petits diviseurs et r\'esonances en g\'eomtrie, EDP et dynamique and 0864
Dynamique dans l'espace de Teichm\"uller) and the Italian MURST (PRIN grant 2007B3RBEY Dynamical Systems and applications).
We are also grateful to
the two former
institutions, to the Centro di Ricerca Matematica ``Ennio De
Giorgi'' in Pisa  and to the Max Planck Institute f\"ur Mathematik in Bonn for hospitality.
We are grateful to the referee for his suggestions and remarks which led to a considerable improvement of our paper. 

\section{Background}

\subsection{Interval exchange maps}

Let $I$ be an open bounded interval. A generalized interval exchange map (g.i.e.m.\ ) $T$ on $I$ is defined by the following
data. Let $\A$ be an alphabet with $d \geq 2$ symbols. Consider two partitions mod.$0$ of $I$  into $d$ open subintervals
indexed by $\A$ (the {\it top} and {\it bottom} partitions):
$$ I = \sqcup \iat = \sqcup\iab \,.$$
The map $T$ is defined on $\sqcup \iat$ and its restriction to each $\iat$ is an orientation-preserving homeomorphism onto
the corresponding $\iab$.

The g.i.e.m.\ $T$ is {\it standard} if $|\iat| =|\iab|$ for each $\alpha \in \A$ and the restriction of $T$ to each $\iat$ is
a translation.

Let $r$ be an integer $\geq 1$ or $\infty$. The g.i.e.m.\ $T$ is {\it of class $C^r$} if the restriction of $T$ to each $\iat$
extends to a $C^r$-diffeomorphism from the closure of $\iat$ onto the closure of $\iab$. For finite $r$, it is easy to see
that the g.i.e.m.'s  with  fixed $\A$ form a Banach manifold.

The points $u^t_1<\cdots<u^t_{d-1}$ separating the $ \iat$ are called the {\it singularities } of $T$. The points
$u^b_1<\cdots<u^b_{d-1}$ separating the $ \iab$ are called the singularities  of $T^{-1}$. We also write $I=(u_0,u_d)$,
$u^t_0 =u^b_0 =u_0$, $u^t_d=u^b_d=u_d$.

The {\it combinatorial data} of $T$ is the pair $\pi = (\pi_t,\pi_b)$ of bijections from $\A$ onto $\{1,\ldots,d\}$ such
that
$$\iat = (u^t_{\pi_t(\alpha)-1},u^t_{\pi_t(\alpha)}),\quad \iab = (u^b_{\pi_b(\alpha)-1},u^b_{\pi_b(\alpha)})$$
for each $\alpha \in \A$.

We always assume that the combinatorial data are {\it irreducible}: for $1 \leq k<d$, we have
$$ \pi_t^{-1}(\{1,\ldots ,k\})\not=\pi_b^{-1}(\{1,\ldots ,k\})\; . $$

\subsection{The elementary step of the Rauzy--Veech algorithm}

A {\it connection} is a triple $(u_i^t,u_j^b,m)$, where $m$ is a nonnegative integer, such that
$$ T^m(u_j^b)=u_i^t\; . $$

Keane has proved [Kea1] that a {\bf standard} i.e.m.\ with no connection is minimal.

Let $T$ be a g.i.e.m.\ with no connection. We have then $u^t_{d-1} \ne u^b_{d-1}$. Set $\widehat u_d :=
\max(u^t_{d-1},u^b_{d-1})$, $\widehat I := (u_0,\widehat u_d)$, and denote by $\widehat T$ the first return map of $T$ in
$\widehat I$. The return time is $1$ or $2$.

One checks that $\widehat T$ is a g.i.e.m.\ on $\widehat I$ whose combinatorial data $\widehat \pi$ are canonically labeled by
the same alphabet $\A$ than $\pi$ (cf.[MMY1] p.829). Moreover $\widehat T$ has no connection; this allows to iterate the
algorithm.

We say that $\widehat T$ is deduced from $T$ by an elementary step of the Rauzy--Veech algorithm. We say that the step is of
{\it top} (resp. {\it bottom}) {\it type} if $u^t_{d-1} < u^b_{d-1}$ (resp. $ u^t_{d-1} > u^b_{d-1}$). One then writes
$\widehat \pi = R_t(\pi)$ (resp. $\widehat \pi = R_b(\pi)$).

\subsection{Rauzy diagrams}

A {\it Rauzy class} on the alphabet $\A$ is a nonempty set of irreducible combinatorial data
which is invariant under $R_t,R_b$ and minimal with respect to this property.
A {\it Rauzy diagram} is a graph whose vertices are the elements of a Rauzy class and whose arrows
connect a vertex $\pi$ to its images $R_t(\pi)$ and $R_b(\pi)$. Each vertex is therefore the origin of two arrows. As
$R_t,R_b$ are invertible, each vertex is also the endpoint of two arrows.

An arrow connecting $\pi$ to $R_t(\pi)$ (respectively $R_b(\pi)$) is said to be of {\it top type} (resp.\ {\it bottom
type}). The {\it winner} of an arrow of top (resp.\ bottom) type starting at $\pi=(\pi_t,\pi_b)$ with $\pi_t(\alpha_t)=
\pi_b(\alpha_b)=d$ is the letter $\alpha_t$ (resp.\ $\alpha_b$) while the {\it loser} is $\alpha_b$ (resp.\ $\alpha_t$).

A path $\gamma$ in a Rauzy diagram is {\it complete} if each letter in $\A$ is the winner of at least one arrow in $\gamma$;
it is $k$--{\it complete} if $\gamma$ is the concatenation of $k$ complete paths. An infinite path is $\infty$--{\it
complete} if it is the concatenation of infinitely many complete paths.

\subsection{The Rauzy-Veech algorithm}

Let $T=T^{(0)}$ be an i.e.m.\ with no connection. We
denote by $\A$ the alphabet for the combinatorial data $\pi^{(0)}$ of $T^{(0)}$ and by ${\mathcal D}$ the Rauzy diagram on
$\A$ having $\pi^{(0)}$ as a vertex.

The i.e.m.\ $T^{(1)}$, with combinatorial data $\pi^{(1)}$, deduced from $T^{(0)}$ by the elementary step of the
Rauzy--Veech algorithm has also no connection. It is therefore possible to iterate this elementary step indefinitely
and get a sequence $T^{(n)}$ of i.e.m.\, with combinatorial data $\pi^{(n)}$, acting on a decreasing sequence $I^{(n)}$ of
intervals and a sequence $\gamma (n,n+1)$ of arrows in ${\mathcal D}$ from $\pi^{(n)}$ to $\pi^{(n+1)}$ associated to the
successive steps of the algorithm. For $m<n$, we also
write $\gamma (m,n)$ for the path from $\pi^{(m)}$ to $\pi^{(n)}$ made of the concatenation of the $\gamma (l,l+1)$, $m\leq
l<n$.

We write $\gamma(T)$ for the infinite path starting from $\pi^{(0)}$ formed by the $\gamma (n,n+1)$, $n\geq 0$.
If $T$ is a standard i.e.m.\ with no connection, then $\gamma(T)$ is $\infty$-complete ([MMY1] ,p.832). Conversely, an
$\infty$-complete path is equal to $\gamma(T)$ for some standard i.e.m.\ with no connection. On the other hand, for a
generalized i.e.m.\ $T$ with no connection, the path $\gamma(T)$ is not always $\infty$-complete.

\begin{definition}
A generalized i.e.m.\ $T$ is {\it irrational} if it has no connection and $\gamma(T)$ is $\infty$-complete. We then call
$\gamma(T)$ the {\it rotation number} of $T$.
\end{definition}

In the circle case $d=2$, the Rauzy diagram has one vertex and two arrows. If the rotation number of a circle homeomorphism
$T$ has a continued fraction expansion $[a_1,a_2,\ldots ]$, the associated $\infty$-complete path takes $a_1$ times the first
arrow, then $a_2$ times the second arrow, $a_3$ times the first arrow, \ldots.

From the definition, a standard i.e.m.\ is irrational iff it has no connection. Two standard i.e.m.\ with no connection are
topologically conjugated iff they have the same rotation number [Y2]. More generally, if $T$ is an irrational g.i.e.m.\ with
the same rotation number than a standard i.e.m.\ $T_0$, then there is, as in the circle case, a semiconjugacy  from $T$ to
$T_0$, i.e a continuous nondecreasing surjective map $h$ from the interval $I$ of $T$ onto the interval $I_0$ of $T_0$ such
that $T_0 \circ h = h \circ T$ (cf.[Y2]).

\subsection{Suspension and genus}
Let $T$ be a standard i.e.m.\ with combinatorial data $\pi= (\pi_t,\pi_b)$. For $\alpha \in \A$ let
$$\lambda_{\alpha} = |\iat|=|\iab|,\quad \tau_{\alpha}= \pi_b(\alpha)-\pi_t(\alpha), \quad \zeta_{\alpha}=\lambda_{\alpha}
+i \tau_{\alpha}.$$

In the complex plane, draw a top (resp. bottom) polygonal line from $u_0$ to $u_d$ through $u_0 + \zeta_{\pi_t^{-1}(1)}, u_0
+ \zeta_{\pi_t^{-1}(1)} +
\zeta_{\pi_t^{-1}(2)}, \ldots$ (resp. $u_0 + \zeta_{\pi_b^{-1}(1)}, u_0 + \zeta_{\pi_b^{-1}(1)} +
\zeta_{\pi_b^{-1}(2)}, \ldots$). These two polygonal lines bound a polygon. Gluing the $\zeta_{\alpha}$ bottom and top sides
of the polygon produces a translation surface $M_T$ ([Zo]). The vertices of the polygon form a set of marked points $\Sigma$
on $M_T$. The cardinality $s$ of $\Sigma$, the genus $g$ of $M_T$ and the number $d$ of intervals are related by
$$d=2g+s-1\;.$$
The genus $g$ can be computed directly from the combinatorial data as follows. Define an antisymmetric matrix $\Omega =
\Omega (\pi)$ by
\begin{equation*}
\Omega_{\alpha \,\beta}=\left\{
\begin{array}{cc}
+1 & \text{if } \pi_t(\alpha)<\pi_t(\beta),\pi_b(\alpha)>\pi_b(\beta),\\
-1 & \text{if } \pi_t(\alpha)>\pi_t(\beta),\pi_b(\alpha)<\pi_b(\beta),\\
0 &\text{otherwise.}
\end{array} \right.
\end{equation*}
Then the rank of $\Omega$ is $2g$. Actually ([Y1],[Y4]), if one identifies $\RA$ with the relative homology group
$H_1(M_T,\Sigma,\Rset)$ via the basis defined by the sides $\zeta_{\alpha}$ of the polygon, the image of $\Omega$ coincides
with the absolute homology group $H_1(M_T,\Rset)$. Another way to compute $s$ (and thus $g$) consists in going around the
marked points, as explained in subsection 3.1.

\subsection{The (discrete time) Kontsevich-Zorich cocycle}

Let ${\mathcal D}$ be a Rauzy diagram on an alphabet $\A$. To each arrow $\gamma$ of ${\mathcal D}$, we associate the matrix
$B_{\gamma} \in SL(\Zset^{\A})$
$$B_{\gamma} =  \Iset + E_{\alpha\,\beta},$$
where $\alpha$ is the loser of $\gamma$, $\beta$ is the winner of $\gamma$, and  $E_{\alpha\,\beta}$ is the elementary
matrix whose only nonzero coefficient is in position $\alpha\,\beta$.
For a path $\gamma$ in $\cD$ made of the successive arrows $\gamma_1\ldots \gamma_l$ we associate the product
$B_\gamma=B_{\gamma_l}\ldots B_{\gamma_1}$. It belongs to $SL(\Zset^{\A})$ and has nonnegative coefficients.

Let $T$ be a g.i.e.m.\ with no connection, whose combinatorial data is a vertex of $\cD$. Let $\whT$ be deduced from $T$ by a
certain number of steps of the Rauzy-Veech algorithm, and let $\gamma$ be the 	associated path of $\cD$. Let $\Gamma$ be
the space of functions on $\sqcup \iat$ which are constant on each $\iat$, and let $\widehat \Gamma$ be the corresponding
subspace for $\whT$. Both $\Gamma$ and $\widehat \Gamma$ are canonically
identified with $\RA$. Then $B_{\gamma}$ is the matrix of the following operator $S$ from $\Gamma$ to $\widehat \Gamma$ :
for $\chi \in \Gamma$
$$S\chi(x) = \sum_{0\leq i <r(x)} \chi(T^i(x))$$
where $x$ belongs to the domain $\widehat I$ of $\whT$ and $r(x)$ is the return time of $x$ in $\widehat I$.

Let ${\mathcal R}$ be the Rauzy class associated to $\cD$. Restricted to standard  i.e.m.\ (considered up to affine
conjugacy), the Rauzy-Veech algorithm defines a map $Q_{RV}$ on the parameter space $ {\mathcal R} \times \Pset (\RA)$. The
operator $S$ define a cocycle over these dynamics called the (extended) Kontsevich-Zorich cocycle.

\section{The cohomological equation revisited}

\subsection{The boundary operator}
Let $T$ be a generalized i.e.m.\ on an interval $I$, $ I = \sqcup \iat = \sqcup\iab$ the associated partitions (mod.$0$),
$\pi = (\pi_t,\pi_b)$ the combinatorial data of $T$ on an alphabet $\A$. We denote by $\,_b\alpha,\,_t\alpha, \alpha_b,
\alpha_t$ the elements of $\A$ such that
$\pi_b(\,_b\alpha) = \pi_t(\,_t\alpha)=1$, $\pi_b(\alpha_b)=\pi_t(\alpha_t)=d$.

We denote by $\AA$ the union of two disjoint copies of $\A$. Elements of $\AA$ are denoted by $(\alpha,L)$ or $(\alpha,R)$
and are associated to the left and right endpoints of the intervals $\iat$ (or $\iab$). More precisely, for $\upsilon \in
\AA$,
we denote by $u^t(\upsilon)$, $u^b(\upsilon)$ the left endpoints of $\iat$, $\iab$ respectively if $\upsilon = (\alpha,L)$,
and  by $u^t(\upsilon)$, $u^b(\upsilon)$ the right endpoints of $\iat$, $\iab$ respectively if $\upsilon = (\alpha,R)$.

Given combinatorial data $\pi = (\pi_t,\pi_b)$, the set $\AA$ is endowed with a permutation $\sigma$ defined as follows:
\begin{eqnarray*}
\sigma(\alpha,R) \;= &(\beta,L), &\quad {\rm when}\; \alpha \ne \alpha_t, \;\pi_t(\beta) = \pi_t(\alpha) +1, \\
\sigma(\alpha_t,R) \;=&(\alpha_b,R), &\\
\sigma(\alpha,L) \;= &(\beta,R), &\quad {\rm when}\; \alpha \ne \, _b\alpha, \;\pi_b(\beta) = \pi_b(\alpha) -1, \\
\sigma(\,_b\alpha,L)\; =&(\,_t\alpha,L). &
\end{eqnarray*}
The cycles of $\sigma$ are canonically associated to the marked  points of any translation surface constructed by suspension
from an i.e.m.\ having $\pi$ as combinatorial data. We denote by $\Sigma$ the set of cycles of $\sigma$, by $s$ the
cardinality of $\Sigma$. We have $d = 2g+s-1$.\\

Let $r\geq 0$ be an integer. We denote by $C^r(\sqcup \iat)$ the space of functions $\varphi$ on $\sqcup \iat$ such that,
for each $\alpha \in \A$, the restriction of $\varphi$ to $\iat$ extends to a $C^r$ function on the closure of $\iat$.

For a function $\varphi$ in $C^0(\sqcup \iat)$ and $\upsilon \in \AA$, we make a slight abuse of notation by writing
$\varphi(\upsilon)$ for the limit of $\varphi$ at the left (resp. right) endpoint of $\iat$ if $\upsilon = (\alpha,L)$
(resp.
$\upsilon = (\alpha,R)$). We also write $\va(\upsilon) =-1$ if $\upsilon = (\alpha,L)$, $\va(\upsilon) =+1$ if $\upsilon =
(\alpha,R)$.

\begin{definition}
The {\it boundary operator} $\partial : C^0(\sqcup \iat) \rightarrow \RS$ is defined by
$$(\partial \varphi)_C = \sum_{\upsilon \in C} \va(\upsilon)\; \varphi (\upsilon),$$
where $C$ is any cycle of $\sigma$. The kernel of the boundary operator is denoted by $C_{\partial}^0(\sqcup \iat)$.
\end{definition}

Note that
\begin{equation}
\sum_{C\in \Sigma} (\partial \varphi)_C = \sum_{\alpha \in \A} (\varphi(\alpha,R) - \varphi(\alpha,L)).
\end{equation}
When $\varphi$ belongs to $C^1(\sqcup \iat)$, this gives
\begin{equation}
\sum_{C\in \Sigma} (\partial \varphi)_C = \int_I D\varphi(x)\;dx.
\end{equation}

The following proposition summarizes the properties of the boundary operator. Recall that $\Gamma \subset C^0(\sqcup \iat)$
is the set of functions which are constant on each $\iat$. We denote by $\RSO$ the hyperplane of $\RS$ formed by the
vectors
for which the sum of the coordinates vanishes.

Let $M$ be a translation surface constructed by suspension from a standard i.e.m.\ $T_0$ having $\pi$ as combinatorial data.
Then, we can identify $\Sigma$ with the set of marked points on $M$, $\Gamma$ with the relative homology group
$H_1(M,\Sigma, \Rset)$ (the characteristic function of $\iat$ corresponds to oriented parallel sides with label $\alpha$ of
the polygon which gives rise to $M$ after the gluing). It is then clear that the operator $\partial$ restricted to $\Gamma$
is indeed the boundary operator
$$  \partial: H_1(M,\Sigma, \Rset) \rightarrow H_0(\Sigma,\Rset) = \RS \,.$$
\begin{proposition}
\begin{enumerate}
\item For  a g.i.e.m.\ $T$ with combinatorial data $\pi$, and $\psi \in C^0(\overline I)$, one has $\partial \psi =
    \partial (\psi \circ T)$.
\item The kernel $\Gamma_{\partial}$ of the restriction of $\partial$ to $\Gamma$ is the image of $\Omega(\pi)$, and the
    image is $\RSO$.
\item The boundary operator $\partial : C^0(\sqcup \iat) \rightarrow \RS$ is onto.
\item Let $T$ be a g.i.e.m.\ with combinatorial data $\pi$, and let $\wt T$, acting on a subinterval $\wt I \subset I$, be
    obtained from $T$ by one or several steps of the Rauzy-Veech algorithm. For $\varphi \in C^0(\sqcup \iat)$, denote
    by
$S \varphi \in C^0(\sqcup \wt {I}_{\alpha}^t)$ be the special Birkhoff sums corresponding to the first return in $\wt
I$. Then we have
$$\partial (S \varphi) = \partial \varphi ,$$
where the left-hand side boundary operator is defined using the combinatorial data $\wt \pi$ of $\wt T$.
\end{enumerate}
\end{proposition}

\begin{proof} Let $\psi \in C^0(\overline I)$, $C \in \Sigma$. For $\upsilon = (\alpha,R) \in C$ with $\alpha \ne \alpha_t$,
we have
$u^t(\upsilon) = u^t(\sigma(\upsilon))$ with $\va(\upsilon) = - \va(\sigma(\upsilon))$. Therefore $(\partial \psi)_C =
\va_1
\psi(1) - \va_0 \psi(0)$, where $\va_0$ (resp. $\va_1$) is $1$ or $0$ depending whether $(\,_t \alpha,L)$ (resp. $(\alpha_t
,R)$) belongs or not to $C$.

Similarly, for $\upsilon = (\alpha,L) \in C$ with $\alpha \ne \,_b\alpha$, we have
$u^b(\upsilon) = u^b(\sigma(\upsilon))$ with $\va(\upsilon) = - \va(\sigma(\upsilon))$. Therefore $(\partial (\psi \circ
T))_C = \va'_1
\psi(1) - \va'_0 \psi(0)$, where $\va'_0$ (resp. $\va'_1$) is $1$ or $0$ depending whether $(\,_b \alpha,L)$ (resp.
$(\alpha_b ,R)$) belongs or not to $C$.

As $\sigma(\,_b \alpha,L)= (\,_t\alpha,L)$ and $\sigma(\alpha_t,R) = (\alpha_b,R)$, we have $\va_0 =\va'_0$ and $\va_1=
\va'_1$. This proves (1).

The restriction of the operator $\partial$ to $\Gamma$ has been described in homological terms just before the proposition.
It follows from this description that the image of $\Gamma$ by $\partial$ is indeed $\RSO$. As the image of $\Omega(\pi)$ is
identified with the image of the absolute homology group $H_1(M,\Rset)$ in $H_1(M,\Sigma, \Rset)$, the proof of (2) is
complete.

Let $\phi^*(x) =x$; then $\sum_{C\in \Sigma} (\partial \phi^*)_C = 1$. Thus the image of $\partial$ is strictly bigger than
$\RSO$, which proves (3).

To prove (4), it is sufficient to consider the case where $\wt T$ is obtained from $T$ by one step of the Rauzy-Veech
algorithm. We assume that this step is of top type, the case of bottom type being symmetric. Denote by $\wt \sigma$ the
permutation of $\AA$ defined from the combinatorial data
$\wt \pi$ of $\wt T$, by $\alpha'_t$ the element of $\A$ such that $\pi_b(\alpha'_t) = \pi_b(\alpha_t)+1$, by $\wt \alpha_b$
the element of $\A$ such that $\pi_b(\wt \alpha_b) = d-1$. We have $\sigma(\upsilon) = \wt \sigma(\upsilon)$, except for
\begin{eqnarray*}
\sigma(\alpha'_t, L) =& (\alpha_t,R), \quad \wt \sigma(\alpha'_t,L) =& (\alpha_b,R),\\
\sigma(\alpha_t,R) =&(\alpha_b,R), \quad \wt \sigma(\alpha_t,R) =& (\wt \alpha_b,R),\\
\sigma(\alpha_b,L)=& (\wt \alpha_b,R) ,\quad  \wt \sigma(\alpha_b,L) =& (\alpha_t,R).
\end{eqnarray*}
Let $\varphi \in C^0(\sqcup \iat)$. For $\upsilon \in \AA$,
we have $S\varphi(\upsilon) = \varphi(\upsilon)$ except for
\begin{eqnarray*}
S\varphi(\alpha_b,L) &=& \varphi(\alpha_b,L) + \varphi(u^b_{d-1}),\\
S\varphi(\alpha_b,R) &=& \varphi(\alpha_b,R) + \varphi(\alpha_t,R),\\
S\varphi(\alpha_t,R) &=& \varphi(u^b_{d-1}).
\end{eqnarray*}
From these formulas, it is easy to see that $\partial(S\varphi) = \partial \varphi$.
\end{proof}

\begin{remark}
Let $\varphi \in C_{\partial}^0(\sqcup \iat)$ such that $\varphi(\upsilon) = 0$ for all $\upsilon \in \AA$. Assume also that
there exists $\psi \in C(\overline I )$ such that $\varphi = \psi \circ T - \psi$. Then, given such a function $\psi$, there
is a family $(\psi_C)_{C \in \Sigma}$ such that
$$ \psi(u^t(\upsilon)) = \psi(u^b(\upsilon)) = \psi_C \,,$$
for all $\upsilon \in C$, all $C \in \Sigma$. The function $\psi$, hence also the family $(\psi_C)_{C \in \Sigma}$, is only
well-defined up to an additive constant by $\varphi$. We will denote by $\nu(\varphi)$ the image in $\RS / \Rset$ of the
family  $(\psi_C)_{C \in \Sigma}$.
\end{remark}

\subsection{Continuity of the solutions of the cohomological equation}

The main tool in [MMY1] to obtain bounded solutions of the cohomological equations was the Gottschalk-Hedlund theorem ([GH],
[He1]).

\begin{theorem}
Let $f$ be a minimal homeomorphism of a compact metric space $X$, and let $\varphi$ be a continuous function on $X$. The
following properties are equivalent:
\begin{enumerate}
\item $\varphi = \psi \circ f - \psi$, for some continuous function $\psi$ on $X$;
\item $\varphi = \psi \circ f - \psi$, for some bounded function $\psi$ on $X$;
\item there exists $C>0$ such that the Birkhoff sums of $\varphi$  satisfy $|S_n \varphi(x)|<C$ for all $n\in \Zset$,
    $x\in X$;
\item there exists $C>0$, $x_0 \in X$ such that the Birkhoff sums of $\varphi$  satisfy $|S_n \varphi(x_0)|<C$ for all
    $n\geq 0$.
\end{enumerate}
\end{theorem}

Let $T$ be a (standard) i.e.m.\ with no connection. Let
$$Z:=\{T^{-m}(u^t_i),T^n(u^b_j); \; 0<i,j<d, \;m\geq 0,\, n\geq 0\}$$
be the union of the orbits of the singularities of $T$ and $T^{-1}$. In the interval $\overline I$ where $T$ is acting, we
split the points of $Z$ into a right and left limits to get a compact metric space $\wh I$ (homeomorphic to a Cantor set) on
which $T$ induces a minimal homeomorphism $\wh T$.
Denote by $p$ the canonical projection from $\wh I$ onto $\overline I$, so that $p \circ \wh T = T \circ p$.
Let $\varphi \in C^0(\sqcup \iat)$; then $\wh \varphi:= \varphi \circ p$ is continuous on $\wh I$. Assume that the Birfhoff
sums $(S_n \varphi)_{n\geq 0}$ of $\varphi$ for $T$ are bounded. Then the same is true for the Birkhoff sums of $\wh
\varphi$ for $\wh T$ and we conclude from Gottschalk-Hedlund's theorem that there exists a continuous function $\wh \psi$ on
$\wh I$ such that $\wh \varphi = \wh \psi \circ \wh T - \wh \psi$.
For an arbitrary continuous function $\wh \psi$ on $\wh I$, there is a priori no continuous function $\psi$ on $\overline I$
such that $\wh \psi = \psi \circ p$. However, we have the following elementary result, which was not observed in [MMY1].

\begin{proposition}
 Let $\wh\psi$ be a continuous function on $\wh I$. Assume that $\wh \varphi =\wh\psi \circ \wh T -\wh \psi $ is induced by
 a function $\varphi \in C^0(\sqcup \iat)$. Then $\wh \psi$ is induced by a continuous function on $\overline I$.
\end{proposition}
\begin{proof}
The continuous function $\wh \psi$ on $\wh I$
is induced by a continuous function on $\overline I$ iff, for every $z \in Z$, the values of $\wh \psi$ on the two points
$z_l,\; z_r$ of $\wh I$ sitting over $z$ are equal. For $z \in Z$, let $\delta \psi (z):= \wh \psi (z_r) - \wh \psi (z_l)$.
If $z \ne u^t_i$, the values of  $\wh\psi -\wh \psi \circ \wh T$
at $z_l$ and $z_r$ are the same, hence $\delta \psi (z) = \delta \psi (T(z))$. Therefore, for every $0< i,j <d, \;m, n \geq
0$, we have $\delta \psi (T^{-m}(u^t_i))=\delta \psi (u^t_i)$ and $\delta \psi (T^{n}(u^b_j))=\delta \psi (u^b_j)$. As $\wh
\psi$ is continuous on $\wh I$ and every half orbit
$\{T^n(u^b_j);\; n\geq 0 \}$ or $\{T^n(u^t_i);\; n\leq 0 \}$ is dense, we must have $\delta \psi(z)=0$ for all $z \in Z$,
and the conclusion of the proposition follows.
\end{proof}
\begin{corollary}
Let $T$ be a (standard) i.e.m.\ with no connection, $\varphi \in C^0(\sqcup \iat)$. If the Birfhoff sums $(S_n \varphi)_{n\geq
0}$ of $\varphi$ for $T$ are bounded, there exists $\psi \in C^0(\overline I)$ such that $\varphi = \psi \circ T - \psi$.
\end{corollary}

\subsection{Interval exchange maps of Roth Type}

We recall the diophantine condition on the rotation number of an i.e.m.\ introduced in [MMY1].

Let $\underline \gamma$ be an $\infty$-complete path in a Rauzy diagram $\mathcal D$. Write $\underline \gamma$ as an
infinite concatenation
$$\underline \gamma = \gamma(1) * \cdots * \gamma(n)* \cdots $$
of finite complete paths of minimal length. Define then, for $n > 0$
$$Z(n) := B_{\gamma(n)},\quad B(n):= B_{\gamma(1)*\cdots *\gamma(n)} = Z(n)\cdots Z(1).$$

We introduce three conditions.

\hspace{1cm} (a) \hspace{5mm} For all $\tau>0$, $||Z(n+1)|| = \mathcal {O} (||B(n)||^{\tau})$. \\

\hspace{1cm} (b) \hspace{5mm} There exists $\theta >0$ and a hyperplane $\Gamma_0 \subset \Gamma = \RA$ such that
$$||B(n)_{|\Gamma_0}|| = \mathcal {O} (||B(n)||^{1-\theta})\,.$$

\hspace{1cm} (c) \hspace{5mm} Define
$$\Gamma_s = \{ \chi \in \Gamma, \exists \tau >0,\; ||B(n)\chi|| = \mathcal {O} (||B(n)||^{-\tau}) \}\,$$
and $\Gamma_s^{(n)} := B(n)\Gamma_s$ for $n \geq 0$. For $k<\ell$, denote by $B_s(k,\ell)$ the restriction of
$B(k,\ell):=B_{\gamma(k+1)*\cdots*\gamma(\ell)}$
to $\Gamma_s^{(k)}$ and by $B_{\flat}(k,\ell)$ the operator from $\Gamma /\Gamma_s^{(k)}$ to $\Gamma /\Gamma_s^{(\ell)}$
induced by $B(k,\ell)$.
We ask that, for all $\tau >0$,
$$||B_s(k,\ell)||= \mathcal {O} (||B(\ell)||^{\tau}), \quad ||(B_{\flat}(k,\ell))^{-1}|| = \mathcal {O}
(||B(\ell)||^{\tau})\,.$$

\begin{remark}
\begin{enumerate}
\item The definition of $Z(n)$ is slightly different from the definition in [MMY1], but an elementary computation shows
    that condition (a) with the present definition is equivalent to condition (a) with the definition of [MMY1].
\item Condition (b) is formulated in a slightly different way than in [MMY1], in order to depend only on the rotation
    number and not of the length data.
But actually condition (b) implies that there exists exactly one normalized standard i.e.m.\ $T$ with rotation number
$\underline \gamma$, that $T$ is uniquely ergodic, and that the hyperplane $\Gamma_0$ of condition (b) must be formed on
functions in $\Gamma$ with mean $0$ on $I$.
\item For any combinatorial data, the set of length data for which the associated i.e.m.\ has a rotation number satisfying
    (a), (b), (c) has full measure.
For (c), this is an immediate consequence of Oseledets theorem. For (b), it is a consequence from the fact that the
larger Lyapunov of the KZ-cocycle is simple (Veech). For (a), a proof is provided in [MMY1], but much better diophantine
estimates were later obtained in [AGY].
\item It follows from Forni's theorem [For2] on the hyperbolicity of the KZ-cocycle that, for almost all rotation
    numbers, one has $\dim \Gamma_s =g$.
\end{enumerate}
\end{remark}

\begin{definition}
A rotation number $\underline \gamma$ (or a standard i.e.m.\ $T$ having $\underline \gamma$ as rotation number) is of {\it
Roth type} if the three
conditions (a), (b), (c) are satisfied. It is of {\it restricted Roth type} if moreover one has $\dim \Gamma_s =g$.
\end{definition}

\begin{remark}
Let $T$ be a standard i.e.m.\ of restricted Roth type. Then $\Gamma_s$ is exactly equal to the subspace $\Gamma_T \subset
\Gamma$ of functions $\chi \in \Gamma$ which can be written as $\psi \circ T - \psi$, for some $\psi \in C^0(\overline I)$.
Indeed, we have $\Gamma_s \subset \Gamma_T$ from [Zo2] or [MMY1]. On the other hand, $\Gamma_T$ is contained in the subspace
$\Gamma^{\sharp} \subset \Gamma$ of functions which go to $0$ under the KZ-cocycle. As the KZ-cocycle acts trivially on
$\Gamma/ \Gamma_{\partial}$, $\Gamma^{\sharp}$ is contained in $\Gamma_{\partial}$ and is actually an isotropic subspace of
this symplectic space. As $\Gamma_s \subset \Gamma^{\sharp}$ and  $\dim \Gamma_s =g$, we must have $\Gamma_T= \Gamma_s =
\Gamma^{\sharp}$.
\end{remark}

Let $T$ be a standard i.e.m.\ of restricted Roth type. Choose a $g$-dimensional subspace $\Gamma_u  \subset \Gamma_{\partial}$
such that $\Gamma_{\partial} = \Gamma_s \oplus \Gamma_u$. We recall the main result of [MMY1], in the form that is convenient
for our purpose. We denote by
$C^{1+BV}(\sqcup \iat)$ the space of functions $\varphi \in C^1(\sqcup \iat)$ such that $D\varphi$ is a function of bounded
variation. We write
$$|D\varphi|_{BV} = \sum_{\alpha} {\rm Var}_{\iat} D\varphi,\quad ||\varphi||_{1+BV} = ||\varphi||_0 + ||D\varphi||_0 +
|D\varphi|_{BV}.$$
We denote by $C_{\partial}^{1+BV}(\sqcup \iat)$ the intersection of $C^{1+BV}(\sqcup \iat)$ with the kernel of the boundary
operator $\partial$.

\begin{theorem}
Let $T$ be a standard i.e.m.\ of restricted Roth type. There exist bounded linear operators $L_0: \varphi \mapsto \psi$ from
$C_{\partial}^{1+BV}(\sqcup \iat)$ to $C^0(\overline I)$ and $L_1: \varphi \mapsto \chi$ from $C_{\partial}^{1+BV}(\sqcup
\iat)$ to $\Gamma_u$ such that, for all $\varphi \in
C_{\partial}^{1+BV}(\sqcup \iat)$, we have
$$\varphi = \chi + \psi \circ T -\psi  \;.$$
\end{theorem}

\begin{remark} In [MMY1], the result was formulated in the following weaker way: for every $\varphi \in C^{1+BV}(\sqcup
\iat)$ with $\int_I D\varphi(x)\,dx =0$, there exists $\chi \in \Gamma$ and a bounded function $\psi $ on $I$ such that
$\varphi = \chi + \psi \circ T - \psi$.
To obtain the present stronger form, we observe that
\begin{itemize}
\item by Proposition 3.5 or Corollary 3.6, the solution $\psi$ is automatically continuous on $\overline I$;
\item the condition $\int_I D\varphi(x)\,dx =0$ means that we ask that the sum of the components of $\partial \varphi$
    is $0$. In view of Proposition 3.2, part (2), it is then possible to substract $\chi \in \Gamma$ in order to have
    $\partial (\varphi - \chi) =0$. However, in view of Proposition 3.2, part (1), it is more natural to start with
    $\varphi \in  C_{\partial}^{1+BV}(\sqcup \iat)$. Then, the correction $\chi$ must belong to $\Gamma_{\partial}$. As
    $\Gamma_s = \Gamma_T$ from Remark 3.8, there is a unique way to find the correction $\chi \in \Gamma_u$ in order to
    have $\varphi - \chi = \psi \circ T - \psi$ for some $\psi \in C^0(\overline I)$.
\item That the operator $\varphi \mapsto \psi$ (and consequently also the operator $\varphi \mapsto \chi$ ) is bounded
    follows from the proof in [MMY1]. One shows that , for some $\chi \in \Gamma$, the Birkhoff sums of $\varphi - \chi$
    satisfy
$$||S_n(\varphi - \chi)||_{0} \leq C ||\varphi ||_{1+BV}\;,$$
and then Gottschalk-Hedlund's theorem imply that $||\psi||_0 \leq C \,||\varphi ||_{1+BV}$.
\end{itemize}
In Appendix A, we show that it is possible to deal in the same way with functions $\varphi \in C^{1+\tau}(\sqcup \iat)$, for
any $\tau >0$.
\end{remark}

\subsection{The cohomological equation in higher smoothness}

This subsection is a slight modification of the corresponding subsection in [MMY1], taking the boundary operator into
account. We assume that $T$ is a standard i.e.m.\ with no connection.

For $r \geq 1$ we denote by $\Gamma(r)$ the set of functions $\chi \in C^{\infty}(\sqcup \iat)$ such that the restriction of
$\chi$ to each $\iat$ is a polynomial of degree $<r$, by $\Gamma_{\partial}(r)$ the subspace of functions $\chi \in
\Gamma(r)$ which satisfy $\partial D^i\chi =0$ for all $0\leq i <r$, by $\Gamma_T(r)$ the subspace of functions $\chi \in
\Gamma(r)$ which can be written as $\psi \circ T - \psi$ for some $\psi \in C^{r-1}(\overline I)$. We observe that for $\psi
\in C^{r-1}(\overline I)$, we have $\partial D^i (\psi \circ T - \psi ) =0$ for all $0\leq i <r$, hence $\gamma_T(r) \subset
\Gamma_{\partial}(r)$.

\begin{proposition}
One has
$${\rm dim} \Gamma(r) =rd, \quad {\rm dim} \Gamma_{\partial}(r) =(2g-1)r+1, \quad {\rm dim} \Gamma_T(r) ={\rm dim} \Gamma_T
+r-1\,.$$
\end{proposition}
\begin{proof}
The first assertion is obvious. For $r\geq 1$, the derivation operator $D$ sends $\Gamma(r+1)$ into $\Gamma(r)$,
$\Gamma_T(r+1)$ into $\Gamma_T(r)$, $\Gamma_{\partial}(r+1)$ into $\Gamma_{\partial}(r)$.

Let $\chi \in \Gamma_T(r)$. Write $\chi = \psi \circ T - \psi$ with $\psi \in C^{r-1}(\overline I)$. Let $\psi_1 \in
C^r(\overline I)$ a primitive of
$\psi$ and $\chi_1 := \psi_1 \circ T - \psi_1$. Then $\chi_1$ belongs to $ \Gamma_T(r+1)$ and $D\chi_1 = \chi$.
Therefore $D: \Gamma_T(r+1) \rightarrow \Gamma_T(r)$ is onto. If $\chi_1 \in \Gamma_T(r+1)$ satisfies $D\chi_1 =0$,
we write $\chi_1 = \psi_1 \circ T - \psi_1$ with $\psi_1 \in C^r(\overline I)$. Then $\psi:= D\psi_1$ is continuous and
$T$-invariant,
hence constant (as T is minimal), which implies that $ \varphi_1 \in \Rset \delta$. Conversely, $\Rset \delta$ is contained
in the kernel of
$D:\; \Gamma_T(r+1) \rightarrow \Gamma_T(r)$, hence equal to this kernel. We conclude that
$\dim \Gamma_T(r) = \dim \Gamma_T +r-1$.

Let $\chi \in \Gamma_{\partial} (r)$. Let $\chi_1 \in \Gamma(r+1)$ with $D \chi_1 = \chi$. We have $\chi_1 \in
\Gamma_{\partial}(r+1)$ iff $\partial \chi_1 =0$. The sum of the components of $\partial \chi_1$ is equal to $\int_I
\chi(x)\,dx$, hence $\int_I \chi(x)\,dx =0$ is a necessary condition for $\chi$ to be in the image by $D$ of
$\Gamma_{\partial}(r+1)$.
On the other hand, the condition is also sufficient by Proposition 3.2, part (2). Also by Proposition 3.2, part (2), the
kernel of $D$ in $\Gamma_{\partial} (r)$ is $\Gamma_{\partial} = {\rm Im} \Omega(\pi)$ which is of dimension $2g$. We
conclude by induction on $r$ that ${\rm dim} \Gamma_{\partial}(r) =(2g-1)r+1$.
\end{proof}
We  define $C^{r+BV}(\sqcup \iat)$ as the space of functions $\varphi \in C^{r}(\sqcup \iat)$ such that $D^r \varphi$ is of
bounded variation. We endow this space with its natural norm. We denote by $C_{\partial}^{r+BV}(\sqcup \iat)$ the subspace
of $\varphi \in
C^{r+BV}(\sqcup \iat)$ such that  $\partial D^i\varphi =0$ for all $0\leq i <r$.

\begin{theorem}
There exists a bounded operator $\Pi: C_{\partial}^{r+BV}(\sqcup \iat) \rightarrow \Gamma_{\partial}(r)/\Gamma_T(r)$,
extending the canonical projection from $\Gamma_{\partial}(r)$ to $\Gamma_{\partial}(r)/\Gamma_T(r)$, and a bounded operator
$\varphi \mapsto \psi$ from the kernel of $\Pi$
 to $C^{r-1}(\overline I)$ such that, if $\varphi \in  C_{\partial}^{r+BV}(\sqcup \iat)$ satisfies $\Pi(\varphi) =0$, then
 we have
 $$\varphi = \psi \circ T - \psi \;.$$
\end{theorem}

In other terms, if we choose a subspace $\Gamma_u(r) \subset \Gamma_{\partial}(r)$ such that $\Gamma_{\partial}(r) =
\Gamma_T(r) \oplus \Gamma_u(r)$
and identify the quotient $\Gamma_{\partial}(r)/\Gamma_T(r)$ with $\Gamma_u(r)$, we can write any $\varphi \in
C_{\partial}^{r+BV}(\sqcup \iat) $ as
$\varphi = \Pi(\varphi) + \psi \circ T - \psi$, with $\psi \in C^{r-1}(\overline I)$.

\begin{proof} The proof is by induction on $r$, the case $r=1$ being the theorem above. Assume that $r>1$ and the result is
true for $r-1$. Let $\varphi \in C_{\partial}^{r+BV}(\sqcup \iat)$. According to the induction hypothesis, we can write
$$D \varphi = \chi_1 + \psi_1 \circ T - \psi_1 \;,$$
where $\chi_1 \in \Gamma_{\partial}(r-1)$ and $\psi_1 \in C^{r-2}(\overline I)$. Let $\psi \in C^{r-1}(\overline I)$ a
primitive of $\psi_1$. Then there exists a primitive $\chi $ of $\chi_1$ such that
$$ \varphi = \chi + \psi \circ T - \psi \;.$$
As $\partial \varphi =0$, we must also have $\partial \chi =0$ and thus $\chi$ belongs to $ \Gamma_{\partial}(r)$. This
completes the proof of the induction step.
\end{proof}

\section{A conjugacy invariant}

\subsection{Definition of the invariant}

Let $r$ be an integer $\geq 1$ or $\infty$. We denote by $J^r$ the group of $r$-jets at $0$ of orientation preserving
diffeomorphisms
of $\Rset$ fixing $0$.

Let $\pi = (\pi_t, \pi_b)$ an element of a Rauzy class $\mathcal R$ on an alphabet $\A$, and let $T$ be a generalized i.e.m.\
of class $C^r$
with combinatorial data $\pi$. For each $\upsilon \in \AA$, we define an element $j(T,\upsilon) \in J^r$ as the $r$-jet at
$0$ of
$$x \mapsto T(u^t(\upsilon) +x) -u^b(\upsilon)\,,$$
where $x$ varies in an interval of the form $(0,x_0)$ when $\upsilon =(\alpha,L)$, $(-x_0,0)$ when $\upsilon =(\alpha,R)$,
and the
$r$-jet at $0$ exists by definition of a generalized i.e.m.\ of class $C^r$.

For each cycle $C$ of $\sigma$, we choose an element $\upsilon_0 \in C$ and we write
$C = \{\upsilon_0, \upsilon_1 = \sigma (\upsilon_0),\cdots,\upsilon_{\kappa}\}$. We then define
$$J(T,C) := j(T,\upsilon_0)^{\varepsilon(\upsilon_0)}j(T,\upsilon_1)^{\varepsilon(\upsilon_1)}\cdots
j(T,\upsilon_{\kappa})^{\varepsilon(\upsilon_{\kappa})}\;\; \in J^r,$$
where, as in subsection 3.1, we have $\va(\upsilon) =-1$ if $\upsilon = (\alpha,L)$, $\va(\upsilon) =+1$ if $\upsilon =
(\alpha,R)$.

\medskip

\begin{definition} The invariant $J(T)$ of $T$ is the family, parametrized by the cycles $C \in \Sigma$ of $\sigma$, of the
conjugacy classes in $J^r$ of the $J(T,C)$.
\end{definition}

It is clear that the conjugacy class of $J(T,C)$ does not depend on the choice of the element $\upsilon_0 \in C$.

When $d=2$, the invariant $J(T)$ is the obstruction for $T$ to be $C^r$-conjugated to a $C^r$-diffeomorphism of the circle.

\subsection{ Conjugacy classes in $J^r$}

The classification of elements in $J^{\infty}$ up to conjugacy is well-known and a simple exercise. The classification in $J^r$ for finite $r$ is an
obvious consequence, truncating to order $r$ the Taylor developments. It is not used in the rest of the paper.

Let $j$ be an element of $J^{\infty}$. If $j$ is distinct from the neutral element of $J^{\infty}$, its contact with the
identity is an integer
$k\geq 1$.

If $k=1$, i.e the linear part of $j$ is distinct from the identity, then $j$ is conjugate to its linear part.

If $k>1$, there exists in the conjugacy class of $j$ a unique element of the form $x \mapsto x \pm x^k + a x^{2k-1}$ ($a \in
\Rset$).

\subsection {Invariance under conjugacy}

Let $r$ be an integer $\geq 1$ or $\infty$. Let $\pi = (\pi_t, \pi_b)$ an element of a Rauzy class $\mathcal R$ on an
alphabet $\A$, and let $T$ be a generalized i.e.m.\ of class $C^r$
with combinatorial data $\pi$.

Let $h$ be a $C^r$ orientation preserving diffeomorphism from the interval $I =(u_0,u_d)$ for $T$ to some other open bounded
interval
$\widehat I = (\widehat u_0, \widehat u_d)$, which extends to a $C^r$ diffeomorphism from the closure of $I$ to the closure
of $\widehat I$.

Let $\widehat T = h \circ T \circ h^{-1}$, acting on $\widehat I$.

\begin{proposition}
The invariants of $T$ and $\widehat T$ are the same.
\end{proposition}
\begin{proof}
 For $x_0$ in the closure of $I$, let $j(h,x_0)$ be the $r$-jet at $0$ of $x \mapsto h(x_0 +x) - h(x_0)$. For $\upsilon \in
 \AA$, we have
$$j(\wh T ,\upsilon) = j(h, u^b(\upsilon)) j(T,\upsilon) j(h,u^t(\upsilon))^{-1} .$$
Writing $j(\wh T ,\upsilon)^{\va(\upsilon)} = j(h,x_+ (\upsilon)) j( T ,\upsilon)^{\va(\upsilon)} j(h,x_- (\upsilon))^{-1}$,
we check that for all $\upsilon \in \AA$ we have $x_- (\upsilon) = x_+(\sigma(\upsilon))$:
\begin{itemize}
\item if $\upsilon =(\alpha,R)$, $\alpha \ne \alpha_t$, then $x_- (\upsilon) = x_+(\sigma(\upsilon))=u^t(\upsilon)$;
\item if $\upsilon =(\alpha,L)$, $\alpha \ne \,_b\alpha$, then $x_- (\upsilon) = x_+(\sigma(\upsilon))=u^b(\upsilon)$;
\item if $\upsilon =(\alpha_t ,R)$, then $x_- (\upsilon) = x_+(\sigma(\upsilon))=1$;
\item if $\upsilon =(\,_b\alpha ,L)$, then $x_- (\upsilon) = x_+(\sigma(\upsilon))=0$.
\end{itemize}
We obtain therefore, for $C \in \Sigma$, $\upsilon_0$ as in the definition of $J(T,C)$
$$ J(\wh T,C) = j(h, x_+(\upsilon_0)) J(T,C) j(h, x_+(\upsilon_0))^{-1}.$$
The proof of the proposition is complete.
\end{proof}

\subsection {Invariance under renormalization}
Let $r$ be an integer $\geq 1$ or $\infty$. Let $\pi = (\pi_t, \pi_b)$ an element of a Rauzy class $\mathcal R$ on an
alphabet $\A$, and let $T$ be a generalized i.e.m.\ of class $C^r$
with combinatorial data $\pi$.

 We assume that $u^t_{d-1} \ne u^b_{d-1}$, so we can perform one step of the Rauzy-Veech algorithm to obtain
a generalized i.e.m.\ $\widetilde T$, which is also of class $C^r$. We denote by $\widetilde \pi$ the combinatorial data for
$\widetilde T$.
As in Proposition 3.2, part (4), the set of cycles of the permutation $\wt \sigma$ of $\AA$ induced by $\wt \pi$ is
naturally identified with $\Sigma$.

\begin{proposition}
The invariants of $T$ and $\widetilde T$ are the same.
\end{proposition}
\begin{proof}
We assume that the step of the Rauzy-Veech algorithm from $T$ to $\wt T$ is of top type, the case of bottom type being
symmetric. Denote by $\alpha'_t$ the element of $\A$ such that $\pi_b(\alpha'_t) = \pi_b(\alpha_t)+1$, by $\wt \alpha_b$ the
element of $\A$ such that $\pi_b(\wt \alpha_b) = d-1$. We have $\sigma(\upsilon) = \wt \sigma(\upsilon)$, except for
\begin{eqnarray*}
\sigma(\alpha'_t, L) =& (\alpha_t,R), \quad \wt \sigma(\alpha'_t,L) =& (\alpha_b,R),\\
\sigma(\alpha_t,R) =&(\alpha_b,R), \quad \wt \sigma(\alpha_t,R) =& (\wt \alpha_b,R),\\
\sigma(\alpha_b,L)= & (\wt \alpha_b,R) ,\quad  \wt \sigma(\alpha_b,L) =& (\alpha_t,R).
\end{eqnarray*}
On the other hand, we have $j(\wt T, \upsilon) = j(T,\upsilon)$ except for
\begin{eqnarray*}
j(\wt T, (\alpha_t,R)) & = & j^*\,\\
j(\wt T, (\alpha_b,L)) & = & j^* j(T,(\alpha_b,L)) \,,\\
j(\wt T, (\alpha_b,R)) & = & j(T,\alpha_t,R)) j(T,(\alpha_b,R))\, ,
\end{eqnarray*}
where $j^*$ is the $r$-jet of $T$ at $u^t_{d-1}$. Thus, we have
\begin{equation*}
j(\wt T, (\alpha_b,L))^{-1} \,j(\wt T, (\alpha_t,R)) = j(T,(\alpha_b,L))^{-1}\,.
\end{equation*}
In view of the formulas for $\wt \sigma$, we obtain the cancellations that prove the proposition.
\end{proof}

\subsection{Relation with the boundary operator}
Let $r$ be an integer $\geq 1$ or $\infty$. Let $\pi = (\pi_t, \pi_b)$ an element of a Rauzy class $\mathcal R$ on an
alphabet $\A$, let $T_0$ be a standard i.e.m.\
with combinatorial data $\pi$, and let $(T_t)_{t \in [-t_0,t_0]}$ be a family of g.i.e.m.'s of class $C^r$ through $T_0$ with
the same combinatorial data $\pi$.

We assume that that $t \mapsto T_t$ is of class $C^1$ in the following sense: denote by $u^t_1(t)<\ldots <u^t_{d-1}(t)$
the singularities of $T_t$, by $u^b_1(t)<\ldots <u^b_{d-1}(t)$ those of $T_t^{-1}$; then the functions $t \mapsto u^t_i(t)$,
$t \mapsto u^b_j(t)$ are of class $C^1$; moreover, for each $\alpha \in \A$, each $0\leq i \leq r$, the partial derivative
$\partial _{t} \partial _x^i  T_t(x)$ should be defined on $\{(t,x); \,t\in [-t_0,t_0],\, x\in \iat (t)\,\}$ and extend to a
continuous function on the closure of this set (i.e including the endpoints of $\iat (t)$).

The function $\varphi(x):= \frac{d}{dt}_{|t=0} T_t(x)$ is then an element of $C^r(\sqcup \iat)$.

\begin{proposition}
\begin{enumerate}
\item One has $\partial \varphi = 0$. Conversely, for any $\varphi \in C^r_{\partial}(\sqcup \iat)$, there exists a
    $C^1$-family of g.i.e.m.'s of class $C^r$ such that $\varphi= \frac{d}{dt}_{|t=0} T_t$.
\item Assume that, for some $1 \leq k \leq r$,  the conjugacy invariant of $T_t$ in $J^k$ is trivial for all $t\in
    [-t_0,t_0]$. Then one has $\partial D^{\ell} \varphi = 0$ for all $1 \leq \ell \leq k$.
\end{enumerate}
\end{proposition}

\begin{proof}
\begin{enumerate}
\item For $\upsilon \in \AA$, let $\delta u^t(\upsilon) = \frac{d}{dt}_{|t=0} u^t(\upsilon,t),
\delta u^b(\upsilon) = \frac{d}{dt}_{|t=0} u^b(\upsilon,t)$. Differentiating at $t=0$ the relation
$T_t(u^t(\upsilon,t))=u^b(\upsilon,t)$ gives
$\delta u^t(\upsilon) + \varphi (u^t(\upsilon)) = \delta u^b(\upsilon)$, from which $\partial \varphi =0$ follows
easily.

Conversely, for $\varphi \in C^r_{\partial}(\sqcup \iat)$, one can choose the $u^t_i(t), u^b_j(t)$ such that $\delta
u^t(\upsilon) + \varphi (u^t(\upsilon)) = \delta u^b(\upsilon)$. Then it is easy to complete the construction to get a
family $(T_t)$ with the required property.
\item Writing the $k$-jet of a germ of $C^r$-diffeomorphism $f$ of $(\Rset ,0)$ as $(Df(0),\cdots,D^k f(0))$, we have,
    for $\upsilon \in \AA$
$$ \frac{d}{dt}_{|t=0} j(T_t,\upsilon) = (D\varphi(u^t(\upsilon)), \cdots ,D^k \varphi(u^t(\upsilon))).$$
As the product close to the identity in any Lie group (like $J^k$) is commutative up to second order terms, the
assertion of the proposition follows.
\end{enumerate}
\end{proof}

\begin{remark}
The conjugacy invariant in $J^1$ (which is commutative) can be defined directly from the boundary operator: identifying
$J^1$ with $\Rset$ by associating to a germ $f$ the logarithm of its derivative at $0$, we have indeed that the invariant in
$J^1$ of a g.i.e.m.\ $T$ of class $C^1$ is $\partial \log DT$, where $\log DT$ is considered as a function in $C(\sqcup
\iat)$.
\end{remark}

\begin{remark}
For an affine i.e.m.\ $T$, the invariant in $J^{\infty}$ coincides with the invariant in $J^1$. The function $\log DT$ belongs
to $\Gamma$, and the invariant takes its values in $\RSO$.
\end{remark}

\section{The main theorem: statement and reduction to the simple case}

\subsection{The setting}
Let $\pi = (\pi_t, \pi_b)$ an element of a Rauzy class $\mathcal R$ on an alphabet $\A$, and let $T_0$ be a standard i.e.m.\
{\bf of restricted Roth type} with combinatorial data $\pi$.

We fix an integer $r\geq 2$. We will consider a smooth family $(T_t)$ through $T_0$ of generalized i.e.m.\ , acting on the same
interval $I=(u_0,u_d)$.
Our main theorem will describe the set of parameters for which $T_t$ is conjugated to $T_0$ by a $C^r$-diffeomorphism of the
closure of $I$ which is $C^r$-close to the identity.

We set
$$d^* = (2r+1)(g-1) +s .$$
Let $\ell$ be an integer $\geq 0$. The parameter $t$ runs in a neighborhood $V:= [-t_0,t_0]^{\ell +d^*}$ of $0$ in
$\Rset^{\ell +d^*}$. We write $t=(t',t'')$ with $t' \in [-t_0,t_0]^{\ell}$ and $t'' \in [-t_0,t_0]^{d^*}$. We also assume
that
\begin{itemize}
\item Each $T_t$ is a generalized i.e.m.\ (with the same combinatorial data than $T_0$) of class $C^{r+3}$ .
\item The map $t \mapsto T_t$ is of class $C^1$ in the following sense. Denote by $u^t_1(t)<\ldots <u^t_{d-1}(t)$
the singularities of $T_t$, by $u^b_1(t)<\ldots <u^b_{d-1}(t)$ those of $T_t^{-1}$. Then the functions $t \mapsto
u^t_i(t)$, $t \mapsto u^b_j(t)$ are of class $C^1$. Moreover, for each $\alpha \in \A$, each $0\leq i \leq {r+3}$, each
$1\leq i \leq \ell +d^*$ , the partial derivative $\partial _{t_j} \partial _x^i  T_t(x)$ should be defined on $\{(t,x);
\,t\in V,\, x\in \iat (t)\,\}$ and extend to a continuous function on the closure of this set (i.e including the
endpoints of $\iat (t)$).
\end{itemize}

As we look for g.i.e.m.'s which are $C^R$-conjugated to standard i.e.m.\ , it is certainly natural and necessary to assume that
the conjugacy invariant in $J^r$ of $T_t$ is trivial for all $t \in V$. We will actually need the stronger assumption
\begin{itemize}
\item For all $t \in V$, the conjugacy invariant of $T_t$ in $J^{r+3}$ is trivial.
\end{itemize}

Consider the derivative with respect to $t$ of $T_t$ at $t=0$. It can be viewed as a linear map $\Delta T$ from $\Rset^{\ell
+d^*}$
to $C^{r+3}(\sqcup \iat)$ (where we write $\iat$ instead of $\iat(0)$). Because the $J_{r+3}$ invariant is trivial for all
$t \in V$, it follows from Proposition 4.4 that any function $ \varphi$ in the image of
 $\Delta T$ satisfies
 $$\partial D^{\ell} \varphi =0, \quad \quad \forall 0 \leq \ell \leq r+3 .$$

 In particular, the image of $\Delta T$ is contained in the space $C^{r+1+BV}_{\partial}(\sqcup \iat)$ of subsection 3.4 and we
 can compose $\Delta T$ with the operator $\Pi: C_{\partial}^{r+1+BV}(\sqcup \iat) \rightarrow
 \Gamma_{\partial}(r+1)/\Gamma_T(r+1)$ of Theorem 3.13 to obtain a map
$\overline{\Delta T}: \Rset^{\ell +d^*} \rightarrow \Gamma_{\partial}(r+1) /\Gamma_T(r+1)$. Observe that, according to
Proposition 3.12, the dimension of $ \Gamma_{\partial}(r+1) /\Gamma_T(r+1)$ is  $g+r(2g-2)=d^* -s+1$. We will make the
following transversality assumption:\\
\begin{itemize}
\item
(Tr1)  The restriction of $\overline{\Delta T}$ to $\{0\}  \times \Rset^{d^*}$ is an homomorphism {\bf onto} $
\Gamma_{\partial}(r+1) /\Gamma_T(r+1)$.\\
\end{itemize}
After a linear change of variables in parameter space, we can and will also assume that $\Rset^{\ell} \times \{0\}$ is
contained in the kernel of $\overline{\Delta T}$.\\

When $s=1$, $d^*$ is equal to the dimension of $ \Gamma_{\partial}(r+1) /\Gamma_T(r+1)$; then (Tr1) means that the
restriction of $\overline{\Delta T}$ to $\{0\} \times  \Rset^{d^*}$ is an isomorphism onto $ \Gamma_{\partial}(r+1)
/\Gamma_T(r+1)$.\\

When $s>1$, we will ask for one more tranversality condition. Let $t \in {\rm Ker}\overline{\Delta T}, \varphi:= \Delta T(t)
\in C^{r+3}(\sqcup \iat)
\cap C_{\partial}^{r+1+BV}(\sqcup \iat)$. The image  $\Pi (\varphi)$ in $\Gamma_{\partial}(r+1) /\Gamma_T(r+1)$ is equal to
$0$.
On the other hand, let $\wh \psi \in C^{r+3}(\overline I)$ a function such that $\wh \psi (0) = \wh \psi (1) = 0$, and
$\wh \psi (u^t_i) = \frac{d}{d\tau} u^t_i(\tau t)_{|\tau=0}$,
$\wh \psi (u^b_j) = \frac{d}{d\tau} u^b_j(\tau t)_{|\tau=0}$ for all
$0<i,j<d$. Then $\varphi_1 := \varphi + \wh \psi -\wh \psi \circ T$ satisfies $\Pi(\varphi_1) =0$ and $ \varphi_1
(\upsilon)
=0$ for all $\upsilon \in \AA$. Writing $\varphi_1 = \psi_1 \circ T_0 - \psi_1$ and considering the values of $\psi_1$ on
the cycles of $\sigma$, we define as in Remark 3.3 an element $\nu(\varphi_1) \in \RS /\Rset$.  It is obvious that this
vector only depends on $t$ (not on the choice of $\wh \psi$) and we denote it by $\overline \nu(t)$.
We ask that\\
\begin{itemize}
\item
(Tr2) The restriction of $\overline \nu$ to the intersection of the kernel of $\overline{\Delta T}$ with  $\{0\} \times
\Rset^{d^*}$ is an isomorphism onto $\RS /\Rset$.\\
\end{itemize}
After a linear change of variables in parameter space, we can and will assume that $ \Rset^{\ell} \times \{0\}$ is equal to
the kernel of $\overline\nu$.

\subsection{Statement of the Theorem}

Under the hypotheses of the last subsection, we have

\begin{theorem}
There exists $t_1 \leq t_0$ and a neighborhood $W$ of the identity in ${\rm Diff}^r(\overline I)$ with the following
properties:
\begin{enumerate}
\item for every $t' \in [-t_1,t_1]^{\ell}$, there exists a unique $t''=:\theta(t') \in [-t_1,t_1]^{d^*}$ and a unique
    $h=:h_{t'} \in W$ such that, with $t=(t',t'')$
$$T_t= h \circ T_0 \circ h^{-1} \;;$$
\item the maps $t' \mapsto t''=\theta(t')$ and $t' \mapsto h_{t'}$ are of class $C^1$; moreover $\theta(0) =0$ and
    $D\theta _{|t'=0} =0$.
\end{enumerate}
\end{theorem}

The theorem thus states that, amongst $C^{r+3}$ g.i.e.m.'s  close to $T_0$ with trivial conjugacy invariant in $J^{r+3}$, those
which are conjugated to $T_0$ by a $C^r$ diffeomorphism close to the identity form a $C^1$ submanifold of codimension $d^* =
(g-1)(2r+1) +s$. The theorem also describes the tangent space to this submanifold at $T_0$, in terms of the cohomological
equation.

\medskip

As we look for a $C^r$-conjugacy to a standard i.e.m.\ , it is natural to restrict our attention to generalized i.e.m.'s with
trivial $C^r$-conjugacy invariant in $J^r$. It is unclear whether it is necessary to assume, as we do, that the $C^{r+3}$
conjugacy invariant is trivial. In the circle case ($d=2$), a linearization theorem   still holds if one only assumes that
the $C^{r+1}$-conjugacy invariant is trivial; the situation is unclear when only the $C^r$-conjugacy invariant is assumed to
be trivial.

\subsection{Simple families}

\begin{definition} We say that a  family $(T_t)$ as above is {\it simple} if $u^t_i(t)$ is, for all $0<i<d$, independent of
$t$, and if, for all $\alpha \in \A$ and all $t$, $T_t$ coincides with $T_0$ in the neighborhood of each endpoint of
$\iat$.
\end{definition}

The aim of this section is to show the

\begin{proposition} There exists $t_2<t_0$ and a $C^1$ family $(\wt h_t)_{t \in [-t_2,t_2]^{\ell +d^*}}$ in ${\rm
Diff}^{r+3}(\overline I)$
such that the family $(\wt T_t) := (\wt h_t^{-1} \circ T_t \circ \wt h_t )$ is simple and still satisfies the hypotheses of
the last section.
\end{proposition}

\begin{proof} Write $u^t_i$, $u^b_j$ for $u^t_i (0)$, $u^b_j (0)$. A first step is to choose a $C^1$ family $(\wh h_t)_{t
\in [-t_2,t_2]^{\ell +d^*}}$ in ${\rm Diff}^{\infty}(\overline I)$
such that $\wh h_t(u^t_i) = u^t_i(t),\; \wh h_t(u^b_j) = u^b_j(t)$ for all $t \in [-t_2,t_2]^{\ell +d^*}$. This is possible,
after taking
$t_2<t_0$ sufficiently small, since the $u^t_i, u^b_j$ are all distinct (as $T_0$ has no connection). Then, for the family
$(\wh T_t)
:= (\wh h_t^{-1} \circ T_t \circ \wh h_t )$, we have that the $u^t_i(t)$ and $u^b_j(t)$ are independent of $t$.

Next, for $\upsilon \in \AA$, $t \in [-t_2,t_2]^{\ell+d^*}$, we introduce the $(r+3)\,$-jet $j(\wh T_t, \upsilon)$ of
Subsection 4.1. For every cycle $C= \{\upsilon_0, \ldots, \upsilon_{\kappa}\}$ of $\sigma$, every $t \in [-t_2,t_2]^{\ell
+d^*}$, we have
$$J(\wh T_t, C):= j(\wh T_t, \upsilon_0)^{\va(\upsilon_0)} \ldots j(\wh T_t, \upsilon_{\kappa})^{\va(\upsilon_{\kappa})} = 1
\;.$$

We look now for a $C^1$ family $(\overline h_t)_{t \in [-t_2,t_2]^{\ell +d^*}}$ in ${\rm Diff}^{r+3}(\overline I)$ such
that:
\begin{enumerate}
\item $\overline h_t(u^t_i) = u^t_i, \; \overline h_t(u^b_j) = u^b_j\quad$, for all $t \in [-t_2,t_2]^{\ell +d^*}$,
    $0<i,j<d$;
\item $\wh T_t \circ \overline h_t(u^t(\upsilon) +x) = \overline h_t(u^b(\upsilon) +x)$, for all $\upsilon \in \AA$ of
    the form $(\alpha,L)$,
, $x>0$ small enough;

\item $\wh T_t \circ \overline h_t(u^t(\upsilon) -x) = \overline h_t(u^b(\upsilon) -x)$, for all $\upsilon \in \AA$ of
    the form $(\alpha,R)$,
 $x>0$ small enough.

\end{enumerate}
These conditions obviously imply that $\wt T_t:= \overline h_t \,^{-1} \circ \wh T_t \circ \overline h_t$ is simple (and
still satisfies the hypotheses of the last subsection). It is possible to solve (1)-(3) for $\overline h_t \in {\rm
Diff}^{r+3}(\overline I)$, since (2) and (3) connect values of $\overline h_t$ on different small intervals bounded by the
singularities. Compatibility conditions then occur on the product of $r$-jets along cycles of $\sigma$; they are fulfilled
as soon as the conjugacy invariant is trivial.
\end{proof}

According to Proposition 5.3, it is sufficient to prove Theorem 5.1 for simple families. This will be done in the next
section for $r \geq 3$, and in Section 7 for $r=2$.

\section{Proof: $C^r$-conjugacy, $r\geq 3$}

In this section, we assume that $r\geq 3$ and will prove the theorem in this case. The case $r=2$ will be dealt with in the
next section. Let therefore
$(T_t)$ be a $C^1$ family of $C^{r+3}$ g.i.e.m.'s satisfying the hypotheses of subsection 5.1. According to Proposition 5.3, we
can and will also assume that the family is simple.

Recall that the Schwarzian derivative of a $C^3$ orientation preserving diffeomorphism $f$ is defined by
$$ Sf:= D^2\Log Df - \frac12 (D\Log Df)^2\;.$$

The composition rule for Schwarzian derivatives is

$$S(f\circ g) = Sf \circ g \;(Dg)^2 + Sg\;.$$

\subsection{Smoothness of the composition map}

The tangent space at $id$ to ${\rm Diff}^r(\overline I)$ is the space  $C^r_{0,0}(\overline I)$ of $C^r$-functions on
$\overline I$ vanishing at $u_0$ and $u_d$.

\begin{lemma} The composition map
$$C^r(\overline I) \times {\rm Diff}^r(\overline I) \rightarrow C^{r-1}(\overline I)$$
$$(\varphi,h) \mapsto \varphi \circ h$$
is of class $C^1$. Its differential at $(0,id)$ is the map $(\delta \varphi, \delta h) \mapsto \delta \varphi$ from
$C^r(\overline I) \times C^r_{0,0}(\overline I)$ to $C^{r-1}(\overline I)$.
\end{lemma}

The map $(\varphi,h) \mapsto \varphi \circ h$ {\bf valued in} ${\mathbf C^r(\overline I)}$ is only continuous. It becomes $C^1$ when seen 
as taking its values in $C^{r-1}(\overline I)$. The formula for the derivative at $(0,id)$ is elementary.

We denote by $C^k_{comp}(\sqcup \iat)$ the space of functions $\varphi \in C^k(\sqcup \iat)$ which vanish in the
neighborhood of the endpoints of each $\iat$. Obviously, a map $\varphi \in C^k_{comp}(\sqcup \iat)$ satisfies $\partial
D^{\ell} \varphi =0$ for $0 \leq \ell \leq k$.

\begin{lemma}
 The map
$$ \Phi: [-t_0,t_0]^{\ell +d^*} \times {\rm Diff}^r(\overline I) \rightarrow C^{r-1}_{comp}(\sqcup \iat)$$
$$(t,h) \mapsto ST_t \circ h (Dh)^2$$
is of class $C^1$. Its differential at $(0,id)$ is the map $(\delta t, \delta h) \mapsto D^3 \delta \varphi$  from
$\Rset^{\ell +d^*} \times C^r_{0,0}(\overline I)$ to $C^{r-1}_{comp}(\sqcup \iat)$
, with $\delta\varphi = \Delta T(\delta t)$.
\end{lemma}

\begin{proof}
From the formula above for $ST_t$, the derivative of $t\mapsto ST_t$ at $t=0$ is $\delta t\mapsto D^3\Delta T(\delta t)$. The Lemma
then follows from Lemma 6.1 and an elementary computation.
\end{proof}

\subsection{The cohomological equation}

 We fix in the following a subspace $\Gamma_u$ in $\Gamma_{\partial}(r-2)$ such that
$$\Gamma_{\partial}(r-2) = \Gamma_T(r-2) \oplus \Gamma_u\ \oplus \Rset 1\;.$$
According to Proposition 3.12, we have
$$\dim \Gamma_u = (2r-5)(g-1).$$
From Theorem 3.13, there exist bounded linear operators $L_0: \;C^{r-1}_{\partial}(\sqcup \iat) \rightarrow C^{r-3}_0(I)$,
$L_1: \;C^{r-1}_{\partial}(\sqcup \iat) \rightarrow \Gamma_u$ such that, for $\varphi \in C^{r-1}_{\partial}(\sqcup \iat)$,
we have
$$\varphi = \int_I \varphi(x) dx + L_1(\varphi) + L_0(\varphi) \circ T_0 - L_0(\varphi)\;.$$
Here, $C^{r-3}_0(I)$ is the space of $C^{r-3}$ functions on $I$ which vanish at $u_0$.

\begin{lemma}
The map
$$\Psi: [-t_0,t_0]^{\ell +d^*} \times {\rm Diff}^r(\overline I) \rightarrow C^{r-3}_0(\overline I) $$
$$(t,h) \mapsto L_0(\Phi(t,h))$$
is of class $C^1$. Its differential at $(0,id)$ is the map $(\delta t, \delta h) \mapsto L_0(D^3 \delta \varphi)$  from
$\Rset^{\ell +d^*} \times C^r_{0,0}(\overline I)$ to $C^{r-3}_0(\overline I)$
, with $\delta\varphi = \Delta T(\delta t)$.
\end{lemma}

\begin{proof}
Indeed, $L_0$ is linear and the derivative of $\Phi$ has been computed in Lemma 6.2. 
\end{proof}

\subsection{Relation between a diffeomorphism and its Schwarzian derivative}

The next three lemmas present the Schwarzian derivative operator as a composition of a first-order quasilinear operator with 
a second-order non-linear differential operator which is sometimes called the nonlinearity operator.

\begin{lemma}
The map
$$ {\mathcal N}:{\rm Diff}^r(\overline I) \rightarrow C^{r-2}(\overline I)$$
$$h \mapsto D\Log Dh$$
is a $C^{\infty}$-diffeomorphism. Its differential at $id \in {\rm Diff}^r(\overline I)$ is the map $\delta h \mapsto
D^2\delta h$ from $C^r_{0,0}(\overline I)$  to $C^{r-2}(\overline I)$.
\end{lemma}

\begin{proof}
That $ {\mathcal N}$ is a  $C^{\infty}$ map and the formula for its differential at $id$ is elementary. Given 
$N\in C^{r-2}(\overline I)$, let $N_1\in C^{r-1}(\overline I)$ be the primitive of $N$ such that the mean value over $\overline I$
of $\exp (N_1)$ is $1$. Define then 
$$
h(x) = \int_{u_0}^x \exp( N_1(t))dt\, . 
$$
It is clear that $h\in {\rm Diff}^r(\overline I)$, that it is the only element of ${\rm Diff}^r(\overline I)$ such that 
${\mathcal N} (h)=N$ and that $N\mapsto h$ is $C^{\infty}$.
\end{proof}

\begin{lemma}
The map
$${\mathcal Q}: C^{r-2}(\overline I) \rightarrow C^{r-3}_0(\overline I) \times \Rset^2$$
$$ N \mapsto (\psi = DN-\frac12 N^2 -c_0, c_0 = DN(u_0) - \frac12 N^2(u_0), c_1= N(u_0)) $$
is of class $C^{\infty}$. Its differential at $0$ is given by
$$\delta \psi = D \delta N - \delta c_0, \delta c_0 = D\delta N(u_0), \delta c_1 = \delta N(u_0) \,,$$
which is an isomorphism from $C^{r-2}(\overline I)$ onto $C^{r-3}_0(\overline I) \times \Rset^2$.
Therefore, the restriction of ${\mathcal Q}$ to an appropriate neighborhood of $0 \in C^{r-2}(\overline I)$ is a
$C^1$-diffeomorphism onto a neighborhood of $(0,0,0) \in  C^{r-3}_0(\overline I) \times \Rset^2$.
\end{lemma}

\begin{proof}
The first two statements are obtained by elementary computation. The last one is a consequence of the inverse function
theorem in Banach spaces. 
\end{proof}

\begin{lemma}
The map
$$ {\mathcal S}:= {\mathcal Q} \circ {\mathcal N}:{\rm Diff}^r(\overline I) \rightarrow C^{r-3}_0(\overline I) \times
\Rset^2$$
$$ h \mapsto (\psi, c_0,c_1)=(Sh-Sh(u_0), Sh(u_0), D\Log Dh(u_0))$$
is of  class $C^{\infty}$, and its restriction to an appropriate neighborhood of $id \in {\rm Diff}^r(\overline I)$ is a
$C^{\infty}$-diffeomorphism
onto a neighborhood of $(0,0,0) \in C^{r-3}_0(\overline I) \times \Rset^2$. The differential of ${\mathcal S}$ at $id \in
{\rm Diff}^r(\overline I)$ is the isomorphism
$\delta h \mapsto (D^3\delta h - D^3\delta h(u_0), D^3\delta h(u_0), D^2\delta h(u_0))$ from $C^r_{0,0}(\overline I)$ to
$C^{r-3}_0(\overline I) \times \Rset^2$.
\end{lemma}

\begin{proof}
This is a direct consequence of the last two lemmas.
\end{proof}

We denote by $W_0$, $W_1$ neighborhoods of $id $ in ${\rm Diff}^r(\overline I)$ and of $(0,0,0) $ in $C^{r-3}_0(\overline I)
\times \Rset^2$ respectively such that
${\mathcal S}$ defines a $C^{\infty}$-diffeomorphism from $W_0$ onto $W_1$. We denote by ${\mathcal P}: W_1 \rightarrow W_0$
the inverse diffeomorphism, and by $P$ the differential of ${\mathcal P}$ at $(0,0,0) $.

\subsection{The fixed point theorem}

\begin{lemma}
The map
$$(t,h,c_0,c_1) \mapsto {\mathcal P}(\Psi(t,h), c_0,c_1)$$
is defined and of class $C^1$ in a neighborhood of $(0,id,0,0)$ in $[-t_0,t_0]^{\ell +d^*} \times {\rm Diff}^r(\overline
I)\times \Rset^2$, with values in $W_0$. Its differential
at $(0,id,0,0)$ is the map $(\delta t, \delta h, \delta c_0, \delta c_1) \mapsto P(L_0(D^3 \delta \varphi),\delta c_0,
\delta c_1)$, with  $\delta\varphi = \Delta T(\delta t)$, from $\Rset ^{\ell +d^*} \times C^r_{0,0}(\overline I) \times \Rset^2$ to
$C^r_{0,0}(\overline I)$.
\end{lemma}

\begin{proof}
This is a consequence of lemmas 6.3 and 6.6.
\end{proof}

\begin{lemma}
There exist an open neighborhood $W_2$ of $id \in {\rm Diff}^r(\overline I)$ and an open neighborhood $W_3$ of $(0,0,0) \in
[-t_0,t_0]^{\ell +d^*} \times \Rset^2$ such that, for each $(t,c_0,c_1) \in W_3$, the map
$$h \mapsto {\mathcal P}(\Psi(t,h), c_0,c_1)$$
has exactly one fixed point in $W_2$, that we denote by ${\mathcal H}(t,c_0,c_1)$. Moreover, the map ${\mathcal H}$ is of
class $C^1$ on $W_3$, and its differential at $(0,0,0)$ is the map $(\delta t, \delta c_0, \delta c_1) \mapsto P(L_0(D^3
\delta \varphi),\delta c_0, \delta c_1)$, with  $\delta\varphi = \Delta T(\delta t)$, from $\Rset ^{\ell +d^*} \times \Rset^2$ to
$C^r_{0,0}(\overline I)$.
\end{lemma}

\begin{proof}
This is a consequence of the implicit function theorem applied to the fixed point equation $ {\mathcal P}(\Psi(t,h), c_0,c_1)=h$.
\end{proof}

Let $(t,c_0,c_1) \in W_3$, $h= {\mathcal H}(t,c_0,c_1)$. Then $h$ satisfies
$$\Phi(t,h)=ST_t \circ h (Dh)^2 = L_1(\Phi(t,h)) + \int_0^1 \Phi(t,h)(x)\;dx\;\; + Sh \circ T_0 - Sh \;.$$

For $(t,c_0,c_1) \in W_3$, we write $H:= T_t \circ h \circ T_0^{-1}$. We have $H=h$ iff $T_t = h \circ T_0 \circ h^{-1}$.

\subsection{Conditions for H to be a diffeomorphism }
\begin{lemma}
For $(t,c_0,c_1) \in W_3$, the following are equivalent
\begin{enumerate}
\item $h(u^t_i) = u^t_i$ for all $0<i<d$;
\item $H$ is an homeomorphism of $\overline I$ satisfying $H(u^b_j) = u^b_j$ for all $0<j<d$.
\end{enumerate}
\end{lemma}

\begin{proof}
The map $H$ is $1$-to-$1$ (mod.\ $0$), fixes $u_0$ and $u_d$ and is continuous except perhaps at the $u_j^b$, $0<j<d$.
Looking at the left and right limits of $H$ at these points gives the lemma.
\end{proof}

When the equivalent conditions of the lemma are satisfied, $H$ is in fact a piecewise $C^r$ diffeomorphism of $\overline
I$,
with possibly discontinuities of the  derivatives of order $\leq 2$ at the $u^b_j$.

\begin{lemma}
Let $(t,c_0,c_1) \in W_3$ such that the equivalent conditions of  the last Lemma  are satisfied. Then $H$ is a
$C^2$-diffeomorphism of $\overline I$ iff one has, for all $\upsilon = (\alpha,L) \in \AA$ with $\alpha \ne \,_b\alpha$
$$\Log Dh(u^t(\upsilon)) = \Log Dh(u^t(\sigma(\upsilon))),\quad D\Log Dh(u^t(\upsilon)) = D\Log
Dh(u^t(\sigma(\upsilon)))\;.$$
\end{lemma}

\begin{proof}
Indeed these relations express that the left and right limits of the first two derivatives of $H$ at the $u_j^b$ are the same. 
\end{proof}

\begin{remark}
When $\upsilon = (\alpha,R)$, $\alpha \ne \alpha_t$, one has $u^t(\upsilon)=u^t(\sigma(\upsilon))$, hence
$$\Log Dh(u^t(\upsilon)) = \Log Dh(u^t(\sigma(\upsilon))),\quad D\Log Dh(u^t(\upsilon)) = D\Log
Dh(u^t(\sigma(\upsilon)))\;,$$
is always true.
\end{remark}

\begin{lemma}
Let $(t,c_0,c_1) \in W_3$ such that the equivalent conditions of the last two Lemmas are satisfied. Assume also that
$L_1(\Phi(t,{\mathcal H}(t,c_0,c_1))) =0$. Then $H$ is a $C^r$-diffeomorphism of $\overline I$.
\end{lemma}
\begin{proof}
We have to prove that the derivative of order $3+k$ of $H$ is continuous at each $u^b_j$ for all $0\leq k \leq r-3$,
$0<j<d$.
 This is equivalent to show that, for all $0\leq k \leq r-3$,  all $\upsilon = (\alpha,L) \in \AA$ with $\alpha \ne
 \,_b\alpha$
$$D^k Sh (u^t(\upsilon))= D^k Sh (u^t(\sigma(\upsilon)))\;,$$
with $h= {\mathcal H}(t,c_0,c_1)$ as above.

As $L_1(\Phi(t,{\mathcal H}(t,c_0,c_1))) =0$, we have
$$ST_t \circ h (Dh)^2 =  \int_0^1 \Phi(t,h)(x)\;dx\;\; + Sh \circ T_0 - Sh \;,$$
and, for $0<k\leq r-3$
$$D^k(ST_t \circ h (Dh)^2) =  D^kSh \circ T_0 - D^kSh \;.$$
As $D^kST_t$ vanishes at the $u^t_i$ for $0\leq i \leq d$, $0\leq k \leq r-3$,  and $D^k Sh$ is continuous at $u^b_j$,  the
required equalities follow.
\end{proof}

\subsection {Equations for the conjugacy class of $T_0$}

\begin{proposition}
Let $(t,c_0,c_1) \in W_3$ such that $h = {\mathcal H}(t,c_0,c_1)$ satisfies
\begin{eqnarray*}
h(u^t_i) &=&u^t_i  \quad \quad{\rm for\; all}\; 0<i<d,\\
\Log Dh( u^t(\upsilon)) &=& \Log Dh(u^t(\sigma(\upsilon))),\quad {\rm for \;all}\; \upsilon = (\alpha, L) \in \AA, \alpha
\ne \, _b\alpha, \\
D\Log Dh( u^t(\upsilon)) &= &D\Log Dh(u^t(\sigma(\upsilon))), \;{\rm for\; all}\; \upsilon = (\alpha, L) \in \AA, \alpha \ne
\, _b\alpha, \\
L_1(\Phi(t,h)) &=&0\;;
\end{eqnarray*}
Then , if $(t,c_0,c_1)$ is close enough to $(0,0,0)$, we have $T_t \circ h = h \circ T_0$.

Conversely, let $t \in [-t_0,t_0]^{\ell +d^*}$ and $ h \in {\rm Diff}^r(\overline I)$ such that $T_t \circ h = h \circ
T_0$.
Let $c_0 = S h(0)$, $c_1= D\Log Dh(0)$. If $t$ is close enough to $0$ and $ h$ is close enough to the identity, then  $ h =
{\mathcal H}(t,c_0,c_1)$
and the relations   above are satisfied.
\end{proposition}
\begin{proof}
We first prove the second part of the proposition. Let $t \in [-t_0,t_0]^{\ell +d^*}$ close to $(0,0,0)$, $ h \in {\rm
Diff}^r(\overline I)$ close to
the identity, such that $T_t \circ h = h \circ T_0$. Then we have
$$ST_t \circ h (Dh)^2 = Sh \circ T_0 - Sh\;.$$
Let $c_0 = S h(0)$, $c_1= D\Log Dh(0)$. Then we have $h= {\mathcal P}(\Psi(t,h), c_0,c_1)$ and therefore
$h = {\mathcal H}(t,c_0,c_1)$. Moreover $L_1(\Phi(t,{\mathcal H}(t,c_0,c_1))) =0$ holds. Finally, $H:= T_t \circ h \circ
T_0^{-1}$ is equal to $h$, hence it follows from Lemmas 6.9 and 6.10 that the other relations in the proposition are
satisfied. This concludes the proof of the second part of the proposition.

 For the proof of the first part, the argument is slightly different, depending whether $(\alpha_t,R)$ and $(\,_b\alpha,L)$
 belong or not to the same cycle of $\sigma$ in $\AA$. Let $t,c_0,c_1,h$ as in the proposition. From Lemma 6.12, we already
 know that $h$ and $H$ belong to ${\rm Diff}^r(\overline I)$.
\begin{itemize}
\item We first assume that $(\alpha_t,R)$ and $(\,_b\alpha,L)$ belong to the same cycle of $\sigma$.\\

By assumption (and Remark 6.11), we have $\Log Dh( u^t(\upsilon)) = \Log Dh(u^t(\sigma(\upsilon)))$ for all $\upsilon
\in \AA$ except $(\alpha_t, R)$ and
$(\,_b\alpha,L)$. In particular, this gives
$$\Log Dh(u_0) = \Log Dh (u^t(\,_t\alpha,L))=\Log Dh (u^t(\alpha_t,R)) = \Log Dh(u_d),$$
$$\Log DH(u_0) = \Log Dh (u^t(\,_b\alpha,L))=\Log Dh (u^t(\alpha_b,R)) = \Log DH(u_d).$$
The same argument applies to $D\Log Dh$ and to $D^k Sh$ for $0\leq k \leq r-3$, according to the proof of Lemma 6.12.
This allows to conclude that both $h$ and $H$ are induced by $C^r$-diffeomorphisms of the circle $\Tset$ obtained by
identifying the endpoints of $\overline I$. Moreover, the relation $ST_t \circ h (Dh)^2 = c + Sh \circ T_0 - Sh $
implies $SH = Sh +c$, with $c = \int_0^1 \Phi(t,h)(x)\;dx\;$. The following lemma allows to conclude that $h = H$.

\begin{lemma}
Let ${\rm Diff}^r( \Tset,0)$ the group of orientation preserving $C^r$ diffeomorphisms of the circle fixing $0$, and let
$C^{r-3}_0(\Tset)$ be the space of $C^{r-3}$ functions on the circle vanishing at $0$. The map
$$ {\rm Diff}^r( \Tset,0) \rightarrow C^{r-3}_0(\Tset)$$
$$h \mapsto Sh - Sh(0)$$
is of class $C^{\infty}$ and its restriction to an appropriate neighborhood of the identity is a $C^{\infty}$
diffeomorphism onto a neighborhood of $0$ in $C^{r-3}_0(\Tset)$.
\end{lemma}
\begin{proof}
The first assertion is trivial, the differential at the identity being the map $\delta h \mapsto D^3 \delta h - D^3
\delta h(0)$ from $C^r_{0}(\Tset)$ (the space of $C^r$ functions on the circle vanishing at $0$) to $C^{r-3}_0(\Tset)$.
This is clearly an isomorphism, hence the lemma follows by the implicit function theorem.
\end{proof}

\item We now assume that $(\alpha_t,R)$ and $(\,_b\alpha,L)$ do not belong to the same cycle of $\sigma$.\\

By assumption (and Remark 6.11), we still have $\Log Dh( u^t(\upsilon)) = \Log Dh(u^t(\sigma(\upsilon)))$ for all
$\upsilon \in \AA$ except $(\alpha_t, R)$ and
$(\,_b\alpha,L)$. This now gives
$$\Log Dh(u_0) = \Log Dh (u^t(\,_t\alpha,L))= \Log Dh (u^t(\,_b\alpha,L))= \Log DH(u_0),$$
$$\Log Dh(u_d)= \Log Dh (u^t(\alpha_t,R)) =\Log Dh (u^t(\alpha_b,R)) = \Log DH(u_d).$$
The same argument applies to $D\Log Dh$ and to $D^k Sh$ for $0\leq k \leq r-3$, according to the proof of Lemma 6.12.
This allows to conclude that the $r$-jets at $u_d$ of $h$ and $H$ are the same. Moreover, the relation $ST_t \circ h
(Dh)^2 = c + Sh \circ T_0 - Sh $ implies $SH = Sh +c$, with $c = \int_0^1 \Phi(t,h)(x)\;dx\;$. As $SH(u_d) =Sh(u_d)$, we
must have $c=0$. As $Sh =SH$ and the $3$-jets of $h$ and $H$ at $u_d$ are equal, we conclude also in this case that
$h=H$.
\end{itemize}
\end{proof}

\subsection{End of the proof of Theorem 5.1 for $r\geq 3$}

From the proposition above, we have to determine in a neighborhood of $0 \in [-t_0,t_0]^{\ell +d^*}$ the set of $t$ for
which, for some $(c_0,c_1)$ close to $(0,0)$, the diffeomorphism  $h = {\mathcal H}(t,c_0,c_1)$ satisfies

\begin{equation}
h(u^t_i) =u^t_i  \quad \quad{\rm for\; all}\; 0<i<d,
\end{equation}
\begin{equation}
\Log Dh( u^t(\upsilon)) = \Log Dh(u^t(\sigma(\upsilon))),\quad {\rm for \;all}\; \upsilon = (\alpha, L) \in \AA, \alpha \ne
\, _b\alpha,
\end{equation}
\begin{equation}
D\Log Dh( u^t(\upsilon)) = D\Log Dh(u^t(\sigma(\upsilon))), \;{\rm for\; all}\; \upsilon = (\alpha, L) \in \AA, \alpha \ne
\, _b\alpha,
\end{equation}
\begin{equation}
L_1(\Phi(t,h)) =0\;.
\end{equation}

We will see that there are exactly $(d^* +2)$ independent equations for $t,c_0,c_1$ in the system above. Looking at the
linearized system at $(0,0,0)$ will allow to apply the implicit function theorem and conclude. We  deal separately with the
same two cases which appeared in the proof of  Proposition 6.13.\\

\begin{itemize}
\item We first assume that $(\alpha_t,R)$ and $(\,_b\alpha,L)$ belong to the same cycle of $\sigma$.\\

There are $(d-1)$ equations in (6.1), $(2r-5)(g-1)$ equations in (6.4) (the dimension of $\Gamma_u$). In (6.2), for each
cycle of $\sigma$ which does not contain $(\,_b\alpha,L)$, there is one redundant equation. So the number of equations
in (6.2) is really $(d-1)-(s-1)=(2g-1)$. Similarly, there are
$(2g-1)$ equations in (6.3).

Therefore the total number of equations in the system (6.1)-(6.4) is $(d-1) +(2r-5)(g-1) +(2g-1) +(2g-1) = d^* +2$ as
claimed.

Consider now the linearized system obtained from (6.1)-(6.4) at $(0,0,0)$. Writing as before $\delta \varphi = \Delta
T(t)$, we have, from Lemma 6.8
$$\delta h = P(L_0(D^3 \delta \varphi),\delta c_0, \delta c_1).$$

From the definition of $L_0$ and $P$ (cf.Lemmas 6.6), this is equivalent to
$$D^3\delta \varphi = D^3 \delta h \circ T_0 - D^3 \delta h + L_1(D^3 \delta \varphi),$$
$$ D^3 \delta h (0) = \delta c_0, \quad D^2 \delta h (0) = \delta c_1 \,,$$
where we have used in the first equation that $\int_0^1 D^3\delta \varphi(x)\, dx =0$.

Now, the linearized version of equation (6.4) is
\begin{equation}
L_1(D^3 \delta \varphi)=0.
\end{equation}

If this holds, we have
$$D^3\delta \varphi = D^3 \delta h \circ T_0 - D^3 \delta h$$
and then, by integration
$$D^2\delta \varphi = D^2 \delta h \circ T_0 - D^2 \delta h + \chi_2\,,$$
for some $\chi_2 \in \Gamma(1)$. But the linearized version of (6.3) is
\begin{equation}
 D^2\delta h( u^t(\upsilon)) = D^2 \delta h(u^t(\sigma(\upsilon))),\quad {\rm for \;all}\; \upsilon = (\alpha, L) \in
 \AA, \alpha \ne \, _b\alpha.
\end{equation}
If this holds, $\chi_2$ has to be constant (at each $u^b_j$, the left and right values of $\chi_2$ are the same). As
$\int_0^1 D^2\delta \varphi(x)\, dx =0$, we must have $\chi_2 =0$. Observe that the equation $D^2\delta \varphi = D^2
\delta h \circ T_0 - D^2 \delta h$ determines $\delta c_0$. One more integration then gives
$$D\delta \varphi = D \delta h \circ T_0 - D \delta h + \chi_1\,,$$
for some $\chi_1 \in \Gamma(1)$. Using now the linearized version of (6.2)
\begin{equation}
 D\delta h( u^t(\upsilon)) = D \delta h(u^t(\sigma(\upsilon))),\quad {\rm for \;all}\; \upsilon = (\alpha, L) \in \AA,
 \alpha \ne \, _b\alpha,
\end{equation}
we proceed in the same way to conclude that, if (6.5)-(6.7) holds, one has $\chi_1 =0$, $D\delta \varphi = D \delta h
\circ T_0 - D \delta h$
and $\delta c_1$ is determined. One last integration gives

$$\delta \varphi =  \delta h \circ T_0 -  \delta h + \chi_0\,,$$
for some $\chi_0 \in \Gamma(1)$. The linearized version of (6.1) is
\begin{equation}
\delta h (u^t_i) =0\, .
\end{equation}
If this holds, one obtains as above first that $\chi_0$ is constant; as $\delta h (u_d) =0 = \delta h
(u^t(\alpha_b,R))$, we have $\chi_0 =0$.

Recalling the definition of $\nu$ in Remark 3.3 and $\Pi$ in Theorem 3.13, we conclude that, if (6.5)-(6.8) holds, then
$$\Pi(\delta \varphi) =0, \quad \nu(\delta \varphi) =0.$$
Going backwards, we see that these relations are actually equivalent to (6.5)-(6.8).
In view of the transversality hypotheses (Tr1), (Tr2) of Subsection 5.1, the theorem in this case now follows from the
implicit function theorem.\\

\item We now assume that $(\alpha_t,R)$ and $(\,_b\alpha,L)$ do not belong to the same cycle of $\sigma$.\\

There are still $(d-1)$ equations in (6.1), $(2r-5)(g-1)$ equations in (6.4). In (6.2), for each cycle of $\sigma$ which
contains neither $(\,_b\alpha,L)$ nor $(\alpha_t,R)$, there is one redundant equation.
This would give $(d-1)-(s-2)=2g$ for the number of equations in (6.2) , and similarly in (6.3), leading to a grand total
of $(d^* +4)$ equations. However, we will now see that the equations
\begin{equation}
\Log Dh(u^t(\,_t\alpha ,L))= \Log Dh(u^t(\sigma(\,_t\alpha ,L))),
\end{equation}
\begin{equation}
D\Log Dh(u^t(\,_t\alpha ,L))= D\Log Dh(u^t(\sigma(\,_t\alpha ,L))),
\end{equation}

are also redundant.

Indeed, assume that (6.1)-(6.4) holds, with the exception of (6.9)-(6.10). Let $H=T_t \circ h \circ T_0^{-1}$ as above.
Following Lemmas 6.9, 6.10, 6.12, $H$ is a homeomorphism of $\overline I$ fixing $u_0, u_d$ and each $u^b_j$; moreover
the restrictions of $H$ to $[u_0,u^b(\,_t\alpha ,L)]$ and
$[u^b(\,_t\alpha ,L),u_d]$ are $C^r$-diffeomorphisms.

As in the end of the proof of Proposition 6.13, we obtain that the $r$-jets of $h$ and $H$ at $u_d$ are the same, and
that $Sh = SH$. This implies that $h=H$ on $[u^b(\,_t\alpha ,L),u_d]$.

Comparing the $r$-jets of $h$ and $H$ at $u^b(\,_t\alpha ,L)$ shows that the $r$-jets of $h$ at $u_0$ and
$u^b(\,_t\alpha ,L)$ are the same.

It is also true that the $r$-jets of $H$ at $u_0$
and $u^b(\,_t\alpha ,L)$ are the same, or, equivalently, that the $r$-jets of $h$ at $u^t(\,_b\alpha,L)$ and
$u^t(\sigma(\,_t\alpha ,L))$ are the same:
 for the first two derivatives, it follows from (6.2), (6.3) (without using (6.9)-(6.10)); for the higher derivatives,
 the argument is the same that in Lemma 6.12.

Therefore, the restrictions to $[u_0,u^b(\,_t\alpha ,L)]$ of both $h$ and $H$ satisfy periodic boundary conditions, and
we conclude by Lemma 6.14 that $h=H$ on the full interval $\overline I$.

This proves that (6.9), (6.10) are redundant and we are left with $(d^*+2)$ equations as in the first case.

We now consider the linearized system as in the first case. We still write
$$\delta \varphi = \Delta T(t),$$
$$\delta h = P(L_0(D^3 \delta \varphi),\delta c_0, \delta c_1),$$
hence
$$D^3\delta \varphi = D^3 \delta h \circ T_0 - D^3 \delta h + L_1(D^3 \delta \varphi),$$
$$ D^3 \delta h (0) = \delta c_0, \quad D^2 \delta h (0) = \delta c_1 \,.$$
Assuming (6.5), we get
$$D^3\delta \varphi = D^3 \delta h \circ T_0 - D^3 \delta h$$
and then, by integration
$$D^2\delta \varphi = D^2 \delta h \circ T_0 - D^2 \delta h + \chi_2\,,$$
for some $\chi_2 \in \Gamma(1)$.
The linearized version of (6.3) minus (6.10) is
\begin{equation}
 D^2\delta h( u^t(\upsilon)) = D^2 \delta h(u^t(\sigma(\upsilon))),\quad {\rm for \;all}\; \upsilon = (\alpha, L) \in
 \AA, \alpha \ne \, _b\alpha, \,_t\alpha.
\end{equation}
This implies that $\chi_2 \circ T_0^{-1}$ is constant on $(u_0,u^b(\,_t\alpha ,L))$ and on $(u^b(\,_t\alpha ,L),u_d)$.
Moreover, as
$D^2 \delta h (u^t(\alpha_b,R)) = D^2 \delta h (u_d)$, the value of $\chi_2$ on $(u^b(\,_t\alpha ,L),u_d)$ is $0$. But
we have also $\int_I \chi_2 = \int_I D^2 \delta \varphi =0$, hence $\chi_2 =0$ everywhere.

One more integration then gives
$$D\delta \varphi = D \delta h \circ T_0 - D \delta h + \chi_1\,,$$
for some $\chi_1 \in \Gamma(1)$.

The linearized version of (6.2) minus (6.9) is
\begin{equation}
 D\delta h( u^t(\upsilon)) = D \delta h(u^t(\sigma(\upsilon))),\quad {\rm for \;all}\; \upsilon = (\alpha, L) \in \AA,
 \alpha \ne \, _b\alpha, \,_t\alpha.
\end{equation}
This now implies $\chi_1 =0$. One last integration gives

$$\delta \varphi =  \delta h \circ T_0 -  \delta h + \chi_0\,,$$
for some $\chi_0 \in \Gamma(1)$. If we finally assume (6.8), we get as in the first case
$$\Pi(\delta \varphi) =0, \quad \nu(\delta \varphi) =0.$$

Going backwards, we see that these relations are actually equivalent to the conjunction of (6.5),(6.8),(6.11),(6.12).
Therefore we conclude as in the first case by the implicit function theorem.

\end{itemize}

\bigskip

The proof of the theorem for $r\geq 3$ is now complete. \hspace{1cm} $\Box$

\section{Proof: $C^2$-conjugacy}

In this section, we prove the theorem in the case $r=2$. It may help the reader to look first at  Appendix B.2, where the
main idea is presented in the simpler setting of circle diffeomorphisms. Let therefore
$(T_t)$ be a $C^1$ family of $C^{5}$ g.i.e.m.'s satisfying the hypotheses of subsection 5.1 with $r=2$. According to
Proposition 5.3, we can and will also assume that the family is simple.

In this section, we state the intermediate steps in the proof of the conjugacy theorem with $r=2$, but only comment on the parts
which substantially differ from the case $r\ge 3$ which was presented in Section 6.

\subsection{Smoothness of the composition map}

This subsection is identical with subsection 6.1.
The tangent space at $id$ to ${\rm Diff}^2(\overline I)$ is the space  $C^2_{0,0}(\overline I)$ of $C^2$-functions on
$\overline I$ vanishing at $u_0$ and $u_d$.

\begin{lemma} The composition map
$$C^2(\overline I) \times {\rm Diff}^2(\overline I) \rightarrow C^{1}(\overline I)$$
$$(\varphi,h) \mapsto \varphi \circ h$$
is of class $C^1$. Its differential at $(0,id)$ is the map $(\delta \varphi, \delta h) \mapsto \delta \varphi$ from
$C^2(\overline I) \times C^2_{0,0}(\overline I)$ to $C^{1}(\overline I)$.
\end{lemma}

\begin{lemma}
 The map
$$ \Phi: [-t_0,t_0]^{\ell +d^*} \times {\rm Diff}^2(\overline I) \rightarrow C^{1}_{comp}(\sqcup \iat)$$
$$(t,h) \mapsto ST_t \circ h (Dh)^2$$
is of class $C^1$. Its differential at $(0,id)$ is the map $(\delta t, \delta h) \mapsto D^3 \delta \varphi$  from
$\Rset^{\ell +d^*} \times C^2_{0,0}(\overline I)$ to $C^{1}_{comp}(\sqcup \iat)$
, with $\delta\varphi = \Delta T(\delta t)$.
\end{lemma}

\subsection{The cohomological equation}

We denote by $P^*$ the operator
$$\varphi \mapsto \int_{u_0}^x (\varphi(y) - \int_I \varphi) \,dy$$
from $C^{1}_{comp}(\sqcup \iat)$ to the space $C^2_c (\sqcup \iat)$ of functions in $ C^2 (\sqcup \iat)$ which are
continuous on $\overline I $ and vanish at $u_0$ and $u_d$.

We choose subspaces $\Gamma_c \subset \Gamma, \;\Gamma_u \subset \Gamma_{\partial}$ such that
$$\Gamma = \Gamma_c \oplus \Gamma_{\partial},\quad \Gamma_{\partial} = \Gamma_u \oplus \Gamma_s\;.$$

From Theorem 3.10, there exist bounded operators
$$L_c: C^2_c (\sqcup \iat) \rightarrow \Gamma_c,\quad L_u: C^2_c (\sqcup \iat) \rightarrow \Gamma_u,\quad L_0: C^2_c (\sqcup
\iat) \rightarrow C^0_0(\overline I)$$
such that, for $\varphi \in C^2_c (\sqcup \iat) $
$$\varphi + L_c(\varphi) + L_u(\varphi) = L_0(\varphi) \circ T_0 - L_0(\varphi)\;.$$

Here, $C^0_0(\overline I)$ denotes the space of continuous functions on $\overline I$ which vanish at $u_0$.

We write $L$ for the bounded operator from $\Gamma_s$ to $C^0_0(\overline I)$ such that $v = L(v)\circ T_0 -L(v)$.

\begin{lemma}
The map $\Psi_1: (t,h,v) \mapsto L(v)+L_0(P^*(\Phi(t,h)))$ from $[-t_0,t_0]^{\ell +d^*} \times {\rm Diff}^2(\overline I)
\times \Gamma_s$ to $C^0_0(\overline I)$ is of class $C^1$. Its differential at $(0,{\rm id}, v)$ is
$$(\delta t, \delta h, \delta v) \mapsto L(\delta v) + L_0(D^2 \delta \varphi).$$
\end{lemma}

\subsection{Relation between a diffeomorphism and the primitive of its Schwarzian derivative}

\begin{lemma}
The map
$${\mathcal Q}_1: C^{0}(\overline I) \rightarrow C^{0}_0(\overline I) \times \Rset$$
$$ N \mapsto (\psi_1(x) = N(x)-N(u_0)-\frac12 \int_{u_0}^x N^2(y)\,dy, c_1 = N(u_0)) $$
is of class $C^{\infty}$. Its differential at $0$ is given by
$$\delta \psi_1 = D \delta N - \delta c_1, \delta c_1 = \delta N(u_0) \,,$$
which is an isomorphism from $C^{0}(\overline I)$ onto $C^{0}_0(\overline I) \times \Rset$.
Therefore, the restriction of ${\mathcal Q}_1$ to an appropriate neighborhood of $0 \in C^{0}(\overline I)$ is a
$C^{\infty}$-diffeomorphism onto a neighborhood of $(0,0) \in  C^{0}_0(\overline I) \times \Rset$.
\end{lemma}

Combining the last Lemma with Lemma 6.4, which is still valid for $r=2$, we obtain

\begin{lemma}
The map
$$ {\mathcal S}_1:= {\mathcal Q}_1 \circ {\mathcal N}:{\rm Diff}^2(\overline I) \rightarrow C^{0}_0(\overline I) \times
\Rset$$
$$ h \mapsto (\psi_1,c_1)=(D\Log Dh-D\Log Dh(u_0) -\frac 12 \int_{u_0} (D\Log Dh)^2 ,  D\Log Dh(u_0))$$
is of  class $C^{\infty}$, and its restriction to an appropriate neighborhood of $id \in {\rm Diff}^2(\overline I)$ is a
$C^{\infty}$-diffeomorphism
onto a neighborhood of $(0,0,0) \in C^{0}_0(\overline I) \times \Rset$. The differential of ${\mathcal S}_1$ at $id \in {\rm
Diff}^2(\overline I)$ is the isomorphism
$\delta h \mapsto (D^2\delta h - D^2\delta h(u_0), D^2\delta h(u_0))$ from $C^2_{0,0}(\overline I)$ to $C^{0}_0(\overline I)
\times \Rset$.
\end{lemma}

We denote by $W_0$, $W_1$ neighborhoods of $id $ in ${\rm Diff}^2(\overline I)$ and of $(0,0) $ in $C^{0}_0(\overline I)
\times \Rset$ respectively such that
${\mathcal S}_1$ defines a $C^{\infty}$-diffeomorphism from $W_0$ onto $W_1$. We denote by ${\mathcal P}_1: W_1 \rightarrow
W_0$ the inverse diffeomorphism, and by $P_1$ the differential of ${\mathcal P}_1$ at $(0,0) $.

\subsection{The fixed point theorem}

\begin{lemma}
The map
$$(t,h,v,c_1) \mapsto {\mathcal P}_1(\Psi_1(t,h,v),c_1)$$
is defined and of class $C^1$ in a neighborhood of $(0,id,0,0)$ in $[-t_0,t_0]^{\ell +d^*} \times {\rm Diff}^2(\overline
I)\times \Gamma_s \times \Rset$, with values in $W_0$. Its differential
at $(0,id,0,0)$ is the map $(\delta t, \delta h, \delta v, \delta c_1) \mapsto P_1(L_0(D^2 \delta \varphi)+L(v), \delta
c_1)$, with  $\delta\varphi = \Delta T(\delta t)$, from $\Rset ^{\ell +d^*} \times C^2_{0,0}(\overline I) \times \Gamma_s \times
\Rset$ to $C^2_{0,0}(\overline I)$.
\end{lemma}

\begin{lemma}
There exist an open neighborhood $W_2$ of $id \in {\rm Diff}^2(\overline I)$ and an open neighborhood $W_3$ of $(0,0,0) \in
[-t_0,t_0]^{\ell +d^*} \times \Gamma_s \times \Rset$ such that, for each $(t,v,c_1) \in W_3$, the map
$$h \mapsto {\mathcal P}_1(\Psi_1(t,h,v), c_1)$$
has exactly one fixed point in $W_2$, that we denote by ${\mathcal H}(t,v,c_1)$. Moreover, the map ${\mathcal H}$ is of
class $C^1$ on $W_3$, and its differential at $(0,0,0)$ is the map $(\delta t, \delta v, \delta c_1) \mapsto P_1(L_0(D^2
\delta \varphi)+L(v), \delta c_1)$, with  $\delta\varphi = \Delta T(\delta t)$, from $\Rset ^{\ell +d^*} \times \Gamma_s \times
\Rset$ to $C^2_{0,0}(\overline I)$.
\end{lemma}

Let $(t,v,c_1) \in W_3$, $h= {\mathcal H}(t,v,c_1)$. Then $h$ satisfies
$$P^*(\Phi(t,h)) +L_c(P^*(\Phi(t,h))) +L_u(P^*(\Phi(t,h)))+v  = N_1 h \circ T_0 - N_1 h  \;,$$
with
$$N_1 h(x) = D\Log Dh(x) -\frac 12 \int_{u_0}^x (D\Log Dh(y))^2 \,dy \;.$$

For $(t,v,c_1) \in W_3$, we write $H = {\mathcal H}(t,v,c_1):= T_t \circ h \circ T_0^{-1}$. We have $H=h$ iff $T_t = h \circ
T_0 \circ h^{-1}$.

\subsection{Conditions for H to be a diffeomorphism }
Lemma 6.9 is still valid in our present setting

\begin{lemma}
For $(t,v,c_1) \in W_3$, the following are equivalent
\begin{enumerate}
\item $h(u^t_i) = u^t_i$ for all $0<i<d$;
\item $H$ is an homeomorphism of $\overline I$ satisfying $H(u^b_j) = u^b_j$ for all $0<j<d$.
\end{enumerate}
\end{lemma}

When the equivalent conditions of the lemma are satisfied, $H$ is in fact a piecewise $C^2$ diffeomorphism of $\overline
I$,
with possibly discontinuities of the  derivatives of order $\leq r$ at the $u^b_j$. We will replace Lemma 6.10 by the next
two lemmas.

\begin{lemma}
Let $(t,v,c_1) \in W_3$ such that the equivalent conditions of  the last Lemma  are satisfied. Then $H$ is a
$C^1$-diffeomorphism of $\overline I$ iff one has, for all $\upsilon = (\alpha,L) \in \AA$ with $\alpha \ne \,_b\alpha$
$$\Log Dh(u^t(\upsilon)) = \Log Dh(u^t(\sigma(\upsilon)))\,.$$
\end{lemma}
\begin{lemma}
 Let $(t,v,c_1) \in W_3$. The function $D\Log Dh \circ T_0$ is continuous on $\overline I$ iff one has, for all $\upsilon =
 (\alpha,R) \in \AA$ with $\alpha \ne \alpha_t$
 $$D\Log Dh(u^b(\upsilon)) = D\Log Dh(u^b(\sigma(\upsilon)))\;.$$
 \end{lemma}
\begin{remark}
When $\upsilon = (\alpha,R)$, $\alpha \ne \alpha_t$, one has $u^t(\upsilon)=u^t(\sigma(\upsilon))$, hence
$$\Log Dh(u^t(\upsilon)) = \Log Dh(u^t(\sigma(\upsilon)))$$
is always true.
Similarly, when $\upsilon = (\alpha ,L)$, $\alpha \ne \,_b\alpha$, one has $u^b(\upsilon)=u^b(\sigma(\upsilon))$, hence
$$D\Log Dh(u^b(\upsilon)) = D\Log Dh(u^b(\sigma(\upsilon)))$$
is always true.
\end{remark}

\subsection {Equations for the conjugacy class of $T_0$}

\begin{proposition}
Let $(t,v,c_1) \in W_3$ such that $h = {\mathcal H}(t,v,c_1)$ satisfies
\begin{eqnarray*}
h(u^t_i) &=&u^t_i  \quad \quad{\rm for\; all}\; 0<i<d,\\
\Log Dh( u^t(\upsilon)) &=& \Log Dh(u^t(\sigma(\upsilon))),\quad {\rm for \;all}\; \upsilon = (\alpha, L) \in \AA, \alpha
\ne \, _b\alpha, \\
D\Log Dh( u^b(\upsilon)) &= &D\Log Dh(u^b(\sigma(\upsilon))), \;{\rm for\; all}\; \upsilon = (\alpha, R) \in \AA, \alpha \ne
\alpha_t, \\
\end{eqnarray*}
Then , if $(t,v,c_1)$ is close enough to $(0,0,0)$, we have $T_t \circ h = h \circ T_0$.

Conversely, let $t \in [-t_0,t_0]^{\ell +d^*}$ and $ h \in {\rm Diff}^2(\overline I)$ such that $T_t \circ h = h \circ
T_0$.
Let $c_1= D\Log Dh(0)$. If $t$ is close enough to $0$ and $ h$ is close enough to the identity, then there exists $v \in
\Gamma_s$ such that $(t,v,c_1) $ belongs to $ W_3$, $ h = {\mathcal H}(t,v,c_1)$
and the relations   above are satisfied.
\end{proposition}
\begin{proof}
We first prove the second part of the proposition. Let $t \in [-t_0,t_0]^{\ell +d^*}$ close to $0$, $ h \in {\rm
Diff}^2(\overline I)$ close to
the identity, such that $T_t \circ h = h \circ T_0$. Then we have
\begin{equation}
(D\Log DT_t \circ h)\; Dh = D\Log Dh \circ T_0 - D\Log Dh \;.
\end{equation}
Let ${\mathcal S}_1 (h) = (\psi_1,c_1)$. Then, we have
$$\psi_1 \circ T_0 -\psi_1 = (D\Log DT_t \circ h)\; Dh -R\,,$$
with
\begin{eqnarray*}
DR &=& \frac 12 ((D\Log Dh \circ T_0)^2 -(D\Log Dh)^2)\\
&=& \frac 12 (D\Log DT_t \circ h)^2 (Dh)^2 + (D\Log DT_t \circ h) Dh (D\Log Dh).
\end{eqnarray*}
It follows that
$$D(\psi_1 \circ T_0 -\psi_1) = (ST_t \circ h) (Dh)^2\;.$$
As we have $\partial(\psi_1 \circ T_0 -\psi_1) =0$, the function $(ST_t \circ h) (Dh)^2$ has zero mean value and we have
$$\psi_1 \circ T_0 -\psi_1 = P^*((ST_t \circ h) (Dh)^2) + \chi$$
for some $\chi \in \Gamma$. But this means that we have $\psi_1 = \Psi_1(t,h,v)$ for some $v \in \Gamma_s$. This implies $ h
= {\mathcal H}(t,v,c_1)$.
Moreover $h = H$, hence the first two sets of relations in the proposition are satisfied from Lemmas 7.8 and 7.9. We also
have from (7.1) above that the function
$D\Log Dh \circ T_0$ is continuous on $\overline I$. Then the third set of relations in the proposition follows from Lemma
7.10. This concludes the proof of the second part of the proposition.

\bigskip

We now assume that $(t,v,c_1), h = {\mathcal H}(t,v,c_1)$ satisfy the three sets of relations in the proposition. We will
prove below that relation
(7.1) is satisfied. With $H$ as above, we have then $D\Log DH = D\Log Dh$. But $H$ is a $C^1$ diffeomorphism of $\overline
I$ (piecewise $C^2$)  by Lemmas 7.8 and 7.9. The relation $D\Log DH = D\Log Dh$ implies then that $H$ is a $C^2$
diffeomorphism and $h=H$.

To see that (7.1) is satisfied, we will use the following lemma, where $C^0(\overline I,0)$ denotes the space of continuous
functions on $\overline I$ with mean value $0$.

\begin{lemma}
The map
$$(\varphi,N) \mapsto {\mathcal T}(\varphi,N) = \varphi -\Delta \varphi$$
$$ C^0(\overline I,0) \times C^0(\overline I) \longrightarrow C^0(\overline I,0)$$
with $$ D\Delta \varphi = \frac 12 \varphi^2 + \varphi N - \int_I (\frac 12 \varphi^2 + \varphi N), \quad \int_I \Delta
\varphi =0$$
is of class $C^1$. Its differential at $(0,0)$ is $(\delta \varphi, \delta N) \mapsto \delta \varphi$. Thus, for $N$ close
enough to $0$, the map
$\varphi \mapsto {\mathcal T}(\varphi,N)$ is a $C^1$ diffeomorphism from a neighborhood of $0 \in C^0(\overline I,0)$ to
another neighborhood of $0$.
\end{lemma}

Let $N= D\Log Dh$. Take first $\varphi_0 = (D\Log DT_t \circ h) Dh$. This function does belong to $C^0(\overline I,0)$.

One has ${\mathcal T}(\varphi_0,N) = \varphi_0 - \Delta \varphi_0$ with
$$ D \Delta \varphi_0 = \frac 12 (D\Log DT_t \circ h)^2 (Dh)^2 + (D\Log DT_t \circ h) Dh (D\Log Dh) -c_0\,,$$
hence
$$ D(\varphi_0 -\Delta \varphi_0) = (ST_t \circ h) (Dh)^2 -c_0,$$
where the constant $c_0$ is the mean value of $(ST_t \circ h) (Dh)^2$ (as $D \varphi_0$ has mean value $0$). Therefore
${\mathcal T}(\varphi_0,N) =P^*(\Phi(t,h))-c$, with $c$ equal to the mean value of $P^*(\Phi(t,h))$.

Next take $\varphi_1 = D\Log Dh \circ T_0 - D\Log Dh$. This function has mean value $0$, and is continuous from Lemma 7.10.
Therefore it belongs to
$C^0(\overline I,0)$. One has ${\mathcal T}(\varphi_1,N) = \varphi_1 - \Delta \varphi_1$ with
$$ D \Delta \varphi_1 = \frac 12 [(D\Log Dh \circ T_0)^2 - (D\Log Dh)^2]\,.$$
Let ${\mathcal S}_1 h =(\psi_1,c_1)$. We have therefore
$$\varphi_1 - \Delta \varphi_1= \psi_1 \circ T_0 -\psi_1 + \chi$$
for some $\chi \in \Gamma$. As $\varphi_1 - \Delta \varphi_1$ is continuous with zero mean value, it must also be equal to
$P^*(\Phi(t,h))-c$.

We conclude that ${\mathcal T}(\varphi_0,N) = {\mathcal T}(\varphi_1,N)$, hence $\varphi_0 = \varphi_1$ by the lemma. This
is (7.1), and the proof of the proposition is complete.

\end{proof}

\subsection{End of the proof of Theorem 5.1 for $r=2$}

From the proposition above, we have to determine in a neighborhood of $0 \in [-t_0,t_0]^{\ell +d^*}$ the set of $t$ for
which, for some $(v,c_1)$ close to $(0,0)$, the diffeomorphism  $h = {\mathcal H}(t,v,c_1)$ satisfies

\begin{equation}
h(u^t_i) =u^t_i  \quad \quad{\rm for\; all}\; 0<i<d,
\end{equation}
\begin{equation}
\Log Dh( u^t(\upsilon)) = \Log Dh(u^t(\sigma(\upsilon))),\quad {\rm for \;all}\; \upsilon = (\alpha, L) \in \AA, \alpha \ne
\, _b\alpha,
\end{equation}
\begin{equation}
D\Log Dh( u^b(\upsilon)) = D\Log Dh(u^b(\sigma(\upsilon))), \;{\rm for\; all}\; \upsilon = (\alpha, R) \in \AA, \alpha \ne
\alpha_t\,.
\end{equation}

We will see that there are exactly $(d^* +g+1)$ independent equations for $t,v,c_1$ in the system above. Looking at the
linearized system at $(0,0,0)$ will allow to apply the implicit function theorem and conclude. We  deal separately with the
same two cases which appeared in Section 6.

\begin{itemize}
\item We first assume that $(\alpha_t,R)$ and $(\,_b\alpha,L)$ belong to the same cycle of $\sigma$.\\

There are $(d-1)$ equations in (7.2). In (7.3), for each cycle of $\sigma$ which does not contain $(\,_b\alpha,L)$,
there is one redundant equation. So the number of equations in (7.3) is really $(d-1)-(s-1)=(2g-1)$. Similarly, there
are
$(2g-1)$ equations in (7.4).

Therefore the total number of equations in the system (7.2)-(7.4) is $(d-1) +(2g-1) +(2g-1) = d^* +g+1$ as claimed.

Consider now the linearized system obtained from (7.2)-(7.4) at $(0,0,0)$. Writing as before $\delta \varphi = \Delta
T(t)$, we have, from Lemma 7.7
$$\delta h = P_1(L_0(D^2 \delta \varphi)+L(v), \delta c_1).$$
From the definition of $L_0$, $P^*$ and $P_1$, this is equivalent to
$$D^2 \delta h \circ T_0 -D^2 \delta h = D^2 \delta \varphi + \chi_2,\quad D^2 \delta h (u_0)=\delta c_1 \,,$$
with $\chi_2 = \delta v + L_c(D^2 \delta \varphi) + L_u(D^2 \delta \varphi)$.
The linearized version of (7.3) is
\begin{equation}
 D^2\delta h( u^b(\upsilon)) = D^2 \delta h(u^b(\sigma(\upsilon))),\quad {\rm for \;all}\; \upsilon = (\alpha, R) \in
 \AA, \alpha \ne \alpha_t \,.
\end{equation}
This implies that $\chi_2$ is continuous at each $u^t_i$, hence constant. As it has mean value $0$, we have $\chi_2 =0$.
Integrating  $D^2 \delta h \circ T_0 -D^2 \delta h = D^2 \delta \varphi $ gives
$$D \delta h \circ T_0 -D \delta h = D \delta \varphi + \chi_1\,,$$
for some $\chi_1 \in \Gamma(1)$. The linearized version of (7.2) is
\begin{equation}
 D\delta h( u^t(\upsilon)) = D \delta h(u^t(\sigma(\upsilon))),\quad {\rm for \;all}\; \upsilon = (\alpha, L) \in \AA,
 \alpha \ne \, _b\alpha\,.
\end{equation}
If it holds, $\chi_1$ has to be constant (at each $u^b_j$, the left and right values of $\chi_1$ are the same). As
$\int_0^1 D\delta \varphi(x)\, dx =0$, we must have $\chi_1 =0$. Observe that the equation $D\delta \varphi = D \delta h
\circ T_0 - D \delta h$ determines $\delta c_1$. One more integration then gives
$$\delta \varphi =  \delta h \circ T_0 -  \delta h + \chi_0\,,$$
for some $\chi_0 \in \Gamma(1)$.
The linearized version of (7.1) is
\begin{equation}
\delta h (u^t_i) =0\, .
\end{equation}
If this holds, one obtains as above first that $\chi_0$ is constant; as $\delta h (u_d) =0 = \delta h
(u^t(\alpha_b,R))$, we have $\chi_0 =0$.

Recalling the definition of $\nu$ in Remark 3.3 and $\Pi$ in Theorem 3.13, we conclude that, if (7.5)-(7.7) holds, then
$$\Pi(\delta \varphi) =0, \quad \nu(\delta \varphi) =0.$$
Going backwards, we see that these relations are actually equivalent to (7.5)-(7.7).
In view of the transversality hypotheses (Tr1), (Tr2) of Subsection 5.1, the theorem in this case now follows from the
implicit function theorem.\\

\item We now assume that $(\alpha_t,R)$ and $(\,_b\alpha,L)$ do not belong to the same cycle of $\sigma$.\\

There are still $(d-1)$ equations in (7.1). In (7.2), for each cycle of $\sigma$ which contains neither $(\,_b\alpha,L)$
nor $(\alpha_t,R)$, there is one redundant equation.
This would give $(d-1)-(s-2)=2g$ for the number of equations in (7.2) , and similarly in (7.3), leading to a grand total
of $(d^* +g+3)$ equations. However, we will now see that the equations
\begin{equation}
\Log Dh(u^t(\,_t\alpha ,L))= \Log Dh(u^t(\sigma(\,_t\alpha ,L))),
\end{equation}
\begin{equation}
D\Log Dh(u^b(\alpha_b ,R))= D\Log Dh(u^b(\sigma(\alpha_b ,R))),
\end{equation}

are also redundant.

Indeed, assume that (7.2)-(7.4) holds, with the exception of (7.8)-(7.9). Let $H=T_t \circ h \circ T_0^{-1}$ as above.
We first show that (7.1) (in subsection 7.6) holds. We will use a variant of Lemma 7.13.

Denote by $C^0_*(\overline I)$ the space of functions on $\overline I$ which vanish at $u_0$ and are continuous
on $\overline I$ except possibly at $u^t(\alpha_b,R)$ where they have a right and left limit. Let $C^0_*(\overline I,0)
= \{\varphi \in C^0_*(\overline I),\;\int_I \varphi =0\}$ and let $\pi: C^0_*(\overline I) \rightarrow C^0_*(\overline
I,0)$ be the projection operator
such that, for $\varphi \in C^0_*(\overline I)$, $\varphi - \pi(\varphi)$ is constant on $(u^t(\alpha_b,R),u_d)$ and $0$
on $(u_0,u^t(\alpha_b,R))$.
We define a map
$$(\varphi,N) \mapsto {\mathcal T}(\varphi,N) := \pi(\varphi -\Delta \varphi)$$
from $ C^0_*(\overline I,0) \times C^0(\overline I)$ to $C^0_*(\overline I,0)$ by the formulas
$$ D\Delta \varphi = \frac 12 \varphi^2 + \varphi N - \int_I (\frac 12 \varphi^2 + \varphi N), \quad  \Delta
\varphi(u_0) =0$$

\begin{lemma}
The map
${\mathcal T}$
is of class $C^1$ and satisfies ${\mathcal T}(0,N)=0$ for all $N \in C^0(\overline I)$. Its differential at $(0,0)$ is
$(\delta \varphi, \delta N) \mapsto \delta \varphi$. Thus, for $N$ close enough to $0$, the map
$\varphi \mapsto {\mathcal T}(\varphi,N)$ is a $C^1$ diffeomorphism from a neighborhood of $0$ in $C^0_*(\overline I,0)$
to another neighborhood of $0$ in $C^0_*(\overline I,0)$.
\end{lemma}

Let $N:=D\Log Dh \in C^0(\overline I)$. The function $\varphi_0 := (D\Log DT_t \circ h) Dh$ belongs to $C^0_*(\overline
I,0)$ (it is actually continuous at $u^t(\alpha_b,R)$). A small computation gives ${\mathcal T}(\varphi_0,N)= \pi
(P^*(\Phi(t,h)))$.

Let $\varphi_1 :=  D\Log Dh \circ T_0 -D\Log Dh$. As (7.4) is satisfied with the exception of (7.9), the function
$\varphi_1$ belongs to $C^0_*(\overline I)$. Moreover, it clearly has mean value $0$, hence we have $\varphi_1 \in
C^0_*(\overline I,0)$. Writing ${\mathcal S}_1 (h)= (\psi_1,c_1)$, we obtain after a short computation that
$${\mathcal T}(\varphi_1,N) = \psi_1 \circ T_0 -\psi_1 + \chi$$
for some $\chi \in \Gamma$.
Observe that it follows from (7.4) minus (7.9) that $\varphi_1 (u_0) =0$. Therefore ${\mathcal T}(\varphi_1,N)(u_0) =0$.
As ${\mathcal T}(\varphi_1,N)$
belongs to $ C^0_*(\overline I,0)$, one must have ${\mathcal T}(\varphi_1,N)= \pi (P^*(\Phi(t,h)))$. We then conclude
from the lemma that $\varphi_0 =
\varphi_1$, which is (7.1).

From (7.2) minus (7.8), the function $\Log DH$ is continuous on $\overline I$, except perhaps at $u^b(\,_t\alpha,L)$.
From (7.1), we deduce by integration that $\Log DH -\Log Dh$ is constant on $(u_0, u^b(\,_t\alpha,L))$ and
$(u^b(\,_t\alpha,L),u_d)$. We have also from (7.2)
$$\Log DH(u_d) = \Log Dh(u^t(\alpha_b,R)) = \Log Dh (u^t(\alpha_t,R) = \Log Dh (u_d)\,.$$
As $\int_I Dh = \int_I DH$, we conclude that $Dh = DH$ and finally (as $h(u_0) = u_0 = H(u_0)$) that $h=H$. We have thus
proven that (7.8) and (7.9) are redundant.

We now consider the system linearized from (7.2)-(7.4) minus (7.8)-(7.9). As in the first case, we have, with $\delta
\varphi = \Delta T(t)$

$$\delta h = P_1(L_0(D^2 \delta \varphi)+L(v), \delta c_1),$$
which is equivalent to
$$D^2 \delta h \circ T_0 -D^2 \delta h = D^2 \delta \varphi + \chi_2,\quad D^2 \delta h (u_0)=\delta c_1 \,,$$
with $\chi_2 = \delta v + L_c(D^2 \delta \varphi) + L_u(D^2 \delta \varphi)$.

The linearized version of (7.3) minus (7.9) is
\begin{equation}
 D^2\delta h( u^b(\upsilon)) = D^2 \delta h(u^b(\sigma(\upsilon))),\quad {\rm for \;all}\; \upsilon = (\alpha, R) \in
 \AA, \alpha \ne \alpha_t, \alpha_b \,.
\end{equation}
This implies that $\chi_2$ is continuous on $(u_0,u^t(\alpha_b,R))$ and $(u^t(\alpha_b,R),u_d)$. We have also from
(7.10) that
$D^2\delta h(u_0) = D^2 \delta h (u^b(\,_t\alpha,L))$ hence $\chi_2(u_0) =0$. As $\chi_2$ has mean value $0$, we obtain
$\chi_2 =0$.

 Integrating  $D^2 \delta h \circ T_0 -D^2 \delta h = D^2 \delta \varphi $ gives
$$D \delta h \circ T_0 -D \delta h = D \delta \varphi + \chi_1\,,$$
for some $\chi_1 \in \Gamma(1)$. The linearized version of (7.2) is
\begin{equation}
 D\delta h( u^t(\upsilon)) = D \delta h(u^t(\sigma(\upsilon))),\quad {\rm for \;all}\; \upsilon = (\alpha, L) \in \AA,
 \alpha \ne \, _b\alpha, \, _t\alpha\,.
\end{equation}
If it holds, $\chi_1 \circ T_0^{-1}$ has to be constant on $(u_0, u^b(\,_t\alpha,L))$ and $(u^b(\,_t\alpha,L),u_d)$.
Also, from (7.11), we have
$D \delta h (u_d) = D \delta h( u^t(\alpha_b,R))$, hence $\chi_1 \circ T_0^{-1}(u_d) =0$. As $\int_0^1 D\delta
\varphi(x)\, dx =0$, we must have $\int_I \chi_1 =0$ and $\chi_1 =0$.

Observe that the equation $D\delta \varphi = D \delta h \circ T_0 - D \delta h$ determines $\delta c_1$. One more
integration then gives
$$\delta \varphi =  \delta h \circ T_0 -  \delta h + \chi_0\,,$$
for some $\chi_0 \in \Gamma(1)$.
The linearized version of (7.1) is
\begin{equation}
\delta h (u^t_i) =0\, .
\end{equation}
If this holds, one obtains as above first that $\chi_0$ is constant; as $\delta h (u_d) =0 = \delta h
(u^t(\alpha_b,R))$, we have $\chi_0 =0$.

Recalling the definition of $\nu$ in Remark 3.3 and $\Pi$ in Theorem 3.13, we conclude that, if (7.10)-(7.12) holds,
then
$$\Pi(\delta \varphi) =0, \quad \nu(\delta \varphi) =0.$$
Going backwards, we see that these relations are actually equivalent to (7.10)-(7.12).
In view of the transversality hypotheses (Tr1), (Tr2) of Subsection 5.1, the theorem in this case now follows from the
implicit function theorem.\\

\end{itemize}

\bigskip

The proof of the theorem for $r=2$ is now complete. \hspace{1cm} $\Box$

\section{Simple deformations of linear flows on translation surfaces}

\subsection{Translation surfaces} 

\begin{definition}
Let $M$ be a compact connected orientable surface and $\Sigma = \{A_1, \ldots , A_s\}$ be a non-empty finite subset 
of $M$. A {\it structure of translation surface} on $(M, \Sigma)$ is a maximal atlas $\zeta$ for $M- \Sigma$ of charts by open sets of $\Cset \simeq \Rset^2$ which satisfies the two following properties :
\begin{enumerate}
\item any coordinate change between two charts of the atlas is locally a translation of $\Rset^2$ ;
\item for every $1 \leq i \leq s$, there exists an integer $\kappa_i \geq 1$, a neighborhood $V_i$ of $A_i$, a neighborhood $W_i$ of $0$ in $\Rset^2$ and a ramified covering $\pi : (V_i , A_i) \rightarrow (W_i , 0)$ of degree $\kappa_i$ such that every injective restriction of $\pi$ is a chart of $\zeta$.
 \end{enumerate}
\end{definition}

\medskip

It is equivalent to equip $M$ with a complex structure and a holomorphic $1$-form which does not vanish on $M-\Sigma$ and has at $A_i$ a zero of order $\kappa_i -1$.

\medskip

For   a structure of translation surface $\zeta$ on $(M, \Sigma)$ and $g \in GL(2,\Rset)$, one defines a new structure $g.\zeta$ by postcomposing the charts of $\zeta$ by $g$.

\medskip

\begin{definition}
Let $\zeta$ be a structure of translation surface on $(M, \Sigma)$. The  {\it vertical vectorfield} is the vectorfield on $M-\Sigma$ which reads as $\frac {\partial}{\partial y}$ in the charts of $\zeta$. The associated flow is the {\it vertical flow}. An orbit of the vertical flow which ends (resp. starts) at a point of $\Sigma$ is called an {\it ingoing }
(resp. {\it outgoing} ) {\it vertical separatrix}. A {\it vertical connexion} is an orbit of the vertical flow which both starts and ends at a point of $\Sigma$. 
\end{definition}

\medskip

More generally, a {\it linear flow} on $(M,\Sigma, \zeta)$ is a flow on $M-\Sigma$ which is vertical for $g.\zeta$, for some $g \in GL(2,\Rset)$.

\medskip

\begin{definition}
A $C^r$ {\it simple deformation} of the vertical vectorfield $X_0$ of  $(M,\Sigma, \zeta)$ is a non vanishing $C^r$-vectorfield $X$ on $M-\Sigma$ which co\"incides with $X_0$ in a neighborhood  of $\Sigma$ and is appropriately $C^r$-close to $X_0$ on $M-\Sigma$.
\end{definition}

\medskip

\begin{definition}
Let $\zeta$ be a structure of translation surface on $(M, \Sigma)$. An open bounded horizontal segment  $I$ is in {\it good position} if 
\begin{enumerate}
\item $I$ meets every vertical connexion; 
\item the endpoints of $I$ are distinct and either belongs to $\Sigma$ or is connected to a point of $\Sigma$ by a vertical segment not meeting $I$.
\end{enumerate}
\end{definition}

If there is no vertical connexion, or no horizontal connexion, then such segments always exist. One may even ask that the left endpoint of $I$ is in $\Sigma$ ([Y4, Proposition 5.7, p.16]) In particular, one can always find $g \in GL(2,\Rset)$, preserving the vertical direction, and a segment in good position which is horizontal for $g.\zeta$.

\medskip

When $I$ is in good position, the return map $T_I$ of the vertical flow on $I$ is an i.e.m.  and the translation surface $(M,\Sigma, \zeta)$ can be recovered from  $T_I$ and the appropriate {\it suspension data} via Veech's {\it zippered rectangles construction}. 
 
\subsection{The boundary operator and the conjugacy invariant}

Let $(M,\Sigma, \zeta)$ be a translation surface, $(V^{\tau})$ its vertical flow, $I$ an horizontal segment in good position,
$T_I$ the associated return map. Let $ I = \sqcup \iat = \sqcup\iab $ the  partitions defining the i.e.m. $T_I$. 
We denote as usual by $u^t_1< \ldots < u^t_{d-1}$ the singularities of $T_I$ , by $u^b_1< \ldots < u^b_{d-1}$ those of $T_I^{-1}$.
\medskip

Let $r$ be an integer $\geq 0$. Denote by $C^r_c(M- \Sigma)$ the functions of class $C^r$ with compact support in $M-\Sigma$.
For a function  $\Phi \in C^r_c(M-\Sigma)$, one defines a function $\varphi :=I(\Phi)$ on $\sqcup \iat$ by
$$ \varphi(x) = \int_0^{r(x)} \Phi(V^{\tau}(x)) \, d\tau,$$
where $r(x)$ is the return time of $x$ to $I$.
Observe that $I$ commutes with horizontal partial derivatives: if $r \geq 1$ and  $\Phi \in C^r_c(M-\Sigma)$, then one has $\frac {\partial}{\partial x} \Phi \in C^{r-1}_c(M-\Sigma)$ and $I(\frac {\partial}{\partial x} \Phi)=  D[ I(\Phi)]$.

\begin{proposition} The operator $I$ sends $C^r_c(M-\Sigma)$ continuously into $C^{r}(\sqcup \iat)$. Its image is the subspace of functions $\varphi \in C^{r}(\sqcup \iat)$ satisfying $\partial D^i \varphi = 0$ for all $0 \leq i \leq r$.
\end{proposition}

\begin{proof}
Regarding  the first assertion, the case $r=0$ is clear and the assertion for higher $r$ follows from the commutation with
horizontal  partial derivatives. 

Let $\Phi \in C^0_c(M-\Sigma), \varphi := I(\Phi) \in C^{0}(\sqcup \iat) $.
For $1\leq j \leq d-1$, define 
$$ L_{\Phi}(u^t_j) = \int_0^{\rho^t_j} \Phi(V^{\tau}(u^t_j)) \, d\tau,$$
$$ L_{\Phi}(u^b_j) = \int_{\rho^b_j}^{0} \Phi(V^{\tau}(u^b_j)) \, d\tau,$$
where $ \rho^b_j <0< \rho^t_j$ are the times  such that $V^{\rho^t_j}(u^t_j)$ and $V^{\rho^b_j}(u^b_j)$ belong to $\Sigma$.

\smallskip
We use the notations of subsection 3.1. Let $\upsilon \in \AA$; if neither $u^t(\upsilon)$ nor $u^b(\upsilon)$ is an endpoint of $I$, one has
$$\varphi(\upsilon) = L_{\Phi}(u^t(\upsilon)) + L_{\Phi}(u^b(\upsilon)).$$

For the remaining elements of $\AA$, we have 
$$\varphi(_t\,\alpha,L) + \varphi(_b\,\alpha,L) = L_{\Phi}(u^t(_b\,\alpha,L) + L_{\Phi}(u^b(_t\,\alpha,L)),$$
$$\varphi(\alpha_t,R) + \varphi(\alpha_b,R) = L_{\Phi}(u^t(\alpha_b,R) + L_{\Phi}(u^b(\alpha_t,R)).$$

In view of the definitions of the boundary operator $\partial$ and the permutation $\sigma$ of $\AA$, there is a total cancellation of the terms in the formula for $\partial \varphi$. 
\smallskip

We have proven that $\partial I(\Phi) = 0$ for $\Phi \in C^0_c(M-\Sigma)$. The case of higher $r$ follows from the commutation of $I$ with horizontal partial derivatives. 

\end{proof}

The following proposition may be seen as a non-linear version of the previous proposition. Let $X_*$ is a $C^r$ simple deformation
of the vertical vectorfield $X_0$ of $\zeta$. We denote by $V_*^{\tau}$ the flow of $X_*$. We assume  that $X$ and $X_0$ co\"incide on the vertical separatrices segments connecting $\Sigma$ to the endpoints of $I$: this warrants that the return map $T_*$ of  $V_*^{\tau}$ on $I$ is a generalized i.e.m. of class $C^r$ with the same combinatorics than $T_0:=T_I$. We can therefore consider the conjugacy invariant $J(T_*) \in J^r$ introduced in subsection 4.1.

\begin{proposition}
The invariant $J(T_*)$ is trivial.
\end{proposition}

\begin{proof}
The proof is essentially the same than the proof of the relation $\partial I(\Phi)=0$ in the previous proposition. Let
$u^t_{1,*}< \ldots < u^t_{d-1,*}$ be the singularities of $T_*$, $u^b_{1,*}< \ldots < u^b_{d-1,*}$ those of $T^{-1}_*$.
For each
$1\leq j \leq d-1$, let $I_j^t$ be a small horizontal segment  transverse to the separatrix from $u^t_j$ to $\Sigma$ and very close to $\Sigma$, with coordinate $x_j^t$ centered at the intersection with the separatrix. Define similarly $I_j^b$ with coordinate $x_j^b$. Let $J_X(u^t_{j,*})$ be the $r$-jet at $u^t_{j,*}$ of the transition map from $I$ to $I_j^t$ along $X_*$. Let $J_X(u^b_{j,*})$ be the $r$-jet at $x^b_j= 0$ of the transition map from $I_j^b$ to $I$  along $X_*$.

Let $\upsilon \in \AA$; if neither $u^t(\upsilon)$ nor $u^b(\upsilon)$ is an endpoint of $I$, one has
$$j(T_*,\upsilon) =  J_X(u^b_*(\upsilon)) J_X(u^t_*(\upsilon)) .$$
For the remaining elements of $\AA$, we have 
$$j(T_*,(_t\,\alpha,L)) j(T_*,(_b\,\alpha,L)) = J_X(u^b(_t\,\alpha,L)) J_X(u^t(_b\,\alpha,L),$$
$$j(T_*,(\alpha_t,R)) j(T_*,(\alpha_b,R)) = J_X(u^b(\alpha_t,R)) J_X(u^t(\alpha_b,R).$$
The same cancellation takes place. Therefore the conjugacy invariant is trivial. 
\end{proof}

\subsection{Statement of the result} 

\begin{definition}
Let $(M,\Sigma, \zeta)$ be a translation surface . It is of {\it Roth type} (resp. {\it restricted Roth type}) if there exists {\bf some} open bounded horizontal segment $I$  in  good position such that the return map $T_I$ of the vertical flow on $I$ is an i.e.m. of Roth type (resp. restricted Roth type).
\end{definition}              
              
 Actually, we show in Appendix C  that for a (restricted) Roth type translation surface, $T_I$ will be of (restricted) Roth type for {\bf any} horizontal segment $I$ in good position.
 
 \medskip
 
 Recall that two  vectorfields $X,Y$ are said to be $C^r$-equivalent if there exist a $C^r$-diffeomorphism $H$ sending
the time-oriented orbits of the flow of $X$ on the time-oriented orbits of the flow of $Y$. An equivalent formulation is that $H^*X$ is a positive scalar multiple of $Y$, or that $Y$ is obtained from $H^*X$ by time reparametrization.

\medskip
 
 \vskip .3 truecm\noindent {\bf Corollary of the main theorem.}\hspace{5mm} {\it Let $(M,\Sigma, \zeta)$ be a translation surface of restricted
 Roth type, and let $r$ be an integer $\geq 2$. Amongst the $C^{r+3}$ simple deformations of the vertical 
vectorfield $X_0$, those which are $C^r$-equivalent to $X_0$ by a diffeomorphism $C^r$ close to the
identity form a $C^1$ submanifold of
codimension $d^*=(g-1)(2r+1)+s$. }

\vskip .3 truecm\noindent
 
 Concerning conjugacy between $X_t$ and $X_0$ instead of equivalence, see the final remark after the proof.
                      
\subsection{Proof of the Corollary}

 Let  $(X_t)_{t \in [-t_0,t_0]^{\ell +d^*}}$ be a $C^1$ family of $C^{r+3}$ simple deformations  of the vertical 
vectorfield $X_0$ . We assume that the family is in general position. We will allow at various stages to restrict $t$ to a smaller neighborhood of $0$. We divide the proof into several steps.

\begin{enumerate}
\item Choose an open bounded horizontal segment $I$ in good position such that the return map $T_0$ of $X_0$ on $I$ is an
i.e.m. of restricted Roth type. By slightly shifting vertically $I$ if necessary, we may assume that the endpoints $u_0,\,u_d$ of $I$ are not in $\Sigma$. By definition of good position, there are vertical segments $J_0,\, J_d$ disjoint from $I$ connecting  these endpoints to points of $\Sigma$. 

\smallskip

There exists a $C^1$ family $(k_t)$ of $C^{r+3}$-diffeomorphisms of $M$, supported on a compact set of $M-\Sigma$, such that
$k_0$ is the identity and 
for all $t$ the vectorfield $k_t^*X_t$ co\"incides with $X_0$ on $J_0$ and $J_d$. Replacing $X_t$ by $k_t^*X_t$, we assume from now on that $X_t$ co\"incides with $X_0$ on $J_0$ and $J_d$. 

\medskip

\item The singularities $u^t_1< \ldots u^t_{d-1}$ of $T_0$ are the last intersections with $I$ of the ingoing vertical separatrices of $X_0$, while the singularities $u^b_1< \ldots u^b_{d-1}$ of $T_0^{-1}$ are the first intersections with $I$ of the outgoing vertical separatrices of $X_0$. As $X_t$ co\"incides with $X_0$ in the neighborhood of $\Sigma$, we can also define ingoing and outgoing separatrices for $X_t$. By the implicit function theorem, the ingoing separatrices will have as last intersection with $I$ points $u^t_1(t)< \ldots u^t_{d-1}(t)$ which are $C^1$-functions of $t$. Notice here that the fact that $X_t$ co\"incides with $X_0$ on $J_0$ and $J_d$ is crucial to guarantee that these are {\bf last} intersections.

\smallskip

Having the separatrices under control, we know that the first return map $T_t$ for $X_t$ on $I$ is a generalized i.e.m of class $C^{r+3}$ with the same combinatorics than $T_0$, and that $(T_t)$ is a $C^1$ family of such g.i.e.m.
Moreover, every infinite half-orbit of $X_t$ (in the past or in the future) intersects $I$. This is a consequence of the implicit function theorem, taking into account that the return times to $I$ for $X_0$ are bounded.

\medskip

\item From subsection 8.2 above, for all $t$ close to $0$, the conjugacy invariant of $T_t$ in $J^{r+3}$ is trivial.
On the other hand, this is the only restriction on $T_t$: the map $X \mapsto T_X$ , which associates to a $C^{r+3}$ 
simple deformation $X$ of $X_0$ such that $X = X_0$ on  $J_0 \cup J_d$ the return map to $I$ is a {\bf submersion} onto 
g.i.e.m's with trivial $J^{r+3}$-invariant.

\smallskip

It follows that the family $(T_t)$ will be itself in general position (amongst g.i.e.m's with trivial $J^{r+3}$-invariant).
By our main theorem, there is a $C^1$-submanifold $\mathcal C$ of codimension $d^*$ through $0$ which consists exactly of the parameters $t$ such that $T_t$ is conjugated to $T_0$ by a $C^r$-diffeomorphism  of $I$ which is $C^r$-close to the identity.
\medskip

\item  We  will promote , for $t \in \mathcal C$, the conjugacy $h_t$ between $T_0$ and $T_t$ (i.e $h_t \circ T_0 = T_t \circ h_t$) to a $C^r$-equivalence $H_t$ between $X_0$ and $X_t$. This is best done in two steps.

\smallskip

First, consider a small neighborhood $U$ of $I$ in $M-\Sigma$ and a $C^1$- family $(H_t^{\sharp})_{t \in \mathcal C}$ of $C^r$-diffeomorphisms of $M$, $C^r$-close to the identity,  with the following properties:
\begin{itemize}
\item for each $t \in \mathcal C$, $H_t^{\sharp}$ has support in $U$, and is the identity on $J_0$ and $J_d$;
\item  for each $t \in \mathcal C$, $H_t^{\sharp}$ preserves $I$ and the restriction of $H_t$ to $I$ is equal to $h_t^{-1}$;
\item $H_0^{\sharp}$ is the identity.
\end{itemize}

Then, for each $t \in \mathcal C$, the vectorfield $X_t^{\sharp}:= (H_t^{\sharp})^* X_t$ is a $C^r$ simple deformation of $X_0$ for which the return map to $I$ is equal to $T_0$. We also still have the property that every infinite half-orbit of 
$X_t^{\sharp}$ intersects $I$. Observe also that $X_0^{\sharp} = X_0$.

\medskip
\item
Denote by $(V_t^{\tau})$ the flow of $X_t^{\sharp}$. 
For $x \in I$,  not a singularity of $T_0$, let $r_t(x)$ be the return time to $I$ of $x$ under 
$X_t^{\sharp}$.

\smallskip

The $C^r$-equivalence $H_t$ we are looking for will be
$$ H_t  =   \wt H_t \circ (H_t^{\sharp})^{-1},$$
where $\wt H_t$ is a $C^r$-equivalence  between $X_0$ and $X_t^{\sharp}$ satisfying
$$ \wt H_t( V_0^{\tau} (x)) = V_t^{g_{x,t}(\tau)}(x), \quad {\rm for } \; x \in I , \; \tau \in [0, r_0(x)].$$
Here $g_{x,t}$ is a diffeomorphism from $[0, r_0(x)]$ onto $[0, r_t(x)]$. However, we have to be careful in the choice of
$ g_{x,t}$ when $x$ gets close to the endpoints of $I$ or the singularities of $T_0$ because we want  $\wt H_t$ to preserve $\Sigma$ and be of class $C^r$ on the whole of $M$. We will actually  define not only $g_{x,t}$ but also the right and left limits 
$g_{\upsilon,t}$, for $\upsilon \in \AA$ (cf. subsection 3.1). Observe that, for any $\alpha \in \A$, the restriction of the return time function $r_t$ to $\iat$ extends to a  $C^r$-function on the closure of $\iat$. In particular, the values
$r_t(\upsilon)$ are well defined.

\smallskip
\item

 For $1 \leq j \leq d-1$, let $A_j \in \Sigma$ be the endpoint of the ingoing $X_t^{\sharp}$-separatrix from $u^t_j$ and let $\rho_{j,t}$ the time span of this separatrix. 


We will only consider the case where both $J_0$ and $J_d$ are {\bf outgoing} separatices. The other cases are dealt in the same manner, with minor modifications. Under this assumption, we have , for each $1 \leq j \leq d-1$, 
$$ 0< \rho_{j,t}  < \min (r_t(\upsilon_{j,-}),r_t(\upsilon_{j,+})),$$
where $\upsilon_{j, \pm}$ are the elements of $\AA$ adjacent to $u^t_j$.

\item Let $1 \leq j \leq d-1$ and let $\va>0$ be small enough so that $X_t$ and $X_0$ are equal in a $3\va$-neighborhood of $\Sigma$. The image by $V_0^{ \rho_{j,0}}$ of the segment $(u_j^t - 2\va,u_j^t + 2\va) \subset I$ is a horizontal segment $I_j(2\va)$ through $A_j$. Let 
$$C(2\va) = \{ z \in \Cset , \, 0 < |z| < 2 \va , \,-\frac{3\pi}2<\arg z < \frac{\pi}2$$
and let $z_j = x_j + iy_j: C_j(2\va) \rightarrow C(2\va)$ be the chart of $\zeta$ such that the equation of $I_j(\va)$ is $y_j =0$. The domain $C_j(2\va)$ is a circular open cone of radius $2 \va$, aperture $2 \pi$ at $A_j$.

\smallskip
For $t$ close enough to $0$, there is a  $C^r$-maps  $G_{j,t}$ defined on $(-\va,\va)$ (and depending in a $C^1$ way on $t$) such that, for $s \in (-\va,\va)$, the point $V_t^{\tau}( u^t_j +s)$ belongs to $I_j(2\va)$ at a time $\tau$ close to $\rho_{j,0}$
and its $x_j$-coordinate is $G_{j,t} (s)$. Observe that 
$$  G_{j,0}(s) \equiv s, \quad \quad 
G_{j,t} (0) = 0.$$ 
Write $G_{j,t}(s) = s w_{j,t}(s)$, with $w_{j,t}(s)$ close to $1$. Define then a map $G^*_{j,t}: C_j(\va) \rightarrow C_j(2\va)$ given in the $z_j$-coordinate by 
$$ G^*_{j,t}(z_j) = z_j w_{j,t}(x_j).$$
Observe that $G^*_{j,0}$ is the identity, that $G^*_{j,t}$ preserves the vertical foliation and that its restriction to 
the segment $I_j(\va)$ is $(x_j,0) \mapsto (G_{j,t}(x_j),0)$.

\item We now choose the $g_{x,t}$ in order to satisfy the following properties:
\begin{itemize}
\item for each $x \in I$ which is  not a singularity of $T_0$, $g_{x,t}$ is a diffeomorphism from $[0, r_0(x)]$ onto $[0, r_t(x)]$;
\item for each $\upsilon \in \AA$, $\upsilon \ne (_t \,\alpha,L),(\alpha_t, R)$,  $g_{\upsilon,t}$ is a diffeomorphism from $[0, r_0(\upsilon)]$ onto $[0, r_t(\upsilon)]$;
\item for $\upsilon_0:= (_t \,\alpha,L)$, $g_{\upsilon_0,t}$ is a diffeomorphism from $[-|J_0|, r_0(\upsilon_0)]$ onto $[-|J_0|, r_t(\upsilon_0)]$; for $\upsilon_d:= (\alpha_t,R)$, $g_{\upsilon_d,t}$ is a diffeomorphism from $[-|J_d|, r_0(\upsilon_d)]$ onto $[-|J_d|, r_t(\upsilon_d)]$;
\item for $\tau \in [0,\va)$, any $x$ or $\upsilon$, we have
$$ g_{x,t}(\tau) = g_{\upsilon,t}(\tau) = \tau, $$
$$g_{x,t}(r_0(x)-\tau) = r_t(x) - \tau , \quad \quad g_{\upsilon,t}(r_0(\upsilon)-\tau) = r_t(\upsilon) -\tau.$$
\item for $1 \leq j \leq d-1$ , the diffeomorphisms $g_{\upsilon_{j,-},t}$ and $g_{\upsilon_{j,+},t}$ coincide on 
$[0, \rho_{j,0}]$ and send this interval onto $[0, \rho_{j,t}]$; we denote by $g_{u^t_j,t}$ this restriction.
\item The similar condition for the outgoing separatrices is as follows; let $1 \leq j \leq d-1$, and let $\upsilon'_{j,\pm}$  be the elements of $\AA$ adjacent to $u^b_j$; then we must have
$$ r_t(\upsilon'_{j,-})  - g_{\upsilon'_{j,-},t}(r_0 (\upsilon'_{j,-})-\tau)= r_t(\upsilon'_{j,+})  - g_{\upsilon'_{j,+},t}(r_0 (\upsilon'_{j,+})-\tau)$$
for $0 \leq \tau \leq \rho'_{j,0}$; here $\rho'_{j,0}$ is the time span from $\Sigma$ to $u^b_j$ for $X_0$.
\item Let $R_t$ be the open subset of $I \times \Rset$ formed of the pairs $(x,\tau)$ such that $0<\tau<r_t(x)$ if $x$ is not one of the $u^t_j$, $0<\tau< \rho_{j,t}$ if $x= u^t_j$; then the map $G_t: (x,\tau) \mapsto (x,g_{x,t}(\tau))$ is a $C^r$-diffeomorphism from  $R_0$ onto $R_t$, and $t \mapsto G_t$ is $C^1$. 
\item We also ask for a similar condition along the outgoing separatrices.
\item If $(x,\tau) \in R_0$ satisfies $V_0^{\tau(x)} \in C_j(\va)$ (for some $1 \leq j \leq d-1$), then
$$ V_t^{g_{x,t}(\tau)}(x) = G^*_{j,t}(V_0^{\tau(x)}).$$
\end{itemize}

It is fastidious but not difficult to check that these conditions are compatible and that one can indeed satisfy all conditions.

As mentioned earlier, from the $g_{x,t}$, we obtain an equivalence $\wt H_t$ between $X_0$ and $X_t^{\sharp}$ by the formula

$$\wt H_t( V_0^{\tau} (x)) = V_t^{g_{x,t}(\tau)}(x).$$

The properties required along the ingoing and outgoing separatrices guarantee that $\wt H_t$ is a $C^r$-diffeomorphism of $M - \Sigma$. Then, the properties of the $G^*_{j,t}$ warrant that $\wt H_t$ is also $C^r$ in the neighborhood of $\Sigma$.
Indeed, if $k$ is the ramification index of $\zeta$ at $A_j$ and we write $z_j = Z_j^k$, we will have in the  $Z_j$-coordinate that 
$$ G^*_{j,t} (Z_j) = Z_j (w_{j,t}(\Re Z_j^k))^{\frac 1k}.$$
We conclude that, for $t \in \mathcal C$, $X_t$ and $X_0$ are indeed $C^r$-equivalent.

\item On the other hand, if $X_t$ and $X_0$ are $C^r$-equivalent; the restriction of the $C^r$-equivalence to $I$ is a $C^r$-conjugacy between $T_t$ and $T_0$, hence $t \in \mathcal C$. $\Box$
\end{enumerate} 
\begin{remark}

One could look for a {\bf conjugacy} (respecting time) rather than an equivalence between $X_t$ and $X_0$. To transform an equivalence into a conjugacy, one needs that the return times to $I$ of $X_0$ and $X_t$ differ by the coboundary of an appropriately smooth function on $I$. Therefore, from the results on the cohomological equation (using also a transversality argument), one finds  a submanifold $\mathcal C^*$ of $\mathcal C$ of codimension $g$ such that $X_t$ and $X_0$ are $C^{r-2}$-conjugated for $t \in  \mathcal C^*$. However, it is not clear at all (and probably just wrong!) that $d^* +g$ is the right codimension for the $C^{r-2}$-conjugacy class of $X_0$ amongst $C^{r+3}$ (or $C^{\infty}$) simple deformations of $X_0$. 

\end{remark}

\appendix\section{The cohomological equation with $C^{1+\tau}$ data}

In this appendix, we show that Theorem 3.10 is also valid with $C^{1+\tau}$ data. Let $\tau \in (0,1)$. We denote by
$C_{\partial}^{1+\tau}(\sqcup \iat)$ the space of functions $\varphi \in C_{\partial}^{1}(\sqcup \iat)$ whose restrictions
to each $\iat$ is of class $C^{1+\tau}$.
Let $T$ be a standard i.e.m.\ of Roth type. We choose a subspace $\Gamma_u \subset \Gamma_{\partial}$ complementing
$\Gamma_T$.

\begin{theorem}
 There exist bounded linear operators $L_0: \varphi \mapsto \psi$ from $C_{\partial}^{1+\tau}(\sqcup \iat)$ to
 $C^0(\overline I)$ and $L_1: \varphi \mapsto \chi$ from $C_{\partial}^{1+\tau}(\sqcup \iat)$ to $\Gamma_u$ such that, for
 all $\varphi \in
C_{\partial}^{1+\tau}(\sqcup \iat)$, we have
$$\varphi = \chi + \psi \circ T -\psi  \;.$$
\end{theorem}

\begin{proof}
We use the notations of subsection 3.3. Associated to any initial subpath $\gamma(1) * \cdots *\gamma(n)$ of the "rotation
number" $\underline \gamma$ of $T$, there is an i.e.m.\ $T^{(n)}$ defined on an interval $I^{(n)}$ with the sama left endpoint
$u_0$ than $I$: $T^{(n)}$ is the first return map
of $t$ on $I^{(n)}$ and is deduced from $T$ by the steps of the Rauzy-Veech algorithm represented by $\gamma(1) * \cdots
*\gamma(n)$. For $\ell < n$ we have a "special Birkhoff sum " operator $S(\ell,n)$ defined as follows: if $\varphi$ is a
function on $\sqcup \iatl$, $S(\ell,n) \varphi$ is defined on $\sqcup \iatn$ by
$$ S(\ell,n) \varphi(x) = \sum_{0\leq i <r(x)} \varphi((T^{(\ell)})^i (x))\,,$$
where $r(x)$ is the return time of $x$ in $I^{(n)}$ under $T^{(\ell)}$.
There are three steps in the proof of the theorem:
\begin{itemize}
\item One first obtains, for some $\delta >0$, and any function $\varphi \in C^{\tau} (\sqcup \iat)$ with $\int_I
    \varphi =0$,
$$||S(0,n) \varphi ||_{C^0}\, \leq\, C\, ||B(n)||^{1-\delta}\, ||\varphi||_{C^{\tau}}\,.$$
Here, only conditions (a) and (b) in the definition of Roth type are used.
\item One then obtain by integration (using also condition (c) in the definition of Roth type) that there exists
    $\delta'>0$ such that, for any
$\varphi \in C_{\partial}^{1+\tau}(\sqcup \iat)$, one can find a unique $\chi \in \Gamma_u$ such that
$$||S(0,n) (\varphi-\chi) ||_{C^0}\, \leq\, C\, ||B(n)||^{-\delta'}\, ||\varphi||_{C^{1+\tau}}\,.$$
\item This last estimate easily imply (using condition (a)) that the ordinary Birkhoff sums of $\varphi -\chi$ are
    bounded; it follows then, as explained in Section 3, that $\varphi -\chi= \psi\circ T - \psi $ for some $\psi \in
    C^0(\overline I)$.
\end{itemize}
 The last two steps are done in exactly the same way in the present setting than in the setting of Theorem 3.10. We will
 therefore only indicate how to prove the estimate of the first step.

 Let therefore $\varphi \in C^{\tau} (\sqcup \iat)$ with $\int_I \varphi =0$. The method is as in [MMY1]. We write
 $$\varphi = \varphi_0 + \chi_0\,$$
 with $\varphi_0$ of mean value $0$ on each $\iat$ and $\chi_0 \in \Gamma$ (of mean value $0$ as $\int_I \varphi =0$). For
 $0<\ell \leq n$, we write in the same way
 $$S(\ell -1, \ell) \varphi_{\ell -1} = \varphi_{\ell} + \chi_{\ell}\,$$
  with $\varphi_{\ell}$ of mean value $0$ on each $\iatl$ and $\chi_{\ell} \in \Gamma^{(\ell)}$ (of mean value $0$).

We have then
$$S(0,n)\varphi = \varphi_n + \sum_0^n S(\ell,n) \chi_{\ell}\,.$$
For $0 \leq \ell \leq n$, $\alpha \in \A$, $x,y \in \iatl$, one has
\begin{eqnarray*}
|\varphi_{\ell}(x) -\varphi_{\ell}(y) | &=& |S(0,\ell)\varphi(x) - S(0,\ell)\varphi(y)| \\
                                        &\leq & r(x)|\iatl |^{\tau} ||\varphi||_{C^{\tau}}.
\end{eqnarray*}
 Here $r(x)$ is the sum of the $\alpha$-column of $B(\ell)$. From condition (a), we have (cf. [MMY1, Proposition p.835])
$|\iatl| \leq C ||B(\ell)||^{-\frac 12}$, hence we obtain
$$ |\varphi_{\ell}(x) -\varphi_{\ell}(y) | \leq C ||B(\ell)||^{1-\frac {\tau}{2}} ||\varphi||_{C^{\tau}}.$$
As $\varphi_{\ell}$ vanishes in each $\iatl$, this implies
$$||\varphi_{\ell} ||_{C^0} \leq C ||B(\ell)||^{1-\frac {\tau}{2}} ||\varphi||_{C^{\tau}}.$$
This gives, for $0<\ell \leq n$
\begin{eqnarray*}
||\varphi_{\ell} + \chi_{\ell} ||_{C^0} &\leq & ||Z(\ell)||\; ||\varphi_{\ell -1} ||_{C^0} \\
                                       &\leq & C ||B(\ell)||^{1-\frac {\tau}{3}} ||\varphi||_{C^{\tau}},\\
||\chi_{\ell} ||_{C^0} &\leq & C ||B(\ell)||^{1-\frac {\tau}{3}} ||\varphi||_{C^{\tau}}.
\end{eqnarray*}
Putting these estimates in the expression for $S(0,n)\varphi$ above, we have to bound from above the sum

\begin{equation}
\sum_0^n ||B(\ell)||^{1-\frac {\tau} {3}} ||B_0(\ell,n)||,
\end{equation}

where $B_0(\ell,n)$ is the restriction of $B(\ell,n)$ to the hyperplane $\Gamma^{(\ell)}_0$ (of functions with mean value
$0$ on $I^{(\ell)}$, constant on each $\iatl$).
To estimate the sum in (A.1), we deal separately with the terms with small $\ell$ and large $\ell$.
\begin{itemize}
\item When $||B(\ell)|| < ||B(n)||^{\frac {\theta}{3}}$, we write
$$B_0(\ell,n) = B_0(n) \, B_0(\ell)^{-1}$$
and get from condition (b) of Roth type (as $B(\ell) $ is symplectic)
\begin{eqnarray*}
||B_0(\ell,n)|| &\leq & ||B_0(n)||\, ||B(\ell)^{-1}|| \\
  &\leq & C ||B(n)||^{1-\theta} \,||B(\ell)|| \\
  ||B(\ell)||^{1-\frac {\tau} {3}} ||B_0(\ell,n)|| &\leq & ||B(n)||^{1-\frac {\theta}{3}}.
\end{eqnarray*}
\item When $||B(\ell)|| \geq ||B(n)||^{\frac {\theta}{3}}$, we just bound $||B_0(\ell,n)||$ by $||B(\ell,n)||$.\\

{\bf Claim:} For every $\eta >0$, there exists $C(\eta)$ such that, for all $0\leq \ell \leq n$, one has
$$||B(n)|| \leq ||B(\ell)|| \,||B(\ell,n)|| \leq C(\eta) ||B(n)||^{1+\eta}\,.$$
The claim gives in this case the following bound
$$||B(\ell)||^{1-\frac {\tau} {3}} ||B_0(\ell,n)|| \leq  C ||B(n) ||^{1-\frac{\tau \theta}{10}}\,.$$
\end{itemize}

As $||B(n)||$ grows at least exponentially fast, one obtains that the sum in (A.1) is indeed bounded by
$C\,||B(n)||^{1-\delta}$ for $\delta < \frac {\tau \theta}{10}$.\\

{\it Proof of the claim}

The left-hand inequality is trivial. If $m-\ell \geq 2d-3$, all coefficients of $B(\ell,m)$ are $\geq 1$ ([MMY1], Lemma
p.833). Therefore, for $n \geq m \geq \ell \geq 0$ with $m-\ell \geq 2d-3$, we have $||B(n)|| \geq ||B(\ell)||\, ||B(m,n)||
$. The right-hand inequality in the claim now follows from condition (a) in the definition of Roth type. \hspace {1cm}
$\Box$

The proof of the inequality for special Birkhoff sums of $C^{\tau}$ functions is now complete. A mentioned above, the rest
of the proof of the theorem is the same than for Theorem 3.10.

\end{proof}

\section{The case of circle diffeomorphisms}

\subsection{The $C^r$-case, $r\geq 3$}

Let $F$ be a $C^{r+3}$ orientation preserving diffeomorphism of the circle $\Tset = \Rset / \Zset$ which is $C^{r+3}$-close
to a rotation $R_{\omega}$.
We assume that $\omega$ satisfies a diophantine condition $CD(\gamma, \tau)$ with $\tau <1$:
$$\forall \frac pq ,\quad |\omega - \frac pq | \geq \gamma q^{-2-\tau}.$$
Following Herman [He], we show that one can write
$$F=R_t \circ h \circ R_{\omega} \circ h^{-1},$$
for some unique $t$ close to $0$ and some unique $h \in {\rm Diff}_+^r(\Tset)$ normalized by $\int_{\Tset} (h-{\rm id}) =0$.
Both $t$ and $h$ are $C^1$-functions of $F$.

We denote by ${\rm Diff}_{+,0}^r(\Tset)$ the set of $h \in {\rm Diff}_+^r(\Tset)$ satisfying $\int_{\Tset} (h-{\rm id}) =0$,
by $C^r_0(\Tset)$ the space of $C^{r}$ functions on $\Tset$ with zero mean-value.

\begin{lemma}
The map $(F,h) \mapsto \Phi(F,h):= (SF \circ h) (Dh)^2$ from ${\rm Diff}_+^{r+3}(\Tset) \times {\rm Diff}_{+,0}^r(\Tset)$ to
$C^{r-1}(\Tset)$ is of class $C^1$. Its differential at $(R_{\omega}, {\rm id})$ is the map $(\delta F, \delta h) \mapsto
D^3 \delta F$.
\end{lemma}
\begin{lemma}
The map $h \mapsto Sh - \int_{\Tset} Sh$ from ${\rm Diff}_{+,0}^r(\Tset)$ to $C_0^{r-3}(\Tset)$  is of class $C^{\infty}$.
Its differential at ${\rm id}$ is $\delta h \mapsto D^3 \delta h$. Therefore its restriction to a neighborhood of the
identity in ${\rm Diff}_{+,0}^r(\Tset)$ is a $C^{\infty}$ diffeomorphism onto a neighborhood of $0$ in $C_0^{r-3}(\Tset)$.
\end{lemma}

Let us write ${\mathcal P}$ for the inverse diffeomorphism, $P$ for its differential at $0$ (consisting in taking thrice a
primitive with mean value zero).

As $\omega$ satisfies $CD(\gamma, \tau)$ with $\tau<1$, there exists a bounded operator $L$ from $C^{r-1}(\Tset)$ to
$C_0^{r-3}(\Tset)$  such that, for every $\varphi \in C^{r-1}(\Tset)$
$$\varphi = \int_{\Tset} \varphi + L(\varphi) \circ \RO -L(\varphi).$$
From the two lemmas above, we see that the map
$$(F,h) \mapsto {\mathcal P}(L(\Phi(F,h)))$$
is defined and of class $C^1$ in a neighborhhood of $(\RO, {\rm id})$ in ${\rm Diff}_+^{r+3}(\Tset) \times {\rm
Diff}_{+,0}^r(\Tset)$, with values in
${\rm Diff}_{+,0}^r(\Tset)$. The differential at $(\RO, {\rm id})$
$$(\delta F,\delta h) \mapsto P(L(D^3 \delta F))$$
does not involve $\delta h$. Therefore, if $F$ is close enough to $\RO$, this map will have a unique fixed point
$h={\mathcal H}(F)$ close to the identity. This fixed point satisfies, with $c= \int_{\Tset} \Phi(F,h) $
$$S(F \circ h) = S(h\circ \RO) +c \,.$$
One then concludes from Lemma B.2 that $F \circ h = R_t \circ h \circ \RO$ for some $t$ close to $0$.

\subsection {The $C^2$-case}

We now show how to adapt the argument when $h$ in only of class $C^2$. The Schwarzian derivative of $h$ no longer exists but
its primitive can still be used!

Let $F \in {\rm Diff}_+^{5}(\Tset)$ be close to $\RO$, with $\omega$ still satisfying $CD(\gamma,\tau)$ for some $\gamma>0,
\tau <1$. Lemma B.1 with $r=2$ is still valid. For $h \in {\rm Diff}_{+,0}^2(\Tset)$ we define $N_1 h \in C^0_0(\Tset)$ by

$$N_1 h(x) = D\Log Dh(x) - \frac 12 \int^x ((D\Log Dh)^2(y) -c_1) dy$$
where $c_1 = \int_{\Tset} (D\Log Dh)^2(y) \,dy$ and the primitive is taken in order to have $\int_{\Tset} N_1 h(x)\,dx=0$.

\begin{lemma}
The map
$h \mapsto N_1h $ from ${\rm Diff}_{+,0}^2(\Tset)$ to $C^0_0(\Tset)$ is of class $C^{\infty}$. Its differential at ${\rm
id}$ is $\delta h \mapsto D^2 \delta h$. Therefore its restriction to a neighborhood of the identity in ${\rm
Diff}_{+,0}^2(\Tset)$ is a $C^{\infty}$ diffeomorphism onto a neighborhood of $0$ in $C_0^{0}(\Tset)$.
\end{lemma}
Let us write ${\mathcal P}_1$ for the inverse diffeomorphism, $P_1$ for its differential at $0$ (consisting in taking twice
a primitive with mean value zero).

Let us also write $P^*$ for the operator from $C^1(\Tset)$ to $C^2_0(\Tset)$
$$\varphi \mapsto \int^x (\varphi(y) - \int_{\Tset} \varphi)\,dy$$
the primitive being taken in order to have mean value $0$.

Consider now the map
$$(F,h) \mapsto {\mathcal P}_1 (L(P^*(\Phi(F,h))))\, . $$
It is defined and of class $C^1$ in a neighborhood of $(\RO, {\rm id})$ in  ${\rm Diff}_+^{5}(\Tset) \times {\rm
Diff}_{+,0}^2(\Tset)$, with values in ${\rm Diff}_{+,0}^2(\Tset)$, sending
$(\RO, {\rm id})$ to ${\rm id}$. The differential at $(\RO, {\rm id})$
$$(\delta F, \delta h) \mapsto P_1(L(D^2 \delta F))\,.$$
does not involve $\delta h$. Therefore, if $F$ is close enough to $\RO$, this map will have a unique fixed point
$h={\mathcal H}(F)$ close to the identity. This fixed point satisfies
\begin{equation}
 P^*( \Phi(F,h)) = N_1 h \circ \RO - N_1 h \;.
 \end{equation}
We will see below that this imply
\begin{equation}
(D\Log DF \circ h) (Dh) = D\Log Dh \circ \RO - D\Log Dh \;.
\end{equation}
From (B.2), we get $\Log D(F\circ h) = \Log D(h\circ \RO) + c_0$ by integration. As the integral over $\Tset$ of both
$D(F\circ h)$ and $D(h\circ \RO)$ is equal to $1$, the constant $c_0$ must be equal to $0$. We conclude that $ F\circ h =
R_t \circ h\circ \RO$ for some $t$ close to $0$.

To see that (B.1) indeed implies (B.2) we introduce the map
$$(\psi,h) \mapsto \psi - \Delta \psi$$
from $C^0_0(\Tset) \times  {\rm Diff}_{+,0}^2(\Tset)$ to $C^0_0(\Tset)$ defined by
$$D\Delta \psi = \frac 12 \psi^2 + \psi D\Log Dh -c(\psi,h)$$,
$$c(\psi,h) = \int_{\Tset}(\frac 12 \psi^2 + \psi D\Log Dh),\quad \int_{\Tset} \Delta \psi =0\;.$$
This map is of class $C^1$. The differential w.r.t.\ $\psi$ at $\psi =0, h= {\rm id}$ is the identity; therefore, as long as
$h$ is fixed close to the identity, it is a $C^1$ diffeomorphism from a neighborhood of $0 \in C^0_0(\Tset)$ to another
neighborhood of $0 \in C^0_0(\Tset)$.

Let $\psi_0 = (D\Log DF \circ h) Dh$. We have
\begin{eqnarray*}
D \psi_0 & = &(D^2 \Log DF \circ h) (Dh)^2 +  (D\Log DF \circ h) D^2h \,\\
D \Delta \psi_0 &=& \frac 12 (D\Log DF \circ h)^2 (Dh)^2 + (D\Log DF \circ h) (Dh) D\Log Dh - c(\psi_0,h) \,\\
D (\psi_0 - \Delta \psi_0)  &=& (SF \circ h) (Dh)^2 + c(\psi_0,h),
\end{eqnarray*}
and therefore $\psi_0 - \Delta \psi_0 = P^*(\Phi (F,h))$.

 On the other hand, let $\psi_1 = D\Log Dh \circ \RO - D\Log Dh$. We have
\begin{eqnarray*}
D \Delta \psi_1 &=& \frac 12 [(D\Log Dh \circ \RO)^2 - (D\Log Dh)^2]
\end{eqnarray*}
hence $\psi_1 - \Delta \psi_1 = N_1 h \circ \RO - N_1 h$.

Equation (B.1) means that $\psi_0 - \Delta \psi_0 = \psi_1 - \Delta \psi_1$. We conclude that $\psi_0 = \psi_1$, i.e
equation (B.2) holds.

\section{ Roth-type translation surfaces}

Let $(M,\Sigma, \zeta)$ be a translation surface with no vertical connexion, $I$ an open bounded horizontal segment  in  good position, $T = T_I$ the i.e.m. on $I$ which is the return map of the vertical flow. 

\smallskip

Let $\A$ the alphabet used to describe the combinatorics of $T$, $\pi$ the combinatorial data of $T$, $\cD$ the Rauzy diagram having $\pi$ as a vertex. Let $\gamma(T)$ be the rotation number of $T$ (cf. subsection 2.4): this is an infinite path in $\cD$ starting from $\pi$ . As in subsection 3.3, write $ \gamma(T)$ as an
infinite concatenation
$$ \gamma(T) = \gamma(1) * \cdots * \gamma(n)* \cdots $$
of finite complete paths of minimal length, and define , for $n > 0$
$$Z(n) := B_{\gamma(n)},\quad B(n):= B_{\gamma(1)*\cdots *\gamma(n)} = Z(n)\cdots Z(1).$$

For $n \geq 0$, let $T^{(n)}$ be the i.e.m. obtained from $T$ by the Rauzy-Veech steps corresponding to $\gamma(1)*\cdots *\gamma(n)$ ; $T^{(n)}$ is the return map of $T$ (or of the vertical flow) on some interval $I^{(n)}\subset I$ having the same left endpoint than $I = I^{(0)}$.

\medskip

We first deal with condition (a) in the definition of a Roth-type i.e.m. (cf. subsection 3.3).

\medskip

\begin{proposition}
The following conditions are equivalent:
\begin{enumerate}
\item Condition (a) of subsection 3.3 is satisfied by $T$: for all $\tau>0$, $||Z(n+1)|| = \mathcal {O} (||B(n)||^{\tau})$. 
\item For all $\tau >0$, we have   $ \max_{\A} |\iatn | =  \mathcal {O} (\min _{\A} |\iatn |^{1-\tau})$.
\item For all $\tau >0$, there exists $C=C(\tau)>0$ such that, for all $1 \leq i,j \leq d-1$, all $x\in I$ and all $N>0$, we have
$$  \min_{0 \leq \ell <N} |T^{\ell}(u^b_i)-u^t_j| \geq C^{-1} N^{-1-\tau}$$
and 

$$  \min_{0 \leq \ell <N} |T^{\ell}(u^b_i)-x| \leq C N^{-1+\tau},\;  \min_{0 \leq \ell <N} |T^{-\ell}(u^t_i)-x| \leq C N^{-1+\tau}.$$
\item For all $\tau >0$, there exists $C=C(\tau)>0$ such that, for any vertical separatrix segment $S$ (ingoing or outgoing)
with an endpoint in $\Sigma$ of length $|S| \geq 1$, and all $P \in M$, there is an horizontal segment of length $\leq C |S|^{-1+\tau}$ from $P$ to $S$, but there is no horizontal segment of length $\leq C |S|^{-1-\tau}$ from a point of $\Sigma$ to $S$. 

\end{enumerate}
\end{proposition}
         
\begin{proof} We will show successively that (1) is equivalent to (2), that (3) is equivalent to (4), that (1)-(2) implies
(3) and that (3) implies (2).
\begin{itemize}
\item $(1) \Leftrightarrow (2)\;$. Recall from the proposition in [MMY1, p.835] that one has always
$ \max_{\A} |\iatn | \geq  ||B(n)||^{-1}|I| \geq \min _{\A} |\iatn |$ and that (1) is equivalent to 
$$ \max_{\A} |\iatn | = \mathcal {O} ( ||B(n)||^{\tau} \min _{\A} |\iatn |), \quad \forall \tau >0.$$

The equivalence of this last relation with (2) is clear.
\item $(3) \Leftrightarrow (4)\;$. Represent $(M,\Sigma, \zeta)$ as a collection of rectangles whose top sides are the $\iab$ and the bottom sides are the $\iat$. The $u^t_i, \,1 \leq i \leq d-1$, are the last intersection points of $I$ with the $d-1$ ingoing separatrices, while  the $u^b_i, \,1 \leq i \leq d-1$, are the first intersection points of $I$ with the $d-1$ outgoing separatrices. As the return times to $I$ (the height of the rectangles) are bounded from above and bounded away from $0$, the length of a (long enough) vertical segment and the cardinality of its intersection with $I$ are comparable. This makes clear the equivalence of (3) and (4). 

\item $(1)+(2) \Rightarrow (3)$.  We start with a result of independent interest. Recall that the return time $r_{\alpha}(n)$ of $\iatn$ in $I^{(n)}$ is given by $r_{\alpha}(n)= \sum_{\beta} B_{\alpha,\beta}(n)$.
\begin{lemma}
Assume that  property (1) holds. Then, for all $\tau>0$, there exists $C=C(\tau)>0$ such that the entrance times $r_i^b(n)$ of $u^b_i$ under $T$ in $I^{(n)}$ and the  entrance times $r_i^t(n)$ of $u^t_i$ under $T^{-1}$ in $I^{(n)}$ satisfy, for all $1 \leq i \leq d-1$:
$$ r^t_i(n) \geq C^{-1} ||B(n)||^{1-\tau},\quad r^b_i(n) \geq C^{-1} ||B(n)||^{1-\tau}.$$
\end{lemma}
{\it Proof of lemma}. Recall ([MMY1,p.833] and [Y4, Proposition 7.12, p.30]) that the product of 
$2d-3$ consecutive matrices $Z(n)$ have only positive coefficients. 
It follows then from the formula for the $r_{\alpha}(n) $ and 
property (1) that, for all $\tau >0$,
 
$$ (\min_{\A} r_{\alpha}(n))^{-1} = \mathcal O ( ||B(n)||^{-1+\tau}).$$
Let $1 \leq i \leq d-1$, and let $\alpha^* \in \A$ be the letter such that $u^b_i$ is the left endpoint of $I_{\alpha^*}^b$. Observe that $\alpha^* \ne \, _b  \alpha$. We have the following dichotomy:
\smallskip

\hspace{1cm} - Either all arrows of $\gamma(n+1)$ with loser $\alpha^*$ are of bottom type. Then we have  $r^b_i(n+1)=r^b_i(n)$. 

\smallskip

\hspace{1cm} -Or $\gamma(n+1)$ contains one arrow of top type with loser $\alpha^*$. Then we have 
$$  r^b_i(n+1)\geq r^b_i(n) + \min_{\A} r_{\alpha}(n).$$

\smallskip
But the first case cannot happen more than $d+1 $ consecutive times: each time $\, _b  \alpha$ is a winner (necessarily of an arrow of top type), $\pi_b(\alpha^*)$ goes up by $1$; once  $\pi_b(\alpha^*)=d$, the next arrow with winner $\alpha^*$ is of bottom type; and any sequence of arrows of bottom type with winner $\alpha^*$ is followed by an arrow of top type with loser $\alpha^*$. 

\smallskip
We get in this way the estimate for $r^b_i(n)$. The proof for  $r^t_i(n)$ is similar. $\Box$

\bigskip

We now assume that (1)-(2) hold . We prove the first inequality in (3). Let $N>0$. Let $n $ be the smallest integer such that $N< r_i^b(n)$. From the lemma, we have $ N \geq r_{i}^b(n-1) \geq C^{-1} ||B(n-1)||^{1-\tau}$, which gives also using (1) that $N \geq C_1^{-1} ||B(n)||^{1-2\tau}$. On the other hand we have with this choice of $n$ that
\begin{eqnarray*}
   \min_{0 \leq \ell <N} |T^{\ell}(u^b_i)-u^t_j| & \geq & \min _{\A} |\iatn | \\
   & \geq & C'^{-1} ||B(n)||^{-1- \tau} \\
   & \geq & {C'}_1^{-1} N^{-\frac {1+\tau}{1-2\tau}}.  
 \end{eqnarray*}
 
 As $\tau>0$ is arbitrary this proves indeed the first inequality of (3).
 \medskip

 We now prove the second part of property (3), regarding the forward orbit of $u^b_i$ (the proof for the backward orbit of $u^t_i$ is similar). 
Let $\alpha^* \in \A, \alpha^* \ne \, _b \alpha$ be the letter such that $u^b_i$ is the left endpoint of $I_{\alpha^*}^b$. Let $n$ be the largest integer such that $N \geq 2B(n)$; we can assume that $n > 3d+4$ and we have from property (1)
$$||B(n)||^{-1} =  \mathcal O (N^{-1+\tau}).$$
 By an argument given in the proof of the lemma, there exists in the path $\gamma(n-d) * \ldots *\gamma(n)$ an arrow of top type with loser $\alpha^*$. This corresponds to a forward iterate $T^m(u^b_i)$ with $0 \leq m \leq ||B(n)||$  which belongs to $I^{(n-d-1)}$ but is {\bf not} one of the endpoints of the $I_{\alpha}^{t,(n-d-1)}, \alpha \in \A$. Let $\beta^* \in \A$ such that  $T^m(u^b_i) \in I_{\beta^*}^{t,(n-d-1)}.$ 
\smallskip
Consider the orbit segment $$T^{\ell}(u^b_i), \, m \leq \ell < m + r_{\beta^*}(n-d-1).$$  Observe that $m + r_{\beta^*}(n-d-1) \leq N$.

One has a partition mod.0 of $I$ by the intervals
$$ T^k( I_{\alpha}^{t,(n-3d-4)}), \quad \alpha \in \A,\; 0 \leq k < r_{\alpha}(n-3d-4) .$$
By ([MMY1,p.833] and [Y4, Proposition 7.12, p.30]), every interval $T^k( I_{\alpha}^{t,(n-3d-4)})$ contains at least an interval $T^{k'}( I_{\beta^*}^{t,(n-d-1)})$ with $0 \leq k' < r_{\beta*}(n-d-1)$, and this last interval contains $T^{m+k'}(u^b_i)$. Choosing $k, \alpha $ such that $x$ belongs to the closure of $ T^k( I_{\alpha}^{t,(n-3d-4)})$, we have 
$$|x - T^{m+k'}(u^b_i)| \leq |I_{\alpha}^{t,(n-d-1)}|.$$

But we have, for all $\tau>0$, from property (2) 
$$ |I_{\alpha}^{t,(n-d-1)}| = \mathcal O (||B(n-d-1)||^{-1+\tau} ).$$
Using once again property (1) and the definition of $N$, we have 
 $$ |I_{\alpha}^{t,(n-d-1)}| = \mathcal O (N^{-1 + \tau})$$
 for all $\tau >0$, which gives the required inequality.

\item $(3) \Rightarrow (2)$. Assume that property (3) is satisfied. Let $n$ be an integer and let $\alpha \in \A$. First assume that  $\alpha \ne \,_b  \alpha, \alpha_b$ . Then  the length of $\iabn$ is given for some $1 \leq i < j < d$ by 
$$ |\iabn | = | T^{ r^b_i(n)} (u^b_i) - T^{ r^b_j(n)} (u^b_j)|.$$

Assume for instance that $r^b_j(n) \leq r^b_i(n)$ and write $r := r^b_i(n)-r^b_j(n)$. After $r^b_j(n)$ backward iterations, we get
$|\iabn | = |u^b_j - T^{r} (u^b_i)|$. This already gives a bound from below for $\iabn$ when $r\leq 1$; 
otherwise, iterating backwards once (if $\alpha \ne \, _t  \alpha $) or twice (if $\alpha =\,  _t  \alpha $), 
we get $|\iatn | = |u^t_{j'} - T^{r-a} (u^b_j)|$ for some $j'$ and some $a \in \{1,2\}$. 
As the entry times $r^b_j(n) , r^b_i(n)$ are bounded above by the return times $r_{\alpha}(n), \alpha \in \A$, 
which are themselves bounded by $||B(n)||$, we get from the first inequality of (3) that
$$ |\iabn |^{-1}  = \mathcal O (||B(n) ||^{1+\tau}) ,\quad \forall \tau >0.$$
The cases $ \alpha =\, _b \alpha$ and $\alpha = \alpha_b$ involve the endpoints of $I^{(n)}$ and require a slightly different argument that we omit, but lead to the same estimate.
\medskip

We now turn to a bound from above for $|I^{(n)}|$. Let $\alpha \in \A$ be the letter such that $r_{\alpha}(n)$ is the largest return time in $I^{(n)}$. We have that $r_{\alpha}(n) = ||B(n)||$ (choosing as norm the greatest column sum).
Assume first that $\alpha \ne \,_b  \alpha, \,_t\alpha$. There exists $1 \leq i,j \leq d-1$ such that
$u^t_i$ is the left endpoint of $T^{r^t_i(n)} (\iatn)$, $u^b_j$ is the left endpoint of $T^{-r^b_j(n)} (\iabn)$, and 
$ r_{\alpha}(n) = r^t_i(n) + r^b_j(n) +1$. Assume for instance that $r^b_j(n) \geq r^t_i(n)$ hence $r^b_j(n)\geq \frac 13 ||B(n)||$. For $0 \leq m < r^b_j(n)$, we have $T^m(u^b_j) \notin I^{(n)}$ by definition of the entrance time. 
Choosing $N = r^b_j(n)$ and for $x$ the middle point in $I^{(n)}$ in the second part of property (3) gives 
$$ |I^{(n)}| = \mathcal O (||B(n) ||^{-1+\tau}).$$
The cases $\alpha = \,_b  \alpha$, $\alpha= \,_t\alpha$ involve the left endpoint $u_0$ of $I$ and require a minor modification of the argument, but lead to the same estimate.

\medskip
These two bounds on the $|\iabn|$ clearly imply property (2).
\end{itemize}

\end{proof}              
              
 \bigskip

We now can prove what was announced in subsection 8.3

\begin{corollary}
Assume that $(M,\Sigma, \zeta)$ is a translation surface of (restricted) Roth type. Then $T_I$ is an i.e.m. of (restricted)
Roth type.
\end{corollary}            
              
\begin{proof}
We have to check that $T_I$ satisfies condition (a) , (b), (c) of subsection 3.3 , and also (d) in the restricted case.
By assumption, there exists  an open bounded horizontal segment $I^{\flat}$ in  good position such that the return map $T^{\flat}$ of the vertical flow to $I^{\flat}$ is an i.e.m. of (restricted) Roth type.
Therefore, property (1) in the proposition is satisfied by $T^{\flat}$. Then, property (4) is satisfied by $(M,\Sigma, \zeta)$. Applying a second time the proposition, we conclude that property (1) is satisfied by $T_I$. This is condition (a) in subsection 3.3 . 
\smallskip
For conditions (b), (c) (and (d) in the restricted case), one has only to observe that, {\bf once (a) is satisfied}, they can be formulated directly in terms of the {\it continuous time} extended Kontsevich-Zorich cocycle over the Teichm\"uller flow in moduli space (without reference to the horizontal segment $I$). The main point is that the continuous times $t_n$ corresponding to the integers $n$ in $B(n)$ satisfy
$$ t_{n+1} = \mathcal O (t_n^{1+ \tau})$$
for all $\tau >0$, so they are "dense enough" to imply the same conditions for all times $t$. We leave the details to the reader.  
\end{proof}

\end{document}